\documentclass[11pt]{article}

\usepackage{geometry}
\usepackage{amsmath}
\usepackage{amssymb, amsfonts, amsthm}
\usepackage{tikz-cd}
\usepackage{enumerate}
\usepackage{mathrsfs}
\usepackage{cite}
\usepackage{graphicx}
\usepackage{xcolor}
\usepackage{hyperref}
\usepackage{subcaption}

\newcommand{\mf}{\mathfrak}
\newcommand{\mc}{\mathcal}
\newcommand{\mb}{\mathbf}
\newcommand{\C}{\mathbb{C}}
\newcommand{\Q}{\mathbb{Q}}

\newcommand{\im}{\operatorname{im}}
\newcommand{\re}{\operatorname{re}}
\newcommand{\R}{\mathbb{R}}
\newcommand{\Z}{\mathbb{Z}}
\newcommand{\N}{\mathbb{N}}
\newcommand{\PP}{\mathbb{P}}
\newcommand{\Hom}{\operatorname{Hom}}

\newcommand{\Ext}{\operatorname{Ext}}
\DeclareMathOperator{\ext}{ext}

\newcommand{\Hilb}{\operatorname{Hilb}}
\newcommand{\on}{\operatorname}
\newcommand{\Id}{\operatorname{Id}}

\newcommand{\Spec}{\operatorname{Spec}}
\newcommand{\GIT}{/\!\!/}
\newcommand{\Heis}{\mc{H}eis}

\renewcommand{\hat}{\widehat}
\renewcommand{\bar}{\overline}
\renewcommand{\tilde}{\widetilde}

\newcommand{\GL}{\operatorname{GL}}

\newcommand{\ab}[2]{\langle #1, #2\rangle }
\newcommand{\ttmx}[4]{\begin{pmatrix} #1 & #2 \\ #3 & #4\end{pmatrix}}

\newcommand{\ses}[3]{0 \to #1 \to #2 \to #3 \to 0}

\newcommand{\RN}[1]{%
  \textup{\uppercase\expandafter{\romannumeral#1}}%
}
\newcommand{\quiv}[3]{\mf{M}_{#1}({#2}, {#3})}
\newcommand{\quivreg}[3]{\mf{M}^{reg}_{#1}({#2}, {#3})}
\newcommand{\quivb}[3]{\mf{M}_{#1}(\mathbf{#2}, \mathbf{#3})}

\newcommand{\quivvar}{\mathfrak{M}_{\theta, \zeta}(\mathbf{v}, \mathbf{w})}
\DeclareMathOperator{\Rep}{Rep}
\DeclareMathOperator{\Amp}{Amp}
\DeclareMathOperator{\NS}{NS}
\DeclareMathOperator{\Stab}{Stab}

\DeclareMathOperator{\Pos}{Pos}
\DeclareMathOperator{\Mov}{Mov}
\DeclareMathOperator{\NE}{NE}
\DeclareMathOperator{\Nef}{Nef}
\DeclareMathOperator{\res}{res}
\DeclareMathOperator{\id}{id}
\DeclareMathOperator{\gr}{gr}

\newtheorem{thm}{Theorem}[section]
\newtheorem{cor}[thm]{Corollary}
\newtheorem{prop}[thm]{Proposition}
\newtheorem{lem}[thm]{Lemma}

\theoremstyle{definition}
\newtheorem{defn}[thm]{Definition}

\newtheorem{const}[thm]{Construction}
\newtheorem{eg}[thm]{Example}

\newtheorem{notn}[thm]{Notation}

\newtheorem{rmk}[thm]{Remark}

\makeatletter
\let\c@equation\c@thm
\makeatother
\numberwithin{equation}{section}

\title{\vspace{-2cm} Affinizations of Lorentzian Kac-Moody Algebras and Hilbert Schemes of Points on K3 Surfaces}
\author{Samuel DeHority}
\begin{document}
\maketitle
\begin{abstract}
For a class of K3 surfaces, the action of a Lie algebra which is a certain affinization of a Kac-Moody algebra is given on the cohomology of the moduli spaces of rank 1 torsion free sheaves on the surface. This action is generated by correspondences between moduli spaces of Bridgeland stable objects on the surface, and is equivalent to an action defined using Fourier coefficients of vertex operators. Two other results are included: a more general result giving geometric finite dimensional Lie algebra actions on moduli spaces of Bridgeland stable objects on K3 surfaces subject to natural conditions and a geometric modular interpretation of some quiver varieties for affine ADE quivers. 
\end{abstract}

\tableofcontents

\section{Introduction}
This paper describes the cohomology of the moduli spaces of all rank 1 torsion-free sheaves on a certain class of K3 surfaces as a module over a specific Lie algebra. 
The characterization of cohomologies of moduli spaces of sheaves as a representation of a Lie algebra is so pervasive that it would be impossible to list every incarnation of this result, but a central role is played by the Hilbert schemes of points on a smooth surface $S$ so that if we combine all of their cohomologies into a single vector space
\[\bigoplus_{n = 0}^\infty H^*(S^{[n]})\]
then this space is a highest weight representation of the Heisenberg algebra $\Heis_{H^*(S)}$ modelled on the cohomology of the surface itself \cite{nakajima1999lectures}. As stated, this result simply says that the given vector space has countable dimension, and even if we asked for the weight spaces to have specific geometric meaning it would only require that these spaces had a specific dimension. Thus the result cited actually says more since it describes a specific set of correspondences between different moduli spaces which induce the action of the generators of $\Heis_{H^*(S)}$.

When the surface $S$ is an ADE surface formed as the minimal resolution of $\C^2/\Gamma$ for a finite group $\Gamma$, which will be the setting of the local result whose global analogue is proven in this paper, the given action of $\Heis_{H^*(S)}$ is related to the action of a larger algebra, namely the affine lie algebra $\hat{\mf g}$ of corresponding type ADE. This algebra acts on the module
\[V_S := \bigoplus_{\alpha \in H^2(S, \Z)} \bigoplus_{n \in \Z} H^*(M(1, \alpha, n))\]
which is the direct sum of the cohomologies of the all spaces of rank 1 torsion free sheaves on $S$. 

There are two ways to produce such an action. The first use the fact that by taking a tensor product with the line bundle $\mc{L}_{\alpha}$ such that $c_1(\mc{L}_{\alpha}) = \alpha \in H^2(S, \Z)$ there is an isomorphism 
\[ \bigoplus_{n = 0}^\infty H^*(S^{[n]}) \simeq \bigoplus_{n \in \Z} H^*(M(1, \alpha, n)) \]
for each $\alpha$, and so $\Heis_{H^*(S)}$ acts on $V_S$. But if we take $H^2(S) \subset H^*(S)$ we get a corresponding inclusion $\Heis_{H^2(S)} \hookrightarrow \Heis_{H^*(S)}$, so $V(S)$ is also a module for $\Heis_{H^2(S)}$, which is of level 1. Then the Frenkel-Kac construction \cite{frenkel1980basic} uses vertex algebra techniques to upgrade the action of $\Heis_{H^2(S)}$ to a level 1 action of $\hat{\mf g}$ on $V_S$ which is compatible with the isomorphism $H^2(S) \simeq \mf{h} \subset \mf{g} \subset \hat{\mf{g}}$ and the corresponding inclusion $\Heis_{H^2(S)} \subset \hat{\mf{g}}$. There are various geometric interpretations of this construction \cite{nakajima1999lectures, carlsson_exts_2012}.

The second way to produce an action of $\hat{\mf g}$ on $V_S$ relies on the general fact \cite{nakajima1994instantons,nakajima1998quiver} that for a quiver $Q$ with associated Kac-Moody algebra $\mf{g}_{KM}$ there is an action of $\mf{g}_{KM}$ on the cohomology of certain Nakajima quiver varieties 
\[ \bigoplus_{\mb v}H^*(\mf{M}(\mb v, \mb w)).\]
When $Q$ is an affine ADE quiver, there is a vector $\mb w = \mb{w_0}$ such that there is a correspondence between data $(\alpha, n) \in H^2(S, \Z)\times \Z$ and dimension vectors $\mb v$ such that
$M(1, \alpha, n)\simeq \mf{M}(\mb v, \mb {w_0})$
for some choice of stability data (see Section \ref{sec:quivers} for details), and in this case the Kac-Moody algebra $\mf{g}_{KM}$ agrees with $\hat{\mf g}$, the relevant affine Lie algebra. Thus combining all of these isomorphisms for all vectors $\mb{v}$ we get an action of $\hat{\mf{g}}$ on $V_S$.

The second way of producing the action of $\hat{\mf {g}}$ makes it apparent that the action of Chevalley generators $e_i$ and $f_i$ are given by convolution with specific (holomorphic) Lagrangian correspondences between quiver varieties for different 
dimension vectors $\mb v$ which are intimately related to the birational geometry of the corresponding quiver varieties. More precisely, after potentially conjugating by the birational transformation induced by variation of GIT stability condition,  they are irreducible components of the variety
\[Z(\mb {v}_1, \mb{v}_2) := \mf{M}(\mb v_1, \mb w)\times_{M_0}\mf{M}(\mb v_2, \mb w)  \]
where $M_0$ is an affine variety and the maps from $\mf{M}(\mb v_i, \mb w)$ are symplectic resolutions of their images. Thus $Z(\mb {v}_1, \mb{v}_2)$ are analogues of the Steinberg variety which encodes the Springer representation of the Weyl group on Springer fibers, see \cite{chriss_representation_2010} for an excellent account of this story. 

The main result of this paper is that both of these methods can be extended to produce equivalent Lie algebra actions on the space $V_S$ 
by a Lie algebra which is a central extension of the loop algebra of a Kac-Moody algebra where $S$ is a K3 surface such that 
\begin{itemize}
  \item $\NS(S)$ is generated by irreducible $-2$ curves
  \item Any pair of irreducible -2 curves on $S$ are either disjoint or intersect transversally at a single point.
\end{itemize}
More general results follow by deforming the complex structure of a given K3 surface to be related to one of this type.  

Before stating this result precisely, it is necessary to say something about the techniques involved and how they relate to possible extensions of the present work. 
First, the K-trivial birational geometry of quiver varieties of affine ADE type is captured by variation of GIT stability of quiver representations, while the K-trivial birational geometry of moduli spaces of sheaves on K3 surfaces is most completely described by variation of Bridgeland stability \cite{bayer2014projectivity,bellamy2018birational}, namely if $M = M_{H}(v)$ is the moduli space of $H$-stable sheaves of primitive Mukai vector $v$ then all K-trivial birational models of $M$ are the moduli spaces $M_{\sigma}(v)$ of $\sigma$-stable objects of Mukai vector $v$ for some stability condition $\sigma$. Also in fairly general settings (e.g. \cite{toda2018moduli}) but especially for K3 surfaces \cite{arbarello2018singularities}, the local structure of singularities of moduli spaces of sheaves induced by varying to a non-generic stability parameter is locally analytically (or \'etale locally) described by the variation of stability for a quiver variety, namely the Ext quiver of a polystable representative of a sheaf represented by a point in the singular moduli space.

The first ancillary result, which is proven as Theorem \ref{thm:finite_action}, is that subject to natural conditions on the stable factors of an element of $M_{\sigma}(v)$ and assuming that one has a local Ext-quiver description of the map contracting S-equivalent objects in $M_{\sigma}(v)$ under a stability condition $\sigma_0$, one can combine local finite dimensional Lie algebra actions on the cohomologies of the Ext quivers to produce global Lie algebra actions on cohomologies of moduli spaces of stable complexes. For a precise set of conditions see Definition \ref{def:amenable}, which we briefly record here.  

Let $S$ be any $K3$ surface and let $v\in H^*_{alg}(S)$ be a primitive Mukai vector. Let $\mc{S} = \{ s_1, \ldots, s_n\}$ be a set of Mukai vectors of norm $-2$ spanning a negative definite sublattice of $H^*_{alg}(S)$ such that for $s_{\mb t} \in \Z \mc{S}$ if $(v + s_{\mb t})^2 \ge -2$ then $v + s_{\mb t}$ is primitive.   Let $\sigma_0\in \Stab^{\dagger}(S)$  be a stability condition with an adjacent generic stability condition $\sigma$ such that the phase of an object in $M_{\sigma}(v)$ overlaps with that of a $\sigma_0$-stable object of Mukai vector $s_i$ for each $i$. Suppose that the $\sigma_0$-stable factors of a $\sigma$-stable object of Mukai vector in $v + \Z \mc S$ are a unique object in $v + \Z \mc S$ and some number of objects of Mukai vector in $\mc{S}$, and that the local map contracting $S$-equivalent objects is locally analytically isomorphic to the map from a quiver variety for generic stability parameter to the affine quotient for the Ext-quiver of a polystable object on the base. Then Definition \ref{def:amenable} says that the data $(v, \mc S, \sigma_0)$ is amenable to a local quiver Lie algebra action. This name is justified because of the following theorem. 

\begin{thm}
 Suppose $(v, \mc S, \sigma_0)$ is amenable to a local quiver Lie algebra action. Let $\mf{g}$ denote the semisimple Lie algebra whose Cartan matrix is the Gram matrix of the set $\mc{S}$. There is a space $M_0$ and maps \[\pi : M_{\sigma}(v + s_{\mb t}) \to M_0\]
for all $s_{\mb t} \in \Z\mc{S}$ and an action of $\mf{g}$  on the cohomology of the moduli spaces
\[\bigoplus_{s_{\mb t} \in \Z\mc{S}} H^*(M_{\sigma}(v+ s_{\mb t}))\]
induced by Lagrangian correspondences which are components of the fiber products
\[ M_{\sigma}(v + s_{\mb t}) \times_{M_0} M_{\sigma}(v + s_{\mb t'}).\]
\end{thm}

The main result then applies this result and combines them for several different finite dimensional Lie algebras. Here we restrict to the stated class of K3 surfaces so that $S$ is a K3 surface such that $\NS(S)$ is generated by irreducible $-2$ curves which are pairwise disjoint or intersect transversely at a single point. Let $\mf{g}$ be the Kac-Moody algebra whose Cartan matrix is the Gram matrix of the irreducible $-2$ curves. Define a Lie algebra $\hat{\mf g}(\NS(S))$ which as a vector space is 
\[\hat{\mf g}(\NS(S)) = \mf{g}[t^{\pm 1}] \oplus \Q c \oplus \Q d\]
  and has commutation relations as in \eqref{eq:affine}. The following is Theorem \ref{thm:main} in the main text. 

\begin{thm}
  Let $\NS(S)^\perp \subset H^*(S)$ denote the orthogonal complement of $\NS(S)$ in $H^*(S)$. Then there is an action of 
  $\hat{\mf g}(\NS(S))\oplus \Heis_{\NS(S)^\perp}$ on 
  \[V_S := \bigoplus_{\alpha \in H^2(S, \Z)} \bigoplus_{n \in \Z} H^*(M(1, \alpha, n)) \]
  such that the action of $\hat{\mf g}(\NS(S))$ satisfies the following:
  \begin{enumerate}[(i)]
    \item  The algebra $\hat{\mf g}(\NS(S))$ is generated by correspondences which are the compositions of components of fiber products over symplectic singularities of $M_{\sigma}(1, \alpha, n)$ with correspondences inducing birational transformations $M(1, \alpha, n)\dashrightarrow M_{\sigma}(1, \alpha, n)$ for stability conditions $\sigma\in \Stab^\dagger(S)$.  
    \item This action induces the quiver affine Lie algebra action $\hat{\mf g}$ on local quiver varieties corresponding to moduli spaces of rank 1 torsion free sheaves on $U\subset S$ where $U$ is biholomorphic to an affine ADE surface and  $\hat{\mf g}\subset \hat{\mf g}(\NS(S))$ is the corresponding affine Lie algebra. 
    \item This action of $\hat{\mf g}(\NS(S))\oplus \Heis_{\NS(S)^\perp}$ agrees with that defined by Fourier coefficients of vertex operators.
  \end{enumerate}
\end{thm}
All of the stability conditions considered may be chosen to have central charge $Z_{\omega, \beta}$ corresponding to complexified K\"ahler form $\beta + i\omega$ is such that there is a bound $N$ (potentially depending on $\alpha$ and $n$) and a contractible collection of $-2$ curves $\mc{C}$ on $S$ such that for all $C \in \mc{C}$ we have
\begin{align*} |\beta\cdot \omega| &< N & \omega^2 &\gg 0 & |\omega\cdot C| &< \frac{N}{\omega^2}.
\end{align*}
Geometrically, this corresponds to the fact that the volume of $S$ and hence each $M_{\sigma}(v)$ is arbitrarily large, but that each contractible curve $C\in \mc{C}$ is extremely small, and also that the volume of the exceptional divisor of the Hilbert-Chow map remains bounded and so is small compared to the volume of the moduli space. We get different affine Lie algebra actions on moduli spaces for these stability conditions depending on which collection $\mc{C}$ is chosen. 

Finally, en route to this result, a geometric modular interpretation of quiver varieties for affine ADE quivers and framing vector $\mb{w_0}$ is given, as long as the corresponding finite ADE diagram can occur as the dual graph of a collection of $-2$ curves on a K3 surface. This is Corollary \ref{cor:quiver_geometric_moduli} in the main text.

\begin{prop}
  Let $Q$ be an affine ADE quiver corresponding to a connected contractible collection $\mc C$ on $S$. Fix an open set $U \subset S_{\mc C}$ containing $\mc{C}$ which is biholomorphic to the corresponding ADE surface.  Fix framing vector $\mb{w_0}$. Then given a generic stability condition $\theta$ and dimension vector $\mb v$ there is a Mukai vector $v$, a generic stability condition $\sigma \in \Stab^\dagger(S)$ in a chamber which has a stability condition on its boundary inducing a contraction $\pi_{\mc C}$ onto $Sym^k(S_{\mc C})$ and an isomorphism 
  \[ \quiv{\theta}{\mb{v}}{\mb{w_0}} \simeq M_\sigma(v, U)\]
  between the corresponding quiver variety and an open set $M_\sigma(v, U) = \pi_{\mc {C}}^{-1}(Sym^k(U))$ of $M_\sigma(v)$ parametrizing $\sigma$-stable objects on $S$ of Mukai vector $v$. This correspondence is such that the birational transformations between different chambers in the stability space for the quiver are induced by those between different chambers in $\Stab^\dagger(S)$.
\end{prop}

This paper is organized as follows. In Section \ref{sec:VOA}, the action of Lie algebras given by the Frenkel-Kac construction in presented. Sections \ref{sec:quivers} - \ref{sec:birational_hilb} review necessary background on Nakajima quiver varieties, birational geometry and Bridgeland stability conditions. Section \ref{sec:stabk3_specific} describes the specific relevant limits of the space of stability conditions on K3 surfaces of the required form and Section \ref{sec:construction} constructs the required Lie algebra actions and provides the proof of the main theorem. 

\paragraph{Acknowledgements} I am extremely grateful for the help and advice of so many people, especially my advisor A. Okounkov, and many other including A. Bayer, A. Craw, I. Danilenko, A. Gyenge, A.J. de Jong, H. Liu, G. Sacc\`a, and J. Sawon. 

\section{Lattice VOAs and Cohomology}\label{sec:VOA}
\paragraph{K3 Lattices}
Recall the K3 lattice
\[\Lambda_{K3} = H^*(X,\Z) = U^{\oplus 4} \oplus -E_8 \oplus -E_8\]
is an even, integral lattice of signature $(4,20)$, where $U$ is the hyperbolic lattice with matrix $\begin{pmatrix} 0 & 1\\1 & 0\end{pmatrix}$, and the $E_8$ lattices are the usual one, where $-L$ refers to $L$ with the negative intersection pairing. Now for a K3 surfaces $S$, the Picard group is equal to the Neron Severi lattice and  $\on{Pic}(S) = \on{NS}(X) =  H^{1,1}(X,\C)\cap H^2(X, \Z)$ has signature $(1,\rho(S) -1 )$ and is contained in $\Lambda_{K3}$. Let $\Delta = \{\alpha\in \on{NS}(S) | \alpha^2 = -2\}$ be the classes with square -2, and let
\[W_S = \langle r_\alpha: v \mapsto v + 2\langle v , \alpha \rangle \alpha\rangle \]
denote the Coxeter group of reflections by these vectors. If we extend them to act on $\on{NS}(S)_\R := \on{NS}(S)\otimes \R$ then
$W_S \subset O(\on{NS}(S)_\R)$, further the cone $\Nef(S) = \overline{\on{Amp}(S)}$ is
a fundamental domain for the action of $W_S$ on the positive cone $C^+ \subset
\on{NS}_\R$.  Here $C_+$ is defined such that $C = \{ \alpha \in \on{NS}_\R | \alpha^2 > 0\} =
C^+\sqcup C^-$ and there is an ample class in $C^+$.  Technically the claim that $\Nef(S)$ is a fundamental domain is true only as long as $\Nef(S)$
avoids the boundary of $C^+$. In the case where $\Nef(S)$ hits the boundary of $C^+$ then $\Nef(S) \backslash (\Nef(S))\cap \partial C^+)$ is a fundamental domain for the action. For this paper we only consider the case that $\Nef(S) \cap \partial C^+$ is a finite number of points.   The only other case is that $\Nef(S)$ is the entire positive cone.

\paragraph{Kac-Moody algebras} Consider the negative of the intersection pairing on $\on{NS}(S)$, denoted $\langle -, - \rangle$ with signature $(\rho(S)-1, 1)$. We will denote the actual intersection pairing $\langle -, - \rangle_{\on{NS}}$.  Let $\Delta$ denote the set of irreducible $-2$ curves, whose corresponding hyperplanes form the walls of the NEF cone. If the symplectic automorphism group of $S$ is infinite then it is possible there are an infinite number. In general we also need to consider additional NEF classes to get a basis of our lattice. So the situation where the representation theory of Kac-Moody algebras is most closely related to the geometry of the K3 surface occurs when $\NS(S)$ is generated by irreducible $-2$ classes.  Then letting $C$ be the intersection matrix according to the simple roots (whose rank may be larger than $\rho(S)$ define the algebra algebra $\mf{g}(S) = \mf{g}(S, \Delta)$ which is generated by $e_i, f_i, h_i$ for $\alpha_i \in \Delta$ subject to the relations
\begin{align}
  [h_i, h_j] &= 0\label{bkm1}\\
[h_i, e_j] = \langle \alpha_i, \alpha_j\rangle e_j,&~~ [h_i, f_j] = -\langle
\alpha_i, \alpha_j\rangle f_j\label{bkm2}\\
[e_i, f_j] &= \delta_{ij} h_i\label{bkm3}\\
(\on{ad}(e_i))^{1 - \langle \alpha_i, \alpha_j\rangle}e_j &= (\on{ad}(f_i))^{1
- \langle \alpha_i, \alpha_j\rangle}f_j = 0\label{bkm4}\\
\langle \alpha_i, \alpha_j\rangle = 0 \Rightarrow [e_i, e_j] &= [f_i, f_j] =
0.\label{bkm5}
\end{align}
Then we have a triangular decomposition as usual with
\[ \mf{g}(S) = \left( \oplus_{\alpha \in \Delta_+} \mf{g}_\alpha\right) \oplus \mf{g}_0 \left( \oplus_{\alpha \in \Delta_+} \mf{g}_{-\alpha}\right)\]
where $\Delta_+$ is the set of effective divisors on $S$ of norm $\ge -2$. The Lie algebra $\mf{g}(S)$ comes with an invariant bilinear form $\langle\cdot, \cdot\rangle$ such that $\mf{g}_\alpha \perp \mf{g}_\beta$ if $\alpha\neq -\beta$ The roots accessible by the Weyl group from the real simple roots are called  real roots, and the corresponding root spaces have dimension 1. It is hard to write an explicit formula for the multiplicity of a general root.
\paragraph{Affinization of $\mf{g}(S)$}
From the invariant form $\langle\cdot, \cdot\rangle$ we can form the affinization of $\mf{g}(S)$, which as a vector space is $\mf{g}(S)\otimes \C[t, t^{-1}]\oplus \C c$ and the usual bracket
\begin{equation} \label{eq:affine}
  [x\otimes t^m, y\otimes t^n] = [x,y]\otimes t^{n+m} + m\langle x, y\rangle \delta_{m+n,0} c. 
\end{equation}
This algebra is not a Kac-Moody algebra although some of the theory carries through. Call this algebra $\tilde{\mf g}(\on{NS}(S))$. If we adjoin an outer derivation $d$ so that $[d, x\otimes t^n] = n x\otimes t^n$ then we get another Lie algebra $\hat{\mf g}(\on{NS}(S)) = \tilde{\mf g}(\on{NS}(S)) \oplus \C d$. In general we will use $\tilde{\mf{g}}$ ($\hat{\mf g}$) for the affinization without (with) outer derivation $d$, when $\mf g$ is a Kac-Moody algebra.

\paragraph{Vertex algebra from K3, Mukai or NS Lattice} We apply the well-known construction of a vertex algebra from an even lattice $\Lambda$ \cite{borcherds1986vertex, frenkel1980basic} in the cases relevant to K3 surfaces. See also \cite[ch.~ 9]{nakajima1999lectures}.
Let $\Lambda = \on{NS}(S)$ or $\Lambda = M := \on{NS}(S)\oplus H^0(S, \Z) \oplus H^4(S,\Z)$ the Mukai lattice, or $\Lambda = \Lambda_{K3}$ the full K3 lattice. Let $\on{Heis}_\Lambda$ be the corresponding Heisenberg algebra and $\mc{F}_\Lambda(1)$ the level 1 Fock representation (over $\Q$). Let $\epsilon: \Lambda\times\Lambda \to \{\pm 1\}$ be a 2-cocycle and let $\Q[\Lambda]_\epsilon$ denote the group algebra twisted by $\epsilon$, i.e.
\[e^\alpha e^\beta = \epsilon(\alpha, \beta) e^{\alpha+\beta}.\]
Then let
\begin{equation}\label{eq:VOA}
  V_\Lambda := \mathcal{F}_\Lambda(1) \otimes \Q[L]_\epsilon
\end{equation}
extend the action of $\alpha(0) \in \on{Heis}_\Lambda$ by
\begin{align}\label{heis_deg_0_wt}
  \alpha(0) (v\otimes e^\beta) = \langle \alpha, \beta\rangle v \otimes e^\beta.
\end{align}

Note that for an appropriate choice of cocycle $V_{\Lambda_1 \oplus \Lambda_2} = V_{\Lambda_1}\otimes V_{\Lambda_2}$. For example $V_M = V_{\on{NS}(S)} \otimes V_U$ where $U$ is the hyperbolic lattice corresponding to the pairing with matrix
\[\begin{pmatrix} 0 & 1\\ 1 & 0\end{pmatrix}.\]
Next define
\[\phi_\alpha(z) = \sum_{n\in \Z \backslash \{0\}} \alpha(n) \frac{z^{-n}}{-n} + \alpha(0) \log(z) + \alpha\]
and the vertex operator
\begin{align*}
X(e^\alpha, z) &= :e^{\phi_\alpha(z)}:\\
&= \exp(\phi_\alpha(z)_-)e^\alpha z^{\alpha(0)} \exp(\phi_\alpha(z)_+)\\
&= \sum_{n\in \Z} x_n(\alpha) z^{-n}
\end{align*}
as an element of $\on{End} V_\Lambda \otimes \Q [[z, z^{-1}]]$.
The vertex operators $X(v, z)$ for other $v\in V_\Lambda$ are recovered by the reconstruction theorem (see e.g. \cite{frenkel_vertex_2004} ). In fact, this equation for $X(e^\alpha, z)$ is essentially uniquely determined by the requirement that $V_{\Lambda}$ form a vertex algebra. This is captured in the following proposition, for which \cite[Prop. 5.4]{kac_vertex_1998} is a good reference. 
\begin{prop}[Frenkel-Kac]\label{prop:frenkel_kac_follows_from_VOA}
Let $\Lambda$ be an even lattice and let $V_\Lambda$ admit the structure of a vertex algebra such that 
\[Y(h(-1)|\mathrm{vac}\rangle \otimes 1,z) = \sum_{n\in \Z} h(n) z^{-n-1} \text{ for } h\in \mc{F}_{\Lambda}(1)\] 
and $Y(e^{\alpha}, w)$ has the same OPE (or a forteriori the same commutation relations between Fourier coefficients) with $h(z)$ as does $X(e^{\alpha}, w)$. Then these vertex operators in fact determine a VOA structure on $V_\Lambda$ such that $Y(e^\alpha, w) = X(e^\alpha, w)$, and this determination is unique up to the choice of cocycle $\epsilon$. 
\end{prop}

Let $\mf h = \Lambda \otimes_{\Z}\Q$. The Fourier coefficients of vertex operators can be seen to have the following commutation relations:
\begin{align}
\label{hx_fourier_rel}[h(n), x_m(\alpha)] &= \langle h, \alpha\rangle x_{n+m}(\alpha), ~h\in \mf{h}\\
\label{xx_fourier_rel}[x_n(\alpha), x_m(\beta)] &= \begin{cases} 0 & \langle \alpha, \beta \rangle \ge 0\\
x_{n+m}(\alpha+ \beta) & \langle \alpha, \beta \rangle = -1\\
\alpha(n+m) + n\delta_{n+m,0} & \alpha + \beta = 0 \text{ and } \ab{\alpha}{\alpha}  = 2.
\end{cases}
\end{align}
\paragraph{Lie algebra from some Fourier Coefficients}
Again recall the restriction that $\on{NS}(S)$ has a basis of irreducible -2 curves so the negative of the intersection matrix has diagonal entries 2 in this basis. Further assume that for any two curves $C_1, C_2$ in this basis they have intersection 1 or 0 so non-zero off diagonal entries are -1. Then $V_{NS(S)}$ is a representation of ${\mf g}(\on{NS}(S))$. 

Using this lemma and the relations between Fourier coefficients of vertex operators we deduce
\begin{prop}\label{prop:main_from_rep_theory}
Let $S$ be a K3 surfaces with generators $\{\alpha_i\}$ for $NS(S)$ of irreducible -2 curves which have intersections with each other 1 or 0. Then $\tilde{\mf{g}}(\on{NS}(S))$ acts on $V_{\on{NS}(S)}$ by
\begin{align*}
  h_i\otimes t^n &\mapsto h_i(n)\\
  e_i\otimes t^n &\mapsto x_{n}(\alpha_i)\\
  f_i\otimes t^n &\mapsto x_{n}(-\alpha_i)\\
  c&\mapsto 1
\end{align*}
and $d$ acts by the degree operator.
\end{prop}
\begin{proof}
  The only non-trivial part is the relation \eqref{bkm4} between $\mf{g}(\NS(S))$ coefficients of the polynomials in $t$. 
  This follows from  a standard argument based on integrability, see e.g. \cite[\S 9.iii]{nakajima1998quiver}.  Any pair of real simple roots $\alpha_1, \alpha_2$ the representation $V$ decomposes into finite dimensional representations of the Lie algebra generated by $e_{\alpha_i}, f_{\alpha_i}$ for $i = 1,2$. Then \cite[ch. 3]{kac_infinite-dimensional_1990} implies \eqref{bkm4} for coefficients. 
\end{proof}

\paragraph{Slightly larger algebra}
As a consequence of the previous we also get an action of $\tilde{\mf{g}}(\on{NS}(S))$ on $V_M$ for the Mukai lattice $M = \on{NS(S)} \oplus U$ by
$x(v\otimes w) = x(v) \otimes w$ for $v\otimes w \in V_{\on{NS}(S)}\otimes V_{U} \simeq V_M$, and similarly on $V_{\Lambda_{K3}}$.

It will be useful for us, since we care about all cohomology and not just algebraic cohomology or middle dimensional algebraic cohomology to consider slightly larger algebras. Currently we have an algebra $\hat{\mf{g}}(\NS(S))$ generated by Fourier coefficients of vertex operators related to algebraic curve classes and Nakajima operators related to algebraic curve classes. We would like to capture
 Fourier coefficients of vertex operators related to algebraic curve classes and Nakajima operators related to more general classes, either algebraic classes lying in the Mukai lattice or any cohomology classes.

 To this end consider the Heisenberg algebras $\Heis(H^0(S) \oplus H^4(S))$ and $\Heis(T(S))$ (taken with the \emph{negative} pairing $\langle -, - \rangle$) where $T(S) = NS(S)^{\perp}\subset H^2(S, \Z)$ is the transcendental lattice. Then define the lie algebras
 \begin{align}
   \label{eq:g_alg} \tilde{\mf{g}}_{alg}(S) &:= \Heis(H^0(S) \oplus H^4(S)) \oplus \hat{\mf{g}}(\NS(S))\\
   \label{eq:g_k3} \tilde{\mf{g}}(S) &:= \Heis(T(S)) \oplus \hat{\mf{g}}_{alg}(S)
\end{align}
Note that $\tilde{\mf{g}}_{alg}(S)$ contains the algebra $\Heis(M)$ modelled on the Mukai lattice and
$\tilde{\mf{g}}(S)$ likewise contains $\Heis(\Lambda_{K3})$ modelled on the entire K3 lattice. 

\paragraph{Fock space as Cohomology of Moduli Spaces}
In general, moduli spaces of sheaves on K3 surfaces, if smooth, are hyperk\"ahler varieties deformation equivalent to Hilbert schemes on points on a K3. For the time being, we will consider moduli spaces of torsion free rank 1 sheaves which are naturally isomorphic to Hilbert schemes of points.  The standard notation in the study of moduli of sheaves on K3 surfaces is to label the discrete invariants of a sheaf $\mc{E}$ by its \emph{Mukai vector} given by
\begin{align*}
  v(\mc E) &= \on{ch}(\mc E)\sqrt{\on{Td}_X}\\
  &= (r, c_1, \on{ch}_2).(1,0,1)\\
  &= (r, c_1,  c_1^2/2 - c_2 + r)\\
  &\in H^0(X, \Z) \oplus H^{1,1}(X,\Z) \oplus H^4(X, \Z)
\end{align*}
then given a Mukai vector $v$ and a polarization $H$ of $X$ there is a moduli space $M_H(v)$ of Gieseker semistable sheaves of Mukai vector $v$. If non-empty, this space has dimension $\ab{v}{v} + 2$ and the stable locus is smooth. The rank one moduli spaces do not depend on the polarization, so we suppress the notation in this case. Supposing $L_\alpha$ is a line bundle with $c_1(L_\alpha) = \alpha\in \on{NS}(X)$ then $M(1, \alpha, \on{ch}_2 + 1)$ consists of subsheaves $\mc E$ of $L_\alpha = \mc{E}^{\vee \vee}$ such that the quotient has finite length $ n := \on{ch}_2 (\mc{E}^{\vee\vee}) - \on{ch}_2(\mc{E})$. It follows that we can also consider the moduli space as parametrizing the quotient, and so
\begin{equation}\label{eq:hilb_eq_m}
  M(1, \alpha, \on{ch}_2 + 1) \simeq S^{[n]} .
\end{equation}
On the other hand, we know that
\[ H_*(\bigsqcup_{n\ge 0} S^{[n]}) \simeq \mathcal{F}_{H_*(S)}(1) = \mc{F}_{\Lambda_{K3}}(1)\]
so
\begin{equation}\label{eq:rk_1_torsion_free_V}
  V:= H_*(\bigsqcup_{\substack{\alpha \in \on{NS}(S)\\ \on{ch_2}\in \Z}} M(1, \alpha, \on{ch}_2 + 1)) \simeq \mathcal{F}_{\Lambda_{K3}}(1) \otimes \Q[\on{NS}(S)]
\end{equation}
but this space $V$ lies inside $V_{\Lambda_{K3}}$ and is preserved by the action of the whole Heisenberg subalgebra $\Heis(\Lambda_{K3})$ (which acts by constants or in the first factor) as well as all of $\hat{\mf g}(\on{NS}(S))$ and therefore is preserved by $\hat{\mf g}(S)$. We write $V$ without a subscript because it is of fundamental importance.

\paragraph{Notation for cohomology classes}
We record notation for the Nakajima basis on the Fock space $\mc{F}_{\Lambda_{K3}}(1)$. Let $\mathfrak{p}_{-i}(\gamma), \mathfrak{p}_{i}(\gamma)$ for $i \in \Z_{> 0}$, $\gamma \in H^*(X, \Z)$ denote the Nakajima creation and annihilation operators respectively. Let $\omega_X$ denote the class dual to a point. Further, if $|\text{vac}\rangle\in H^*(X^{[0]})$ is the vacuum of the Fock space we use the notation
\begin{align*}
  |n_{1, \alpha_1}^{k_1}, \ldots, n_{j, \alpha_j}^{k_j}\rangle := \mf{p}_{-n_1} (\alpha_1)^{k_1}\cdots \mf{p}_{-n_j} (\beta)^{k_j} |\text{vac}\rangle
\end{align*}
for elements of the Fock space, where in addition we drop the subscript $\alpha$ if $\alpha = \omega_X$ is the class of a point.  Further, we know that if $d\ge 2$ then the classes
\begin{align*}
  C_\beta := |1_{\beta},  1^{d-1}\rangle &= \mf{p}_{-1} (\beta)\mf{p}_{-1} (\omega_X)^{d-1} |\text{vac}\rangle\\
  D := |2,   1^{d-2}\rangle &= \mf{p}_{-2} (\omega_X)\mf{p}_{-1} (\omega_X)^{d-2} |\text{vac}\rangle
\end{align*}
generate $H_2(X^{[d]}; \Z)$ in the sense that
\begin{equation}\label{eq:h2xn}
  H_2(X^{[d]}; \Z) = \{ |1_{\beta},  1^{d-1}\rangle + k |2,   1^{d-2}\rangle  ~ | k\in \Z, \beta \in H^2(X, \Z)\}
\end{equation}

\paragraph{Virasoro algebra}
The usual construction to produce a Virasoro action in $V_{\Lambda}$ works for any choice of $\Lambda$, e.g. $\Lambda = \on{NS}(S)$ and $\Lambda = \Lambda_{K3}$. For our purposes the most sensible one is the one coming from $\Lambda_{K3}$. Then $L(0)$ has commutation relations with elements of $\tilde{\mf g}(S)$ given by
\[[L(0), h(n)] = -n h(n) ~~~ [L(0), x_\alpha(n)] = -nx_\alpha(n)\]
and in this way we extend our Lie algebra to $\hat{\mf g}(S)$ which contains $\tilde{\mf g}(S)$ and also $\hat{\mf g}(\on{NS}(S))$ by the equality $ d = -L(0)|_{\tilde{\mf g}(\on{NS}(S))}$. We will now extend this and call $d =-L(0)$.
\paragraph{Roots}
Our Lie algebra $\hat{\mf g}(S)$ has Cartan subalgebra
\begin{align*}
\mf h &:= H^{*}(S,\Q) \oplus \Q c \oplus \Q d\\
&= \Lambda_{K3}\otimes \Q \oplus \Q c \oplus \Q d
\end{align*}
with dual
\begin{align*}
\mf h^* &= \Lambda_{K3}^{\vee}\otimes \Q \oplus \Q \Lambda \oplus \Q \delta
\end{align*}
Now let $\mathring{\Delta} = \mathring{\Delta}_{+, re} \cup\mathring{\Delta}_{-, re} \cup\mathring{\Delta}_{+, im} \cup\mathring{\Delta}_{-, im}$ be the root system for the Kac-Moody algebra $\mf{g}(\on{NS}(S))$ and $\mf{g}(\on{NS}(S)) = \mf{g}_0\oplus\bigoplus \mf{g}_{\alpha}$ the root space decomposition. Then
\begin{equation}\label{affinization_root_decomp}
  \hat{\mf g}(S) = \mf{h} \oplus \bigoplus_{\alpha\in \mathring{\Delta}, n\in \Z} \mf{g}_{\alpha}\otimes t^n\oplus \bigoplus_{n\in \Z\backslash\{0\}} \Lambda_{K3, \Q}\otimes t^n
\end{equation}
is a root space decomposition (just as in building the affine algebra from a finite dimensional Lie algebra)
and so our roots $\Delta \subset \mf{h}^*$ are given by
\[\Delta = \{ \alpha + n\delta |\alpha \in \mathring{\Delta}, n\in \Z \}\cup \{n \delta | n \in \Z \backslash \{0\}\}.\]
\paragraph{Weights in $V_M$}
Using \eqref{heis_deg_0_wt} we know the weight decomposition of $V$ with respect to
$\Lambda^{\vee}_{K3,\Q}$ is just the grading on $\on{NS}(S)$. If $h\in \Lambda_{K3, \Q}^{\vee}$ then $h$
acts by $\ab{h}{\alpha}$ on $\mathcal{F}_{\Lambda_{K3}}(1) \otimes e^{\alpha}$. But on $1\otimes
e^\alpha$ we know that the $L(0)_{\on{NS}}$ operator coming from $V_{\on{NS}(S)}$ acts by $-
\on{deg}(v) - \ab{\alpha}{ \alpha}/2$ on $v\otimes e^\alpha$ but choosing an orthonormal basis
$\{b_i\}_{i\in I}$ for $U_\Q\oplus T(S)_{\Q}$ our $L(0)$ is given as
\[L(0) = L(0)_{\on{NS}} + \frac{1}{2}\sum_{k\in \Z, i\in I} : b_i(-k)b_i(k):\]
and so $L(0)$ acts according to the same formula, as $-\deg(v) - \ab{\alpha}{\alpha}/2$. Thus $d$ acts by the scalar $\on{ch}_2$ on the component $H_*(M(1,c_1, \on{ch_2}+1))$. Note that under \eqref{eq:hilb_eq_m} increasing the number of points by 1 decreases $\on{ch_2}$ by $1$. Thus our grading implicit in the indices on $M(1, \alpha, ch_2 + 1)$ is is the weight grading on $V$.

\paragraph{Useful choices of Picard lattices}
Some choices of Picard lattice will be very useful for proving general things about K3 surfaces, and also to serve as examples. The first one is a generic elliptic K3 with a section whose intersection pairing has matrix
\[ \ttmx{0}{1}{1}{-2}.\]
in this case, $\NS(S)$ is not generated by $-2$ classes. It is interesting to note that if we look at the Lie algebra spanned by the Fourier coefficients of the vertex operators for the $-2$ class $S$ for the section and the norm $0$ fibre class $F$, instead of getting the Borcherds-Kac-Moody algebra with the Cartan matrix, we get a Lie algebra with larger abelian symmetry.  In particular the Fourier coefficients $x_n(F)$ and $x_m(-F)$ commute with each other for all $n$ and $m$.  

It will be more useful to consider the case where there is one singular fiber in the elliptic fibration of type other than $\RN{1}_0$. For example, with one type $\RN{1}_n$ fiber for $n  = 2,3$ this will have intersection matrix
\[\begin{pmatrix} -2 & 2 & 0 \\ 2 & -2 & 1 \\ 0 & 1 & -2\end{pmatrix}, ~ \begin{pmatrix} -2 & 1 & 1 & 0  \\ 1 & -2 & 1 & 0 \\ 1 & 1 & -2 & 1  \\ 0 & 0 & 1 & -2  \end{pmatrix} \]
respectively. The existence of K3 surfaces with these two intersection pairings follows from \cite{morrison1984k3} (see also \cite[ch.~14]{huybrechts2016lectures}) where it is shown that every even Lorentzian lattice of rank $\rho \le 10$ occurs as $\on{Pic}(S)$ for a K3 surface $S$. More generally, the same argument will show that if we consider the intersection pairing of the form
\[\left( \begin{array}{ccc | c}
&&& 0\\
&-A&&1\\
&&&\vdots\\\hline
0&1&\cdots&-2
\end{array}\right)\]
where $A$ is the Cartan matrix of affine type $\tilde{A}_n$ for $n \le 9$, $\tilde{D}_n$ for $4 \le n \le
9$ or $\tilde{E}_6, \tilde{E}_7, \tilde{E}_8$ and the single one in the off-diagonal blocks is
placed in such a way that it corresponds to adding a single extra node with one edge connecting
the new node to a node of the affine Dynkin diagram which has label 1, i.e. the corresponding irreducible component of the singular fiber has multiplicity 1.  Then a K3 surface $S$ exists
with this intersection pairing on $\on{Pic}(S)$, the null root corresponds to class of the elliptic
fiber, which admits one singular fiber which we take to be of type $\RN{1}_n$, $\RN{1}_n^*,
\RN{2}^*, \RN{3}^*, \RN{4}^*$ respectively, and the additional node corresponds to the class of
the section. Technically this argument doesn't show that we can produce $\RN{1}_2$ and  $\RN{1}_3$
fibers instead of type $\RN{2}, \RN{2}$ or $\RN{4}$, but it is true that these K3 surfaces
exists with the desired fibers, and in fact is is shown that there exist K3 surfaces with a single $\RN{1}_n$ fiber and a section for $2\le n \le 19$ in \cite{miranda1989configurations}. Because the total Euler characteristic of the singular fibers must sum to 24, there is a bound on which singular fibers may occur. For example, the Euler characteristic of a $\RN{1}_n^*$ fiber, corresponding to $\tilde{D}_{n+4}$ is $n + 6$, so the largest possible $n$ for which this occurs as the fiber of a elliptic K3 with section is $18$. I am not sure whether an elliptic K3 exists with a section, one $\RN{1}_n^*$ fiber and $24 - (n + 6)$ $\RN{1}_0$ fibers when $10 \le n \le 18$.

Even more generally we can consider elliptic K3 surfaces with one or more sections and with multiple singular fibers of type $\RN{1}_n$, $\RN{1}_m^*, \RN{2}^*, \RN{3}^*, \RN{4}^*$ which, if $n\ge 3$ correspond to Cartan matrices with $-2$ along the diagonal and $1$ everywhere else if we ignore the effect of the linear relations stemming from the fact that the fiber lies in the same cohomology class regardless of whether it comes from a linear combination of $(-2)$-curves in one singular fiber or in a different one.

\section{Nakajima Quiver Varieties and Hilbert Schemes on ADE surfaces}\label{sec:quivers}
We first recall from \cite{nakajima1994instantons, nakajima1998quiver} the general construction of Nakajima quiver varieties. Let $Q$ be a quiver, i.e. a directed graph with vertex set $I$ and edges $E$ where loops and multiple edges are allowed. Consider the doubled quiver $\overline{Q}$ with vertices $I \sqcup \overline I$ doubled and edges $E \sqcup E^T \sqcup \{ i \to \overline i | i\in I\} \sqcup  \{ \overline i \to i | i\in \overline I\}$ with transposes included and one edge to and from each doubled vertex to the original vertex.

Given a graded vector space $V$ let $\dim_I V\in \N^I$ denote its graded dimension vector. Given $\mathbf{v}, \mathbf{w}\in \N^{I}$ let $\on{Rep}_{\overline Q}(\mathbf{v}, \mathbf{w})$  denote the variety of representations of $\overline Q$ where the dimensions of vertices in the index set $I$ are given by the coordinate in $\mathbf{v}$ (the dimension vector) and the dimensions of coordinates in $\overline I$ are given by coordinates of $\mathbf{w}$ (the framing vector). Let $V = \oplus_{i\in I} V_i$ and $W = \oplus_{i\in I} W_i$ be graded vector spaces with dimension vectors $\mathbf{v}$ and $\mathbf{w}$. Then
\[\on{Rep}_{\overline Q}(\mathbf{v}, \mathbf{w}) = \big[\bigoplus_{e\in E} \Hom(V_{s(e)}, V_{t(e)}) \oplus  \Hom(V_{t(e)}, V_{s(e)}) \big] \oplus \Hom(W, V) \oplus \Hom(V,W) \]
where $\Hom$s between graded vector spaces are taken in the graded sense. A point in $\on{Rep}_{\overline Q}(\mathbf{v}, \mathbf{w}) $ will be denoted $(B, i, j)$.

The groups $G_{\mathbf v} := \prod \GL(v_i)$ and $G_{\mathbf w} := \prod\GL(w_i)$ with lie algebras $\mf{g}_{\mathbf v}, \mf{g}_{\mathbf w}$ act on $\on{Rep}_{\overline Q}(\mathbf{v}, \mathbf{w})$ preserving the symplectic form $\omega$ which arises because $\on{Rep}_{\overline Q}(\mathbf{v}, \mathbf{w})$
is the cotangent bundle to the representations of the quiver with half the edges.  Finally consider the moment map
\[ \mu : \on{Rep}_{\overline Q}(\mathbf{v}, \mathbf{w})\to \mf{g}_{\mathbf v}^*\]
and a vector $\theta \in \Z^I$ corresponding to the character
\[ \prod \det(g_i)^{-\theta_i} \in \C^\times\]
of $G_{\mathbf{v}}$. Values of $\theta$ in $\R^I$ are also considered but the quiver variety is defined with respect to a Hyperk\"ahler quotient rather than how it's defined below. Then for an appropriate value of $\zeta$ (i.e., $\zeta$ is a fixed point of the coadjoint action)
\begin{defn}
  a Nakajima quiver variety is the GIT quotient
  \[ \mathfrak{M}_{\theta, \zeta}(\mathbf{v}, \mathbf{w}) := \mu^{-1}(\zeta)/\!\!/_{\theta} G_\mathbf{v}.\]
\end{defn}
The cases where $\zeta = 0$ play a central role for us, so if there is only one subscript we assume $\zeta = 0$, i.e.
\[  \mathfrak{M}_{\theta}(\mathbf{v}, \mathbf{w}) :=  \mathfrak{M}_{\theta, 0}(\mathbf{v}, \mathbf{w}).\]
Also, letting $\mathfrak{z}_{\mathbf{v}}$ denote the fixed points of the coadjoint orbit, we will use the fact that we may first perform the algebraic quotient on $\mu^{-1}(\mathfrak{z}_{\mathbf{v}})$ and then take the fiber over $\zeta$ and recover $\quivvar$, meaning that for fixed $\theta$, different quiver varieties fit into a family
\begin{equation}\label{eq:quiverfamily}
  \widetilde{\mf{M}}_{\theta}(\mathbf{v}, \mathbf{w}) \to \mf{z}_{\mathbf{v}}
\end{equation}
whose fiber over $\zeta$ is $\quivvar$.

Let $\mf{M}^{reg}_{\theta, \zeta}(\mathbf{v}, \mathbf{w})\subset \quivvar$ denote the regular locus, which coincides with the locus of stable quiver representations.

\paragraph{Tautological bundles} The trivial bundles of rank $v_i$ on $\mu^{-1}(\zeta)$ pass to the quotient via descent where we consider them with the the defining action of $\GL(v_i)$ on $\C^{v_i}$ and the trivial action for other factors of $G_{\mathbf{v}}$. In this way we get a \emph{tautological bundle} denoted
\begin{equation}\label{eq:tautological_bundle}
  \mc{V}_i \to \quivvar
\end{equation}  of rank $\mathbf{v}_i$ for every $i\in I$. For every edge $e$ (including the doubles ones) between $i, j\in I$ there is also a bundle map the \emph{tautological map} denoted $\phi_e : \mc V_i \to \mc V_j$ which is also constructed by descent.

\paragraph{Wall and chamber structure} How are quiver varieties with different stability parameters related? First, we know that $\mathfrak{M}_{0, \zeta}(\mathbf{v}, \mathbf{w}) = \Spec \C[\mu^{-1}(\zeta)]^{G_{\mathbf{v}}}$ is an affine variety and by the general theory of GIT quotients there is a projective map
\[\mathfrak{M}_{\theta, \zeta}(\mathbf{v}, \mathbf{w}) \to \mathfrak{M}_{0, \zeta}(\mathbf{v}, \mathbf{w})\]
which is often an equivariant symplectic resolution. We also have the following key result. Let $A$ be the adjacency matrix of $Q$ and $C = 2I - A$ the Cartan matrix. Then define the \emph{positive roots}
of which only a few are relevant to a given dimension vector:
\[R_+ := \{ \theta \in \N^ I \mid \theta \cdot  C \theta \le 2, \theta \neq 0\}\]
\[R_+ (\mathbf{v}) := \{ \theta \in R_+ \mid  \theta_i \le v_i\} \]
\[D_\theta := \{ \alpha\in \R^I \mid \alpha \cdot \theta = 0 \}\]
note that the walls $D_\theta$ for $\theta \in R_+(\mathbf{v})$ give a polyhedral decomposition of $\R^I$, and a face will refer to any dimensional face of this decomposition.
\begin{thm}[Nakajima, \cite{nakajima1994instantons, nakajima2009quiver}]\label{thm:partial_resolutions} With the above definitions
\begin{enumerate}
  \item[(i)] When $\alpha \cdot \theta \neq 0$ for all $\alpha \in R_+ (\mathbf{v})$ or $\alpha \cdot \zeta \neq 0 $ for all $\alpha \in R_+ (\mathbf{v})$, there are no strictly semistable points in $\quivvar$. Further, for two different generic stability parameters the corresponding varieties are $\C^*$-equivariantly diffeomorphic.
  \item[(ii)] If $\zeta = 0$, and $\theta, \theta'$ lie in the same face, stability (semistability) for $\theta$ is equivalent to stability (semistability) for $\theta'$ and the corresponding quiver varieties are naturally isomorphic.
  \item[(iii)] If $\zeta = 0$ and $F' \subset \overline F$ with $\theta' \in F', \theta \in F$, then $\theta$-semistable $\implies \theta'$-semistable and $\theta'$-stable $\implies \theta$-stable. Further, there exists a natural projective map
  \begin{equation}\label{eq:semisimplification}
    \pi_{\theta, \theta'} :   \mathfrak{M}_{\theta}(\mathbf{v}, \mathbf{w}) \to  \mathfrak{M}_{\theta'}(\mathbf{v}, \mathbf{w})
  \end{equation}
  sending a quiver representation $V$ to the direct sum of its Jordan H\"older factors $\gr_{\theta'}(V)$ with respect to $\theta'$.
\end{enumerate}
\end{thm}
\begin{rmk}
The set of walls $W$ where the map $\pi_{\theta, \theta'}$ is not generically an isomoprhism onto its image is contained in the set $\{D_\theta \mid \theta \in R_+(\mb{v})\}$ but when the quiver is not a finite type quiver there may be walls $D_\theta$ where the map $\pi_{\theta, \theta'}$ is not surjective, and is an isomoprhism onto its image. Section \ref{ssec:bellamy_craw} following \cite{bellamy2018birational} describes the walls inducing non-trivial contractions in the cases relevant to the present work.
\end{rmk}
\paragraph{Reduction to $\mb{w} = 0$} We review a formulation of quiver representations due to Crawley-Boevey \cite{crawley2001geometry} which replaces the framing nodes with a single extra vertex which is useful in many circumstances.

Let $Q$ be a quiver with vertices $I$ and edges $E$. Given dimension vector $\mathbf{v}$ and framing vector $\mathbf{w}$ let $Q_\infty(\mathbf{w})$ be the quiver with vertices $\{\infty\}\sqcup I$ and edges $E\sqcup E^T\sqcup W\sqcup W^T$ where $W$ consists of $w_i$ edges from $\infty$ to the $i$th vertex of $I$, and $W^T$ are the same edges with orientation reversed. We write $Q_\infty$ for $Q_\infty(\mb{w})$ when the framing vector is understood.  Then there is a $G_{\mathbf{v}}$ and $G_{\mathbf{w}}$ equivariant isomorphism
\begin{equation}\label{eq:cb_trick_rep_equiv}
   \on{Rep}_{\overline Q}(\mathbf{v}, \mathbf{w}) \simeq  \on{Rep}_{ Q_\infty(\mathbf{w})}((1, \mathbf{v}))
 \end{equation}
where $G_{\mb{w}}$ now acts on the edge spaces.  Thus the moment maps for the $G_{\mb{v}}$ actions coincide but we refer to the moment map $\on{Rep}_{ Q_\infty(\mathbf{w})}((1, \mathbf{v}))$ as $\mu_\infty$ for clarity. Pick a complex stability parameter $\zeta\in \mf{z}_{\mb{v}}$. The natural group to act on $\on{Rep}_{ Q_\infty(\mathbf{w})}((1, \mathbf{v}))$ is $\GL(1) \times G_{\mb{v}}$ but the diagonal $\C^*$ acts trivially so we take $G_{\mb{v}, \infty} = \GL(1) \times G_{\mb{v}}/ \GL(1)$. Because of this the natural characters correspond to vectors $\theta^\infty = (\theta_\infty, \theta_1, \ldots, \theta_r)$ with $\theta ^\infty( (1, \mb{v}) )= 0$ giving the rational character
\[ \prod \det (g_i)^{-\theta_i}\]
of $G_{\mb{v}, \infty}$. There is a 1-1  correspondence between rational characters $\theta = (\theta_1, \ldots, \theta_r)$ of $G_{\mb{v}}$ and characters $\theta^\infty = ( - \theta(\mb{v}), \theta_1, \ldots, \theta_r)$ of $G_{\mb{v}, \infty}$ such that
\[\quivb{\theta, \zeta}{v}{w} = \mu_\infty^{-1}(\zeta)\GIT_{\theta^\infty} G_{\mb{v}, \infty}.  \]

\paragraph{Algebraic descritption of stability} We recall an algebraic formulation of stability for quiver representations introduced by King in \cite{king1994moduli} equivalent to stability defining the GIT quotient. See also \cite{ginzburg2009lectures}.

\begin{prop}[King]\label{prop:king_stability}
A point $(B,i,j) \in \mu^{-1}(\zeta)$ corresponding to a representation $V$ is $\theta$-semistable if and only if for every $B$-invariant subspace $S\subset V$ we have
\begin{align*}
S \subset \ker(j) &\Rightarrow \theta\cdot \dim_{I} S \le 0\\
S  \supset \im(i) &\Rightarrow  \theta \cdot \dim_I S \le \theta \cdot \dim_I V.
\end{align*}
The point is $\theta$-stable if in addition for non-zero proper subrepresentations $S\subset V$ the inequalities are strict.
\end{prop}

\subsection{Affine ADE quivers} One of the most important cases for present applications is when the quiver $Q$ is of affine ADE type such that the affine root is the zeroth index in $I = \{0, \ldots, r\}$. Unless otherwise stated, a quiver variety will correspond to a quiver of this type for the remainder of the paper.  These quiver varieties for some stability parameters have interpretations as moduli spaces of framed sheaves on ADE surfaces or as $\Gamma$-equivariant sheaves on $\C^2$ for a finite group $\Gamma$ acting on $\C^2$. We briefly recall the relevant isomorphism between quiver representations and rank 1 torsion free sheaves, see \cite{nakajima1994instantons, nakajima2007sheaves}.

\paragraph{Gieseker/Hilbert choice of stability parameter} There is a close relationship between the wall and chamber structure in this case with the Weyl chambers of the corresponding affine root system. Let $\delta = (1,\delta_1, \ldots,\delta_r)$ denote the dimension vector corresponding to the null vector for the Cartan matrix with all positive entries and value $1$ on the affine root. Note that this is also the multiplicities of components in a singular elliptic fiber with the corresponding Dynkin diagram. Let $\mathbf{w_0}$ be the framing vector $(1,0,\ldots, 0)$ corresponding to the affine root. The complex stability parameter $\zeta$ can vary in the space $\C^I$ and $\theta$ can vary in $\Q^I$ (or $\R^I$ if we consider Hyper-K\"ahler quotients). First we fix $\zeta = 0$. As long as $\theta \not\in D_\delta$ then $\mf{M}_\theta(\delta, \mathbf{w_0})$ is isomorphic to the ADE surface $X_\Gamma = \widetilde{\C^2/\Gamma}$ where $\Gamma$ is the finite group corresponding to $Q$ \cite{kronheimer1989construction}. This is also true whenever $\zeta$ lies in $\C\otimes D_{\delta}$ and $\theta\not \in D_\theta$.

Now we intend to find a stability condition $\theta$ which results in $\mf{M}_\theta(\mathbf{v}, \mathbf{w})$ having the desired interpretation as a moduli space of sheaves. There is not a single $\theta$ which works for all $\mathbf{v}$ because the set of walls is not locally finite in the region of $\R^I$ adjacent to $D_\delta$. The set of walls is however finite for any fixed $\mathbf{v}$. The real root hyperplanes partition $D_\delta$ according to the finite ADE root system. Let $C$ denote the usual choice of positive chamber
\[ C = \{ \theta \in D_\delta \mid \theta \cdot \alpha_i > 0, ~ i = 1, \ldots, r\}\]
and given $\mathbf{v}$ let $C(\mathbf{v})$ be the unique chamber of the decomposition of $\R^I$ with respect to the roots in $R_+(\mathbf{v})$ such that $C(\mathbf{v})$ has face $C$ and such that $\theta \cdot \delta > 0$ for $\theta \in C(\mathbf{v})$. Let $\theta_{\text{Hilb}}(\mathbf{v})$ denote a fixed stability condition in this chamber, and $\theta_U$ a fixed stability condition in $C$.
\paragraph{Moduli of torsion free sheaves}
First let $\overline{X_\Gamma}$ denote the orbifold compactification of $X_\Gamma$ from \cite{nakajima2007sheaves}, i.e. we have an action of  $\Gamma$ on $\mathbb{P}^2$ and we resolve the singularity at $[0:0:1]$. Let $\ell_\infty$ be the divisor at infinity. Then
\begin{thm}[Nakajima \cite{nakajima2007sheaves}]
  \label{thm:quivers_are_torsionfree_moduli}
  There is a correspondence between stable quiver representations and framed torsion free sheaves such that quiver variety $\mf{M}_{\theta_{\text{Hilb}}(\mathbf{v})}(\mathbf{v}, \mathbf{w})$ is isomorphic to the moduli space of framed torsion free sheave $(E, \Phi)$. Further,
  \begin{enumerate}[(i)]
    \item Let $\mathbf{u} = \mathbf{w} - C\mathbf{v}$. The chern classes of $E$ are given by
    \begin{align*}
      \on{c}_1(E) &= \sum_{i \neq 0} u_i c_1(\mc{V}_i)\\
      \on{ch}_2(E) &= \sum_i u_i \on{ch}_2(\mc{V}_i) + 2\mathbf{v} \cdot \delta \on{ch}_2(\mc{O}(\ell_\infty))
    \end{align*}
    \item The framing $\Phi$ gives an isomorphism $E\simeq E_\infty$ over the end of $X_\Gamma$ such that the representation of $\Gamma$ at $\infty$ on $E_\infty$ decomposes as $\oplus \rho_i^{\oplus w_i}$ where $\rho_i$ is the irreducible representation of $\Gamma$ corresponding to vertex $i\in I$ under the McKay correspondence.
    \item The variety $\mf{M}_{\theta_{U}}(\mathbf{v}, \mathbf{w})$ is isomorphic to the Uhlenbeck compactification of $\mf{M}^{reg}_{\theta_{\text{Hilb}}(\mathbf{v})}(\mathbf{v}, \mathbf{w})$, so we have an identification
    \begin{equation}\label{eq:uhl}
      \mf{M}_{\theta_{\text{U}}}(\mathbf{v}, \mathbf{w}) =
      \bigsqcup_{k \ge 0}  \mf{M}^{reg}_{\theta_{U}}(\mathbf{v}- k\delta, \mathbf{w})\times S^{k}X_{\Gamma}
    \end{equation}
    such that the natural map
    \begin{equation}\label{eq:geis_to_uhl}
            \mf{M}_{\theta_{\text{Hilb}}(\mathbf{v})}(\mathbf{v}, \mathbf{w}) \to \mf{M}_{\theta_{\text{U}}}(\mathbf{v}, \mathbf{w})
    \end{equation}
    coincides with the Gieseker-Uhlenbeck map of \cite{li1993algebraic}.
  \end{enumerate}
\end{thm}
In particular $\mf{M}_{\theta_{Hilb}(\mathbf{n\delta})}(n\delta, \mathbf{w_0})$ is isomorphic to the Hilbert scheme of $n$ points on the corresponding smooth ADE surface.

\paragraph{Equivariant Hilbert Scheme}
There is another well studied stability condition on the affine ADE quivers, especially useful for geometric actions of various algebras and quantum groups. This corresponds to the choice $\theta =: \theta_+$ a stability condition in the positive chamber $C_+$ where $\theta_i > 0$ for all $i = 0, \ldots, n$. Define the analogous negative chamber $C_- = -C_+$.

Then we have the following modular description of this quiver variety if we identify framing/dimension vectors as representations of the finite group $\Gamma$ using the McKay correspondence.
\begin{thm}[\cite{nakajima1999lectures}]
There is an isomorphism $\mf{M}_{\theta_+}(\mathbf{v}, \mathbf{w}) \simeq M(\mathbf{v}, \mathbf{w})$ where $M(\mathbf{v}, \mathbf{w})$ is the moduli space of framed torsion free sheaves $(E, \phi)$ on $\mathbb{P}^2$ with $H^1(\mathbb{P}^2, E(-1)) = \mathbf{v}$ as a representation of $\Gamma$ and the induced action on $E|_{\ell_\infty}$ under $\phi: E|_{\ell_\infty} \simeq \mc{O}_{\ell_\infty}^{\oplus |\mathbf{w}|}$ corresponds to the representation $\mathbf{w}$.
\end{thm}
Consider the case $\mathbf{w} = \mathbf{w_0}$ of the trivial representation at $\ell_\infty$. Here the map to the affine quotient
\[ \mf{M}_{\theta_+}(\mathbf{v}, \mathbf{w_0})\to \mf{M}_{0}(\mathbf{v}, \mathbf{w_0})
\]
is an analogue of the Hilbert-Chow map (or Gieseker-Uhlenbeck in the higher rank case).

It is helpful to consider all of the affine quotients together in a stratified space
\[\mf{M}_0(\infty, \mathbf{w_0}) := \bigcup_{n\ge 0} Sym^n(\C^2)^\Gamma = \bigcup_{n \ge 0} Sym^n(\C^2/ \Gamma)\]

of symmetric powers of the affine quotient, where we include $Sym^n(\C^2/ \Gamma)\hookrightarrow
Sym^{n+k}(\C^2/ \Gamma)$ by adding $k>0$ points at the origin. The image of one
$\mf{M}_{\theta_+}(\mathbf{w}, \mathbf{v})$ surjects onto a maximal stratum. One way of finding
this stratum is by noting that the affine quotients agree and it is the corresponding affine
quotient for the quiver variety with the same dimension vector but the stability condition
$\theta_{Hilb}(\mathbf{v})$. Then the number of points is the length of the double dual of an
element of $\mf{M}_{\theta_{Hilb}(\mathbf{v})}(\mathbf{v}, \mathbf{w_0})$, namely $ n =  v_0 - v^T C v/2$ where $C$ is the Cartan matrix.

\paragraph{Stratifications}

There is an essential structure on quiver varieties induced from their definition as GIT quotients, which is a stratification by the conjugacy class of the stabilizer at a point see \cite[\S 6]{nakajima1994instantons}\cite[\S 3]{nakajima1998quiver} \cite[\S 2.6]{nakajima2009quiver}. This is a more general description of the stratification \eqref{eq:uhl} of the Uhlenbeck moduli space. Let $\hat{G}$ run over conjugacy classes of subgroups of $G_{\mathbf{v}}$ then
\begin{equation}\label{eq:strata_quiver}
  \quivb{\theta}{v}{w} = \bigsqcup_{\hat{G}} \quivb{\theta}{v}{w}_{\hat{G}}
\end{equation}
where $\quivb{\theta}{v}{w}_{\hat{G}} $ consists of points represented by $(B,i,j)$ with stabilizer
$\hat{G}$. The locus with trivial stabilizer, or $\hat{G} = 1$ is $\quivreg{\theta}{\mb{v}}{\mb{w}}$.
Possible choices of $\hat{G}$ are all isomorphic to $\prod_{j = 1}^{k}\GL(n_j)$ for some choices of
$n_j$,  and a polystable representative $(B,i,j)$ of a point in the stratum corresponding to a
quiver representation in $V$ decomposes as
\begin{equation}\label{eq:statum_decomp}
  \begin{split}
V &\simeq V^\infty \oplus (V^1)^{\oplus n_1} \oplus \cdots \oplus (V^k)^{\oplus n_k}\\
(B, i, j) & \simeq (B^\infty,  i^\infty, j^\infty) \oplus (B^1, 0 , 0)^{\oplus n_1} \oplus \cdots \oplus (B^k, 0, 0)^{\oplus n_k }
\end{split}
\end{equation}
where the representations $(V^i, (B^i, 0,0))$ are simple and pairwise non-isomorphic. Under the equivalence \eqref{eq:cb_trick_rep_equiv}, the summand $(B^\infty, i^\infty, j^\infty)$ corresponds to the unique summand as a representation of $Q_\infty$ which contains the one dimensional subspace at the infinite vertex.

 Let $\C^{2}_{\circ}/\Gamma = (\C^2\backslash\{0\})/\Gamma$.  When $\theta = 0$ for an affine ADE  quiver this stratification is the decomposition
\[\quivb{0}{v}{w} = \bigsqcup_{\substack{\mb{v'} \le \mb{v}\\ |\mb{v'}| + |\lambda| + k = |\mb{v}|}} \mf{M}^{reg}_0(\mb{v'}, \mb{w})\times S_\lambda (\C^{2}_{\circ}/\Gamma )\times \{ k [0]\} \]
where $k\ge 0$,  $\mf{M}^{reg}_0(\mb{v'}, \mb{w})\subset \quivb{0}{v'}{w}$ consists of framed equivariant \emph{locally free} sheaves on $\PP^2$ (which may be empty) and
\[S_\lambda(Y) = \{ \sum \lambda_i [x_i] \in Sym^n Y \mid \lambda \vdash n, ~ x_i \neq x_j \text{ if } i \neq j\}.  \]
When $\mb{w} = \mb{w_0}$ the only trivially framed line bundle on $\PP^2$ is the trivial line bundle so the only non-empty $\mf{M}^{reg}_0(\mb{v'}, \mb{w})\subset \quivb{0}{v'}{w}$ occurs when $\mb{v'} = 0$, and in this case it is a point. Thus the decomposition becomes
\begin{equation}\label{eq:pts_quiv_stratification}
  \quivb{0}{v}{w_0} = \bigsqcup_{\substack{k \ge 0\\ |\lambda| + k = |\mb{v}|}} S_\lambda (\C^{2}_{\circ}/\Gamma )\times \{ k [0]\}.
\end{equation}
which also is a stratification of $\mf{M}_0(\infty, \mathbf{w_0})$ where we identify strata by identifying factors $\{k[0]\}$ and $\{k'[0]\}$ for $k \neq k'$.

\paragraph{Local description of resolution} We recall a local analytic description of the
resolution or partial resolution \eqref{eq:semisimplification} from variation of GIT parameters following \cite[\S 6]{nakajima1994instantons}, \cite{crawley2003normality} near a point $x$ in terms of the the map to the affine quiver variety near the central point for a different quiver and different data. Our notation follows \cite{bellamy2018birational}.

Let $\theta$ be a generic stability parameter in a chamber $\mf C$ and let $\theta'$ be a stability parameter in $\overline{\mf C}$, and let $\pi_{\theta, \theta'} : \quivb{\theta}{v}{w} \to \quivb{\theta'}{v}{w}$ be the corresponding map. Let $C$ be the Cartan matrix for the quiver $Q_\infty$, with vertex set $I_\infty$ and let $(-,- )$ be the pairing defined by this matrix, and let $p(\alpha) := 1  - 1/2(\alpha, \alpha)$.

\begin{defn}\label{defn:ext_quiver_quiv}
Let $(B,i,j)$ be a $\theta_0$-polystable representative of a point $x$ in $\quivb{\theta'}{v}{w}$ with decomposition, $V^\infty, V^i$ and $n_i$ for $i = 1, \ldots, k$ defined as in \eqref{eq:statum_decomp}. Let $\beta^i\in \Z_{\ge 0}^{I_\infty}$ be the dimension vectors of $V^i$
for $i \in \{\infty, 1, \ldots, k\}$. Let $Q'$ be the quiver with
\begin{itemize}
  \item vertices $\{1, \ldots, k\}$,
  \item  $p(\beta^i)$ loops at vertex $i$,
  \item  and $-(\beta^i, \beta^j)$ edges between vertices $i$ and $j$.
  \end{itemize} Then $Q'$ is called the \emph{Ext-quiver} of $(B,i,j)$ or of $x$.
\end{defn}
Now choose dimension vector $\mb{n} = (n_1, \ldots, n_k)$
and framing vector
$ \mb{m} = (m_1, \ldots, m_k)$
with $n_i$ as above and where $m_i = -(\beta^\infty, \beta^i)$ for $i \in \{1, \ldots, k\}$. The stability parameter $\theta'$ will correspond to the 0 stability parameter for this quiver. Next we find stability parameter to correspond to $\theta$.  Pick a real stability parameter $\rho$ corresponding to a rational character of $G_{\mb{n}}$ (which is $\hat{G}$ for the stratum containing $(B,i,j)$) by restricting the character $\theta$ to $G_{\mb{n}}$, namely
\begin{align*}
  \rho &= \res_{G_{\mb{n}}}^{G_{\mb{v}}} \theta \\
   \rho(\gamma) &=  \theta\big( \sum_{i = 1}^k \gamma_i \beta^i \big).
\end{align*}
\begin{thm}[Nakajima \cite{nakajima1994instantons} Crawley-Boevey \cite{crawley2003normality}]\label{thm:local_quiver}
Consider a point $x \in \quivb{\theta'}{v}{w}$ with Ext-quiver and data $\beta^i, \mb{n}, \mb{m}, \rho$ as above. Let $\ell = p(\beta^\infty)\ge 0$.  There are local analytic neighborhoods $U$ of $x\in \quivb{\theta'}{v}{w}$ and $V$ of $0\times 0 \in \quivb{0}{n}{m}\times \C^\ell$ and isomorphisms fitting into a commutative diagram
\begin{figure}[h!]
  \centering
\begin{tikzcd}
 \pi_{\theta, \theta'}^{-1}(U)\arrow[r, "\sim"]\arrow[d, "\pi_{\theta, \theta'}"]& (\pi_{\rho, 0} \times \id)^{-1}(V)\arrow[d, "\pi_{\rho, 0}\times \id"] \\
 U \arrow[r,  "\sim"]& V
 \end{tikzcd}
\end{figure}

such that in particular, the fibers $\quivb{\theta}{v}{w}_x$ over $x$ and $\quivb{\rho}{n}{m}_0\times \{0\}$ over $0\times 0$ are identified.
\end{thm}
\begin{rmk}
  \begin{enumerate}[(i)]
\item The value $\ell = p(\beta^\infty)$ may equivalently be thought of as the number of loops at the framing vertex. 
\item   This result was extended in \cite{bellamy2018birational} to non-generic 
$\theta$, i.e. for any map of the form \eqref{eq:semisimplification}.
  \end{enumerate}
\end{rmk}
\paragraph{Other stability conditions} In addition to those specific stability conditions already described, there are additional stability conditions which are important to consider.  We fix notation for these. A key tool in understanding the birational geometry of quiver varieties is the fact that the polyhedral structure of the stability space captures partial symplectic resolutions of the corresponding quiver variety as in Theorem \ref{thm:partial_resolutions}.

Fix an affine ADE graph $\Gamma$ and let $I_0$ be a set of indices other than the one corresponding to the affine root. Recall the positive chamber $C_+ = \{ \theta \mid \theta_i > 0, ~ i  = 0, \ldots, r\}$. Let $\theta_{I_0, +}$ be a stability condition generic in the face $C_{I_0, +} \subset \overline{C_+}$ defined by
\begin{equation}\label{eq:theta_io_plus}
  \theta_{I_0, +}\in C_{I_0, +} := \{ \theta \mid \theta_i = 0 \text{ if }  i \in I_0, \theta_j > 0 \text{ if } j\not\in I_0\}.
\end{equation}
This is a non-generic stability parameter but we take it to be generic within its face. Therefore we have a projective resolution of symplectic singularities specializing \eqref{eq:semisimplification} for any pair $I_1 \subset I_0 \subset  \Gamma$
\[\pi_{I_1, I_0} : \quivb{\theta_{I_0',+}}{v}{w} \to \quivb{\theta_{I_0, +}}{v}{w}\]
which are compatible in the sense that for a chain $I_2 \subset I_1 \subset I_0 \subset \Gamma$ we have
\[\pi_{I_2, I_0} = \pi_{I_1, I_0}\circ \pi_{I_2, I_1}.\]
These maps are also compatible with the map to the affine quotient. This gives a chain of partial symplectic resolutions which when $\mathbf{w} = \mathbf{w_0}$ all resolve to the unique resolution $\quivb{\theta_+}{v}{w_0}$.

\subsection{Universal enveloping algebra action}\label{sec:quiver_action}
There is a geometric description of the universal enveloping algebra of an affine Kac-Moody algebra
via Steinberg correspondences on the quiver varieties with stability parameter $\theta_+$. We
combine the quiver varieties with one framing vector and different dimension vectors
\[\mf{M}_{\theta_+}(\mathbf{w}) := \bigsqcup_{\mathbf{v}} \mf{M}_{\theta_+}(\mathbf{v},
\mathbf{w}).\]
We again restrict to the case $\mathbf{w_0}$ and consider
\[Z := \mf{M}_{\theta_+}(\mathbf{w_0}) \times_{\mf{M}_{0}(\infty, \mathbf{w_0})} \mf{M}_{\theta_+}(\mathbf{w_0})\]
which is an analogue of the Steinberg variety for the Springer resolution. We now state the result on the action of an affine Lie algebra in the specific case of the equivariant hilbert scheme. Given $x\in
\mf{M}_0(\mathbf{v}, \mathbf{w_0})$ consider the fiber
$\mf{M}_{\theta_+}(\mathbf{v}, \mathbf{w_0})_x$ over the point $x$ of the map
$\mf{M}_{\theta_+}(\mathbf{v}, \mathbf{w_0})\to \mf{M}_0(\mathbf{v}, \mathbf{w_0})$. When $x  =
0$ we denote the fiber $\mf{L}(\mathbf{v}, \mathbf{w_0})$. Extending our previous notation let
$\mf{M}(\mathbf{w_0})_x = \sqcup_{\mathbf{v}} \quivb{\theta_+}{v}{w_0}_x$.
\begin{thm}[Nakajima \cite{nakajima1994instantons}]\label{thm:uea_action_nakajima}
There is an algebra morphism
\[
 U(\hat{\mf g}) \to H_{top}^{BM}(Z, \C)\]
where $\hat{\mf g}$ is the affine lie algebra corresponding to $\Gamma$, $H_{top}^{BM}(Z)$ is top dimensional Borel-Moore homology of $Z$ with the algebra action given by convolution.
\begin{enumerate}[(i)]
  \item The images of $h_i, d, c$ are multiplies of diagonal subvarieties.
  \item  Restricted to a pair of quiver varieties, the image of $e_i$ is the fundamental class of the Hecke correspondence
\begin{equation}\label{eq:hecke}
\{ (E, E') \in\quivb{\theta_+}{v}{w_0}\times\quiv{\theta_+}{\mathbf{v} + \rho_i}{\mathbf{w_0}} \mid E \subset E'\}
\end{equation}
and the image of $f_i$ is, up to sign, the same class in the opposite direction. Each connected component is a nonsingular Lagrangian subvariety of $\quivb{\theta_+}{v}{w_0}\times\quiv{\theta_+}{\mathbf{v} + \rho_i}{\mathbf{w_0}}$.
\item
Under the convolution action on $H_{top}(\mf{M}(\mathbf{w_0})_x)$, this space is an irreducible integrable highest weight representation with weight spaces $H_{top}(\mf{M}_{\theta_+}(\mathbf{v}, \mathbf{w_0})_x)$. When $ x = 0$ the highest weight is $\Lambda_0$.
\item The representation only depends on the stratum in which $x$ lies in the decomposition \eqref{eq:pts_quiv_stratification}.
\end{enumerate}
\end{thm}

We denote the component of the Hecke correspondence between $\quivb{\theta_+}{v}{w_0}$ and $\quiv{\theta_+}{\mathbf{v} + \rho_i}{\mathbf{w_0}} $ by $\mf{P}_i(\mb{v})$, and more generally the one between $\quivb{\theta_+}{v}{w}$ and $\quiv{\theta_+}{\mathbf{v} + \rho_i}{\mathbf{w}} $ by $\mf{P}_i(\mb{v}; \mb{w})$.

\paragraph{Equivalent descriptions of Hecke correspondence}

We collect a few description of the Hecke correspondences.

First calculate how it acts on fibers $H_{top}(\mf{M}(\mathbf{w_0})_x)$ for $x \neq 0$ under the equivalence of Theorem \ref{thm:local_quiver}. If the variety is positive dimensional, the Ext quiver is the original quiver and some number of copies of the Jordan quiver (i.e. with one vertex and on loop). The map $\pi_{\theta_+, 0}$ near $x$ corresponds to the product of $\pi_{\theta_+, 0}$ on the first factor and the Hilbert-Chow map on the Jordan quiver factors. For the Jordan quiver the quiver varieties are $\Hilb^n(\C^2)$ for non-zero stability parameter and $Sym^n(\C^2)$ for zero stability parameter. Let $x$ lie in the stratum of \eqref{eq:pts_quiv_stratification} corresponding to the partition $\lambda$.  Thus we have
\begin{equation}\label{eq:factor_fiber}
   \mf{M}(\mb{v}, \mathbf{w_0})_x \simeq \mf{L}(\mb{v} - |\lambda| \delta, \mb{w_0}) \times \prod_{i} \Hilb_0^{\lambda_i}(\C^2)
\end{equation}
where $\Hilb_0^{\lambda_i}(\C^2)$ is the punctural Hilbert scheme (i.e. the fiber over $k\cdot[0]$ of the Hilbert-Chow map), which is irreducible.

We need the following description, which is used in \cite{nakajima2009quiver}.
 Pick a stability condition $\theta_{\alpha_i}$ which is on a maximal dimensional face of the positive cone $\C_+$ and also on the wall $\alpha_i^\perp$, and generic in this face. Let $a_i$ be a simple root of the affine ADE root system and let $S_i$ be simple quiver representation which is $1$ dimensional and supported in degree $i$, on the $i$th vertex. Form the ind-variety
 \[\quiv{\theta_{\alpha_i}}{\mb{v} + \infty \rho_i}{ \mb{w}} = \bigcup_{k = 0}^\infty \quiv{\theta_{\alpha_i}}{\mb{v} + k \rho_i}{ \mb{w}}\]
 where we include $\quiv{\theta_{\alpha_i}}{\mb{v} + k\rho_i }{ \mb{w}}\hookrightarrow \quiv{\theta_{\alpha_i}}{\mb{v} + (k+n) \rho_i}{ \mb{w}}$ for $n > 0$ by sending the representation $V$ to $V \oplus S_i^{\oplus n}$. For convenience in stating the following, we consider quiver varieties for non-generic stability parameter as parametrizing polystable quiver representations. 
\begin{prop}[Lemma 5.12 \cite{nakajima1998quiver}]
  \label{prop:hecke_from_1_wall}In this situation we have that the Hecke correspondence
  $\mf{P}_i(\mb{v}; \mb{w})$ is an irreducible component of the fiber product
   \[\quiv{\theta_+}{\mb{v}}{ \mb{w}}\times_{\quiv{\theta_{\alpha_i}}{\mb{v} + \infty \rho_i}{ \mb{w}}}\quiv{\theta_+}{\mb{v} +  \rho_i}{ \mb{w}}\]
where the fiber product comes from the maps $\pi_{\theta_+, \theta_{\alpha_i}}$ from  \eqref{eq:semisimplification} composed with the inclusion into $\quiv{\theta_{\alpha_i}}{\mb{v} + \infty \rho_i}{ \mb{w}}$. More specifically, it is the unique irreducible component such that 
\begin{itemize}
\item For a point $x\in \quiv{\theta_+}{\mb{v}}{ \mb{w}}$ such that $\pi_{\theta_+, \theta_{\alpha_i}}(x)$ contains a direct summand of exactly $r$ copies of $S_i$, the fiber of $\mf{P}_i(\mb v, \mb w)$ over $x$ under the first projection is a projective space of dimension $\rho_i^T (\mb w - C\mb v) + r -1$. 
\item For a point $y\in \quiv{\theta_+}{\mb{v} + \rho_i }{ \mb{w}}$ such that $\pi_{\theta_+, \theta_{\alpha_i}}(y)$ contains a direct summand of exactly $r$ copies of $S_i$, the fiber of $\mf{P}_i(\mb v, \mb w)$ over $y$ under the second projection is a projective space of dimension $r -1$. 
\end{itemize}
\end{prop}

It will also be useful to know how the Hecke correspondence restricts to fibers of the map $\pi_{\theta_+,0}$.

\begin{prop}\label{prop:action_on_x_fiber}
  The restriction of $\mf{P}_i(\mb{v})$ to $\mf{M}(\mb{v}, \mathbf{w_0})_x\times \mf{M}(\mb{v}+ \rho_i, \mathbf{w_0})_x$ under the decomposition \eqref{eq:factor_fiber} is
  \[ \{ (E, \xi, E', \xi') \in \mf{M}(\mb{v}, \mathbf{w_0})_x\times \mf{M}(\mb{v}+ \rho_i, \mathbf{w_0})_x \mid (E, E') \in \mf{P}_i(\mb{v} - |\lambda| \delta), \zeta = \zeta' \} \]
  or in other words
  \[
  \mf{P}_i(\mb{v} - |\lambda| \delta)\times \Delta \subset \mf{L}(\mb{v} - |\lambda| \delta, \mb{w_0})\times \mf{L}(\mb{v} - |\lambda| \delta, \mb{w_0}) \times\prod_{i} \Hilb_0^{\lambda_i}(\C^2)\times \prod_{i} \Hilb_0^{\lambda_i}(\C^2).
  \]
\end{prop}
\begin{proof}
  Suppose $Z, Z'$ are $\Gamma$-fixed zero dimensional subschemes of $\C^2$ with $\on{len}(Z) +1 = \on{lem}(Z')$. Then the cycles must differ only at the origin, by exactly 1. Therefore the Hecke correspondence is trivial on the punctural Hilbert scheme factors, and the computation is reduced to the central fiber case, which is exactly what needed to be shown. Note that in the second line of the proposition we have denoted the Hecke correspondence and its restriction to the product of the central fibers by the same letter.
\end{proof}

\subsection{Compatibility of Lie algebra actions}
Fix an affine quiver $Q$ corresponding to a finite group $\Gamma$ and smooth surface $X_\Gamma$.
We know the Frenkel-Kac construction described in Section \ref{sec:VOA} gives a representation of the corresponding affine lie algebra $\widehat{\mf{g}}$ on the Fock space $V_{H^2(X_\Gamma)}$ from \eqref{eq:VOA} modeled on the middle-dimensional cohomology of the surface $X_\Gamma$. This is an irreducible highest weight representation. It also gives an irreducible highest weight representation of a larger algebra
\[\widehat{\mf{g}}_{H^*} := \Heis(H^4(X_\Gamma)) \oplus \hat{\mf{g}}\]
 on the analogue of \eqref{eq:rk_1_torsion_free_V} for our present context, namely
 \[V := \bigoplus_{c_1, \mathrm{ch}_2} H^*(M(c_1, ch_2))\]
 where $M(c_1, \mathrm{ch}_2)$ is the moduli space of rank 1 torsion free sheaves on $X_\Gamma$ with these characteristic classes.  When $\Gamma = \Z/n\Z$ these algebras are $\hat{\mf{sl}}_{n-1}$ and $\hat{\mf{gl}}_{n-1}$. The algebra $\widehat{\mf{g}}_{H^*}$ contains the entire Heisenberg algebra $\mc{H}eis(H^*(X_\Gamma))$ and for fixed $c_1$ each $\oplus_{ch_2}  H^*(M(c_1, ch_2))$ is the usual level 1 Fock module for this algebra. We have identifications
 \begin{align*}
   V_{H^2(X_\Gamma)} &= \bigoplus_{\mathbf{v}}H^{mid}(\mf{M}_{\theta_{Hilb}(\mathbf{v})}(\mathbf{v}, \mathbf{w_0}))\\
   V &= \bigoplus_{\mathbf{v}}H^*(\mf{M}_{\theta_{Hilb}(\mathbf{v})}(\mathbf{v}, \mathbf{w_0}))
 \end{align*}
 between the Fock spaces and cohomology groups of spaces of rank 1 torsion free sheaves, which are Nakajima quiver varieties for a specific choice of stability parameter.

The action of Theorem \ref{thm:uea_action_nakajima} of the universal enveloping algebra of $\widehat{\mf{g}}$ on cohomologies is for quiver varieties for a different stability parameter, namely on the space
\[ \bigoplus_{\mathbf{v}}H^{*}(\mf{M}_{\theta_+}(\mathbf{v}, \mathbf{w_0}))\]
preserving the subspace
\[ \bigoplus_{\mathbf{v}}H^{mid}(\mf{M}_{\theta_+}(\mathbf{v}, \mathbf{w_0}))\] making the latter an irreducible highest weight representation. Here $\theta_+$ lies in the positive chamber for the corresponding affine Weyl group, which is separated by walls from the Hilbert scheme chamber for each $\mb{v}$, and the number of walls tends to infinity.  However by the flop relating the quiver varieties in the different chambers, there is a $\C^*$-equivariant diffeomorphism between the corresponding quiver varieties so in particular the corresponding vector spaces are isomorphic. It is a natural question to ask how these two lie algebra actions relate to one another under this isomorphism. By an detailed analysis of the fixed points of a torus action in the $A_n$ case, they were shown to coincide in this case.
\begin{thm}[Nagao \cite{nagao2009quiver}]\label{thm:nagao}
Consider the $\C^*$ equivariant diffeomorphism corresponding to a flop $F$ between quiver varieties corresponding to the $\widetilde{A}_{n}$ with stability parameters $\theta_+$ and $\theta_{Hilb}(\mathbf{v})$. Denoting the isomorphism on cohomology by the same letter $F$ we have that
\[ \bigoplus_{\mathbf{v}}H^{mid}(\mf{M}_{\theta_+}(\mathbf{v}, \mathbf{w_0})) \xrightarrow{F} H^{mid}(\mf{M}_{\theta_{Hilb}(\mathbf{v})}(\mathbf{v}, \mathbf{w_0}))\]
intertwines the Nakajima action with the Frenkel-Kac action for a specific choice of cocycle determining the latter.
\end{thm}
Since we are interested in all cohomology groups we need
\begin{prop}\label{prop:flop_intertwines} Under the same flop
  \begin{equation}\label{eq:flop_hilb_an} \bigoplus_{\mathbf{v}}H^{*}(\mf{M}_{\theta_+}(\mathbf{v}, \mathbf{w_0})) \xrightarrow{F} H^{*}(\mf{M}_{\theta_{Hilb}(\mathbf{v})}(\mathbf{v}, \mathbf{w_0}))\end{equation}
  the $\hat{\mf{sl}}_{n-1}$ actions are also intertwined.
\end{prop}
\begin{proof}
  On the right hand side of \eqref{eq:flop_hilb_an}, the action of $\widehat{\mathfrak{sl}}_{n-1}$ on $V$ is induced by its action on $V_{H^2(X_\Gamma)}$ and the isomorphism
    $V =V_{H^2(X_\Gamma)} \otimes \mc{F}$ where $\mc{F}$ is the Fock module for the rank 1 free boson, modeled on the line spanned by the class of a point. This $\mc F$ has a basis labeled by all partitions $\lambda$ and we have
    \[
    V = \bigoplus_{\lambda} V_{H^2(X_\Gamma)}\otimes |\lambda_1, \ldots, \lambda_n\rangle.
    \]
    We need a geometric description of this decomposition. Let $\pi_{\theta_{Hilb}(\mb{v}), 0}$ be the semisimplification map and we decompose based on the inclusion of a fiber of this map into $\bigsqcup_{\mb v} \mf{M}_{\theta_{Hilb}(\mathbf{v})}(\mathbf{v}, \mathbf{w_0})$. Namely we have
    \begin{equation}\label{eq:hilb_decomp_strata}
V = \bigoplus_{\lambda} H_{top}(\mf{M}_{\theta_{Hilb}(\mb{v})}(\mb{w_0})_y)
    \end{equation}
    where $y\in \mc{O}_\lambda$ and $\mc{O}_\lambda$  runs over the strata \eqref{eq:pts_quiv_stratification} so $\lambda$ is a partition for any integer and the empty partition corresponds to the central point, and again we combine fibers for different $\mb{v}$ and define $\mf{M}_{\theta}(\mb{w_0})_x$ for any stability condition $\theta$. Let $\mf{L}_\theta(\mb{w_0})$ be the central fiber. The isomoprhsim $V_{H^2(X_\Gamma)}\otimes |\lambda_1, \ldots, \lambda_n\rangle \simeq  H_{top}(\mf{M}_{\theta_{Hilb}(\mb{v})}(\mb{w_0})_y)$ arises by noting that a basis for the left hand side coincides with a basis for the right hand side in terms of the Nakajima basis.

    The action on the other side of \eqref{eq:flop_hilb_an} is induced by
    a similar factorization\cite[\S ~5.2]{nakajima2002geometric}. Recall that the central fiber of the map to the affine quotient $\mf{L}(\mathbf{v}, \mathbf{w_0})$ is a Lagrangian subvariety homotopic to $\mf{M}_{\theta_+}(\mathbf{v},
    \mathbf{w_0})$ so $H_{*}(\mf{M}_{\theta'}(\mathbf{v}, \mathbf{w_0})) \simeq H_{*}(\mf{L}(\mathbf{v},
    \mathbf{w_0}))$ and in particular
    $$H_{mid}(\mf{M}_{\theta_+}(\mathbf{v}, \mathbf{w_0})) \simeq H_{top}(\mf{L}(\mathbf{v},
    \mathbf{w_0})).$$
  The lower degree terms may be found via the decomposition theorem, which is particularly nice in this case.  In particular, the map is semismall so no shifts appear, and it has also been shown that no non-trivial local systems appear \cite{nakajima2001quiver_finite}. Thus we have a decomposition
  \[
      H_{top - d} (\mf{L}_{\theta_+}(\mb{w_0}), \C) = \bigoplus_{\lambda} H^d (i_0^! IC(\mc{O}_\lambda)) \otimes H_{top}(\mf{M}_{\theta_+}(\mb{w_0})_y, \C)
    \]
       Then $i_0: \{0\} \to \quiv{0}{\infty}{\mb{w_0}}$ is the inclusion. This is a decomposition as $\hat{\mf {sl}}_{n-1}$ modules, where the action is trivial on the cohomology of the IC complex. In this case the strata are $S_\lambda(\C^2_\circ / \Gamma)$  and the $IC$ complexes are constant sheaves $\C_{\overline{\mc{O}_y}}[\dim]$ because $S_\lambda(\C^2_\circ / \Gamma)$ have only finite quotient singularities. Thus our decomposition becomes
    \[
        H_* (\mf{L}_{\theta_+}(\mb{w_0}), \C) = \bigoplus_{\lambda} H_{top}(\mf{M}_{\theta_+}(\mb{w_0})_y, \C)
    \]
    The flop $F$ induces by Theorem \ref{thm:local_quiver} the flop $F\times \Delta$ between neighborhoods of the fibers $\mf{M}_{\theta_+}(\mb{w_0})_y$ and $\mf{M}_{\theta_{Hilb}(\mb{v})}(\mb{w_0})_y$, where $F$ acts on the components which quiver varieties for affine quivers and $\Delta$ is the diagonal on Jordan quiver components. Then Proposition \ref{prop:action_on_x_fiber} implies that the action of the Hecke correspondences coincides with the action on the first component of the product in the fibers. Thus  this flop intertwines the $\hat{\mf{sl}}_{n-1}$ actions as in Theorem \ref{thm:nagao} for each chosen point $y$.  This decomposition for the $\theta_+$ parameter coincides with the one from \eqref{eq:hilb_decomp_strata}. Thus the actions are intertwined by $F$ on all degrees of cohomology.

\end{proof}
We will not explicitly record a proof or statement of the analogous result for types D and E but the proof is essentially contained in the proof of Theorem \ref{thm:main}. 

\section{Birational geometry of hyperk\"ahler varieties}\label{sec:hk_birational}
We review structures on $\NS(M)$ for a hyperk\"ahler variety $M$ and their relation to birational geometry, developed in numerous sources. See \cite{hassett2009moving} and references therein.
\begin{defn}
  Let $M$ by a  hyperk\"ahler variety with Beauville-Bogomolov form $(, )$ on $H^2(M, \Z)$.
  \begin{itemize}
    \item The \emph{cone of curves} of $M$ is the cone $\NE_\R(M)$ generated by effective curves.
    \item The \emph{ample cone} is the cone  $\Amp(M)$ generated by ample classes.
    \item The \emph{nef cone} is the cone $\Nef(M)$ dual to $\overline{\NE(M)}$.
    \item The \emph{positive cone} of $M$ is the component
    \[\Pos(M) \subset \{ \alpha \in \NS_\R(M) \mid (\alpha,\alpha) > 0\}\]
    of the locus of positive-self-pairing classes such that it contains an ample class.
    \item The \emph{movable cone} is the cone $\Mov(M)$ generated by divisors $D$ such that there is an $N> 0$ with $ND$ having no fixed components (i.e. with fixed locus having codimension $\ge 2$).
  \end{itemize}
\end{defn}
We also have a number of interesting wall and chamber structures on these cones. The first comes from the following proposition
\begin{prop}[Markman \cite{markman2013prime}]\label{prop:markman_refl}
  Let $D$ be an exceptional divisor, and let $\rho_D$ be the corresponding reflection, which is an integral involution of $\NS(M)$. Let $W_{Exc}$ be the Weyl group of exceptional reflections. Then the cone $\Mov(M)\cap \Pos(M)$ of big moveable divisors is the fundamental chamber for $W_{Exc}$ on $\Pos(M)$.
\end{prop}
Also given a birational hyperk\"ahler $M'$ the map induced from graph $\Gamma \subset M\times M'$ on cohomology $\Gamma_*: H^2(M, \Z) \to H^2(M', \Z)$ is an isomorphism preserving $(,)$ and $\Gamma_*\alpha \in \Amp(M')$.
We can thus identify the ample cone of a birational model $M'$ of $M$ with a subset of the positive cone of $M$. Varieties whose ample classes are in the same orbit of $W_{Exc}$ are isomorphic, so we can consider the ample cone of every birational model as a subset of the movable cone.  If the \emph{birational ample cone} $\mc{B}\Amp(M)$ is the union of the images of all ample cones of birational models of $M$ under the maps induced on cohomology groups then
\begin{prop}[\cite{hassett2009moving}]
  We have inclusions
    \[\mc{B}\Amp(M) \subset \Mov(M)\subset \overline{\mc{B}\Amp}(M).\]
\end{prop}
It follows that $\overline{\Mov}(M) = \overline{\mc{B}\Amp}(M)$ and so  the chamber decomposition
\begin{equation}\label{eq:amp_decomp}
\bigcup_{M' \sim M} \overline{\Amp}(M') = \overline{\mc{B}\Amp}(M)
\end{equation}
is also a chamber decomposition for $\overline{\Mov}(M)$. It is locally polyhedral \cite{hassett2009moving}.

A similar structure is known for the relevant quiver varieties, in particular affine ADE quiver varieties for the framing vector $\mb w_0$. Consider a crepant resolution $X\xrightarrow{\pi} Y$ of an affine variety $Y$, let $N^1(X/Y) = \NS(X/Y)_{\Q}$ be the vector space of $\Q$-cartier divisors up to numerical equivalence, with cones $\Mov(X/Y)$ as above, $\Nef(X/Y)$ consisting of divisors $D$ with $D\cdot \ell \ge 0$ if $\ell$ is a curve contracted by $\pi$, and $\Amp(X/Y)$ the interior of the Nef cone. We will write $\Amp(X)$ for $\Amp(X/Y)$ when the base is clear.

\subsection{Birational geometry of $X_\Gamma^{[n]}/ Sym^n \C^2/\Gamma$}\label{ssec:bellamy_craw}
We now follow the description in \cite{bellamy2018birational}, which shows that all crepant resolution of $Sym^n \C^2/\Gamma$ are given by variation of stability for a quiver variety, and the cone structure on \[N^1(X_\Gamma^{[n]}/ Sym^n \C^2/\Gamma)\] coincides with the wall and chamber structure on the space of stability conditions. In particular, a particular description is provided of the walls $D_\theta$ for $\theta\in R_+(\mb{v})$ which actually induce birational contractions. 

 Fix framing vector $\mathbf{w} = \mathbf{w_0}$ and dimension vector $\mathbf{v} = n\delta$. let $F$ be the``quadrant"
\[ F = \{ \theta  \mid \theta\cdot \alpha_i \ge 0 \text{ for } i = 1, \ldots, r, \theta\cdot \delta \ge 0\} \] and let
\[L_F: F\to N^1(X_\Gamma^{[n]}/Sym^n(\C^2/\Gamma)))\]
 be defined by $\theta \mapsto \bigotimes \mc{R}_i^{\theta_i}$ sending a stability parameter to a line bundle on the Hilbert scheme on the resolved surface. This map corresponds to the map \eqref{eq:lmap} for K3 surfaces and the following theorem is the analogue of Theorem \ref{thm:bm_mmp}.
\begin{thm}[Bellamy, Craw \cite{bellamy2018birational}]\label{thm:birat_sym}
The map $\theta \mapsto L_F(\theta)$ identifies $F$ with the relative movable cone $\Mov(X_\Gamma^{[n]}/Sym^n(\C^2/\Gamma))$ such that
\begin{enumerate}[(i)]
  \item The Namikawa Weyl group $W = \langle \rho_{\delta}, \rho_{\alpha_1}, \ldots, \rho_{\alpha_r}\rangle$ generated by reflections through the roots of the finite root system and through the imaginary root $\delta$ acts on $\mf{h} \simeq H^2(X_\Gamma^{[n]}, \Q)$ with fundamental chamber $F$ and quiver varieties with parameters $\theta, \theta'$  in the same $W$-orbit are isomorphic.
  \item Let $\Delta_f$ denote the set of roots for the finite root system (i.e. they have no $\alpha_0$ component). The walls $W\subset \R^I$ which induce birational contractions are the hyperplanes
  \[ \{ \delta^\perp, (m\delta + \alpha)^\perp \mid 0 \le m < n, \alpha \in \Delta_f\}\]
  so for a stability parameter $\theta \in \R^I$, the condition of not lying on one of these hyperplanes is equivalent to the quiver variety $\quiv{\theta}{n\delta}{\mb{w_0}}$ being smooth.
  \item The image of a stability chamber $C$ under the map $L_F$ (with the wall and chamber structure given by the aforementioned hyperplanes) is exactly the ample cone $\Amp(\quiv{\theta}{n\delta}{\mathbf{w_0}})$ for the corresponding birational model.
  \item For a stability condition $\theta_0$ generic on one of the boundary walls $\delta^\perp$ or  $\alpha_i^\perp$, for $i  = 1, \ldots, r$  of the movable cone and $\theta$ in an adjacent chamber, the map $\pi_{\theta, \theta_0}$ contracts an irreducible exceptional divisor. All other walls (i.e. on the interior of $F$) correspond to flops.
\end{enumerate}
\end{thm}
\begin{cor}[Bellamy, Craw \cite{bellamy2018birational}]
   Every crepant resolution of $Sym^n(\C^2/\Gamma)$ is given by a quiver variety. And every partial symplectic resolution between a nonsingular quiver variety $\quivb{\theta}{v}{w_0}$ and $Sym^n(\C^2/\Gamma)$ is given by a map $\pi_{\theta', 0}$ \eqref{eq:semisimplification} for some $\theta'$.
 \end{cor}
\begin{eg}\label{eg:2ptsA2}
Consider $\quiv{\theta_{Hilb}}{2\delta}{ \mb{w_0}} \simeq (T^*\PP^1)^{[2]}$. Then the walls are the hyperplanes orthogonal to
\[\{(1,0),(0,1), (1,0), (1,2)\} \]
while the potential walls are hyperplanes orthogonal to elements of $R_+ (2\delta)$, namely the set
\[ \{(1,0), (0,1), (1,0), (1,2), (2,1), (1,1), (2,2)\}\]
in particular the hyperplane $(2,1)^\perp$ is not a wall for the GIT stability space for this quiver variety. One way to see this is as follows: suppose $\theta$ is a stability condition generically on the wall $(2,1)^\perp$ with a strictly polystable quiver representation $V$ with a stable summand $V'$ of graded dimension $(2,1)$. Then because there are no edges between the framing node and the non-affine root node, $V$ must also have a stable summand with dimension vector $(0,1)$ which contradicts semistability from \ref{prop:king_stability}if, say, $\theta\cdot (0,1) > 0$.
\end{eg}

\begin{figure}[h!]
  \centering
\includegraphics[width=.7\linewidth]{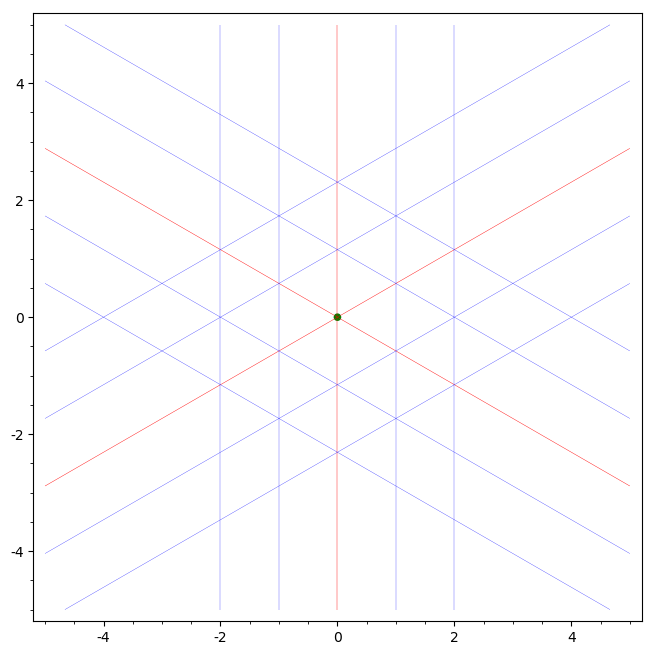}
\caption{\label{fig:3delta} Level 1 hyperplane for the stability space of the affine $A_2$ quiver for framing vector $\mb{w_0}$ and dimension vector  $\mb{v} = 3\delta$. The red lines correspond to the walls $e_\alpha^\perp$ for finite roots $e_{\alpha}$. }
\end{figure}

\subsection{Other dimension vectors}
The previous section completely describes the birational geometry of $\quiv{\theta}{n\delta}{\mb{w_0}}$ over $\quiv{0}{n\delta}{\mb{w_0}}$, which is provided by variation of GIT stability. On the other hand, for a different dimension vector $\mb{v}$ the result of Theorem \ref{thm:quivers_are_torsionfree_moduli} says $\quiv{\theta_{Hilb}(\mb{v})}{\mb{v}}{\mb{w_0}}$ over its image in $\quiv{0}{\mb{v}}{\mb{w_0}}$ is
isomorphic under \eqref{eq:hilb_eq_m} to $\quiv{\theta_{Hilb}(n\delta)}{n\delta}{\mb{w_0}}$ over $\quiv{0}{n\delta}{\mb{w_0}}$. Therefore variation of GIT stability for the dimension vector $n\delta$ controls the birational geometry over the affine quotient of quiver varieties with framing vector $\mb{w_0}$ and other dimension vectors. In particular, birational transformations given by variation of stability for quiver varieties $\quiv{\theta}{\mb{v}}{\mb{w_0}}$ must also be given under this isomorphism by variation of stability for the dimension vector $n\delta$.

However, it is not true that for some stability vector $\theta$ there is always an isomorphism between $\quivb{\theta}{v}{w_0}$ and $\quiv{\theta}{n\delta}{\mb{w_0}}$; there is a modification involved in the stability condition, which at a slice $\theta\cdot \delta = 1$ corresponds to a shift by $c_1(\mc{L})$ where $\mc{L}$ is the line bundle which one tensors with to obtain the specific isomorphism between the moduli space of rank 1 torsion free sheaves and the Hilbert scheme. 

The proof we give in probably not the most natural way to identify the stability spaces of $\quivb{\theta}{v}{w_0}$ and $\quiv{\theta}{n\delta}{\mb{w_0}}$, which would involve looking at where the determinant line bundle is sent, explicitly identifying a specific vector in the ample cones of $\quivb{\theta}{v}{w_0}$ and $\quiv{\theta}{n\delta}{\mb{w_0}}$. This proof highlights the strange interplay between variation of GIT stability for quiver representations and the birational geometry of the moduli spaces.

\begin{prop}\label{prop:affine_quiver_shift}
Let $Q$ be an affine ADE quiver, with fixed framing vector $\mb{w_0}$. For any dimension vector $\mb{v}$, let $\mb{u} = \mb{w} - C \mb{v}$ so that Theorem \ref{thm:quivers_are_torsionfree_moduli} implies that $\quivb{\theta_{Hilb}(\mb{v})}{v}{w_0}$ is isomorphic to the moduli space of rank 1 torsion-free subsheaves $E$ of a line bundle $\mc{L}_\mb{v} = E^{\vee \vee}$ with $c_1(E) = \sum_{i \neq 0} u_i c_1(\mc{V}_i) = c_1(\mc{L}_{\mb{v}})$ and quotient $ \mc{L}_{\mb{v}}/E$ of length $n = v_0 - \mb{v}^T C \mb{v}/2$. Consider the isomorphism 
\[\phi:  \quiv{\theta_{Hilb}(n\delta)}{n\delta}{\mb{w_0}} \to \quivb{\theta_{Hilb}(\mb{v})}{v}{w_0} \]
given on points representing torsion-free sheaves by 
\[\phi: E \mapsto E \otimes \mc{L}_{\mb{v}}.\]
Let $\Theta_{\mb{v}} \simeq \R^I$ denote the stability space for dimension vector $\mb{v}$.
\begin{enumerate}[(i)]
  \item If $n > 1$ then under the identifications 
  \[H^2(\quivb{\theta_{Hilb}(\mb{v})}{v}{w_0}, \R) = \Theta_{\mb{v}}\] and \[H^2(\quivb{\theta_{Hilb}(n\delta )}{n\delta }{w_0}, \R) = \Theta_{\mb{n\delta}}\] the isomorphism 
  \[H^2(\quivb{\theta_{Hilb}(n\delta )}{n\delta }{w_0}, \R) \simeq H^2(\quivb{\theta_{Hilb}(\mb{v})}{v}{w_0}, \R)\]
   induced by $\phi$ is (up to a global scaling) the unique map $\phi :  \Theta_{n\delta}\to  \Theta_{\mb{v}} $ which acts by the identity on the level $0$ hyperplane $H_0 := \{\theta\cdot \delta = 0 \}$ and sends $e_0 = (1, \ldots, 0)$ to
   \[e_0 + (-\sum_{i\neq 0} u_i, u_1, u_2, \ldots,u_n). \]
   In other words, $\phi$ preserves the level 1 hyperplane $H_1 := \{\delta \cdot \theta = 1\}$ and acts by a shift by $(u_1, \ldots, u_n)$ on this hyperplane. 
   \item If $n = 1$ the map $\phi$ induces an isomorphism $\phi: \Theta_{n\delta} \to \Theta_{\mb{v}}$ given by the same formula and the corresponding isomorphism between $H^2(\quivb{\theta_{Hilb}(n\delta )}{n\delta }{w_0}, \R) \simeq H^2(\quivb{\theta_{Hilb}(\mb{v})}{v}{w_0}, \R)$ is given (up to a scale) by the isomorphism $\phi: \Theta_{n\delta} / e_0 \to \Theta_{\mb{v}} / \phi(e_0).$ 
\end{enumerate}
\end{prop}

\begin{proof}
For $n> 1$, under the identification $\Theta_{v} = H^2(\quivb{\theta_{Hilb}(\mb{v})}{v}{w_0}, \R)$ we know that a stability condition $\theta$ corresponds to an ample bundle on the corresponding birational model. Theorem \ref{thm:birat_sym} and the isomorphism 
\[H^2(\quivb{\theta_{Hilb}(n\delta )}{n\delta }{w_0}, \R) \simeq H^2(\quivb{\theta_{Hilb}(\mb{v})}{v}{w_0}, \R) \] also imply that every birational model for $\quivb{\theta_{Hilb}(\mb{v})}{v}{w_0}$ over $Sym^n(\C^2/ \Gamma)$ is given by variation of $\theta \in \Theta_v$ such that the birational model corresponding to $\theta$ is the birational model associated to the image of $\theta$ in $H^2(\quivb{\theta_{Hilb}(n\delta )}{n\delta }{w_0}, \R)$. In particular, if $\theta$ lies on a wall for GIT stability on $\Theta_{\mb{v}}$ but $\phi(\theta) \in \Theta_{n\delta}$ does not lie on a wall for GIT stability, then the map 
$\pi_{\theta', \theta}$ from \eqref{eq:semisimplification} for generic adjacent $\theta'$ is an isomorphism onto its image. 

We can also deduce that $\phi: \Theta_{\mb{v}} \to \Theta_{n\delta}$ is an isomorphism the level $0$ hyperplane $H_0$ since one chamber in $H_0$ is a wall of the nef cone of the Hilbert scheme chamber corresponding to $Sym^n(X_\Gamma)$ by Theorem \ref{thm:quivers_are_torsionfree_moduli}, and the other chambers correspond to $Sym^n(\on{Flop}(X_\Gamma))$ the symmetric power on the (isomorphic) surface obtained by flopping some $-2$ curves. 

The remainder of the identification essentially arises by noting that for each finite root wall $\alpha^\perp \subset H_0$ where $\alpha$ is a finite root, there is a unique hyperplane $H$ in the family $\{ (k \delta + \alpha)^\perp | k \in \Z\}$ of all hyperplanes in $\Theta_{\mb{v}}$ intersecting $H_0$ transversally at $\alpha^\perp\cap H_0$ which induces a divisorial contraction. This fact follows from the case $\mb{v} = n\delta$ where it is part of Theorem \ref{thm:birat_sym}. In this case, the unique hyperplane is $\alpha^\perp$, and these hyperplanes intersect at the line spanned by $e_0$. For other $\mb{v}$ we show that the divisorial hyperplanes intersect at the line spanned by $e_0 +(-\sum_{i\neq 0} u_i, u_1, u_2, \ldots,u_n)$.

To show that it actually preserves $H_1$ it is actually necessary to identify something about the birational contractions induced by the other hyperplanes intersecting $\alpha^\perp \cap H_0$. To this end, fix a positive finite root $\alpha = \sum_{i = 1}^n a_i \alpha_i$ where $\alpha_i$ are the simple positive roots.  Let 
\[\tau := a\cdot u = \sum_{i = 1}^n a_i u_i = -a^T C \mb{v}\]

Then for $k \in \Z$, a generic $\theta \in ((\tau + k)\delta + \alpha)^\perp$, if it induces a birational contraction, by Theorem \ref{thm:local_quiver} induces a contraction whose generic singular fiber is the central fiber of $\pi_{\rho, 0}$ for the Ext quiver of the decomposition $\mb v = (\mb v - \ell \beta_k) \oplus  \beta_k^{\oplus \ell} $ where $ \ell \ge 1$ is an integer, $\rho$ is a generic stability condition and 
\[ \beta_k = \begin{cases}(\tau + k)\delta - \alpha & \tau + k > 0 \\ -(\tau + k)\delta + \alpha & \tau + k \le 0\end{cases}. \] 
  This Ext quiver has only a single dimension node and no loops, and dimension $\ell$ at this node.  Also the framing dimension $w_\ell$ is $ - ((1,\mb{v} - \ell \beta_k), (0,\beta_k) )$ where $(-,-)$ is the Cartan pairing for the quiver $Q_\infty$ with Cartan matrix 
  \[C_\infty = \left( \begin{array}{c|cc c}
    2&-1&& \\ \hline
    -1&&&\\
    &&C&\\
    &&&
    \end{array}\right).\]
 Let $v_\infty = (1, \mb v)$ and identify $\beta_k$ with $(0, \beta_k)$. Thus we can calculate (in the $\tau + k > 0$ case) 
  \begin{align*}
     w_\ell &= - (v_\infty - \ell\beta_k)^T C_\infty\beta_k\\
    &=  -v_\infty^T C_\infty\beta_k + 2\ell \\
&= (\tau + k) - \mb{v}^T C ((\tau + k)\delta - \alpha)  + 2\ell \\
&= \tau + k + \mb{v}^T C \alpha + 2\ell \\
&= k + 2\ell
  \end{align*}
by the fact that $\tau = -\mb{v}^T C \alpha$. In the case $\tau + k \le 0$ the same calculation shows that $w_\ell = -k + 2\ell $.  The central fiber of $\pi_{\rho , 0}$ of a quiver with one node, dimension $\mb{v} = \ell$ and framing $\mb{w_\ell} = \pm k + 2\ell$ is a Grassmannian $Gr(\ell, \pm k + 2\ell)$. We know from the $\mb{v} = n\delta$ case that all of these fibers, if the relevant wall induces a contraction, are $\PP^r$ for some $r$. Thus either $\ell = 1$ and $\pm k \ge 0$ or $\ell = \pm k + 2\ell - 1$ so $\ell = \mp k  + 1$. If $\ell \neq 1$ we must have $\pm k < 0$.  

Since for a semismall map, a divisorial contraction must have a curve as its generic positive dimensional fiber, the fiber must be $\PP^1$ and we can identify which $k$ corresponds to the unique wall in the family inducing a divisorial contraction. Namely, the divisorial wall must be when $k = 0$, with $\ell = 1$.  

Then since walls for adjacent $k$ have the same difference in dimension of generic positive dimensional fibers as do walls for adjacent $k$ for the dimension vector $n\delta$, the level 1 hyperplane is preserved. This actually only needs to be checked for $k \ge 0$, in which case adjacent walls $((\tau + k)\delta - \alpha)^\perp$ and $((\tau + k+1)\delta - \alpha)^\perp$ induce contractions with generic positive dimensional fibers $\PP^{k+1}$ and $\PP^{k+2}$ respectively, the same as for $\mb{v} = n\delta$.

But by letting $\alpha$ range over $\alpha_i$, we know that the shift at the level 1 hyperplane is exactly by $u$, since the line spanned by $\phi(e_0) = \cap_{i =1}^n \phi(\alpha_i^\perp)$ is the intersection of $(\tau_i \delta - \alpha)^\perp$ for $i = 1, \ldots, n$. 

For $n = 1$ the exact same argument gives the identification between $\Theta_{\mb v}$ and $\Theta_{n\delta}$ and the rest of the result follows from the fact that the map from a stability vector in $\Theta_{n\delta}$ to the line bundle in $H^2(X_\Gamma)$ is equivalent to forgetting the  $e_0$ component of $\theta$.  
\end{proof}

\begin{figure}[h!]
    \centering
  \includegraphics[width=.72\linewidth]{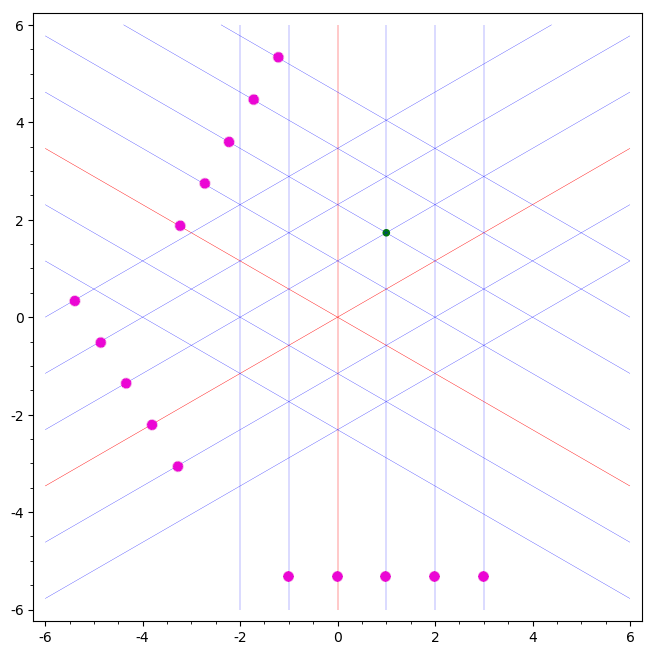}
  \caption{Level 1 hyperplanes for stability space of the affine $A_2$ quiver for framing vector $\mb{w_0}$ and dimension vector $\mb v = 3\delta + e_0 = (4,3,3)$. The red lines correspond to the walls $e_\alpha^\perp$ for finite roots $e_{\alpha}$ and the green dot is placed in the center of all of the divisorial walls other than the $\delta^\perp$ wall. A purple dot is placed on every wall actually inducing a birational contraction. Compare with Figure \ref{fig:3delta}}.
\end{figure}
\newpage
\section{Birational geometry of Hilbert schemes of points on K3 surfaces}\label{sec:birational_hilb}
The work of Bellamy and Craw in \cite{bellamy2018birational} was intended to describe the birational geometry for $X_\Gamma^{[n]}$ in a way analogously to how it had been done for moduli spaces of complexes of sheaves on a K3 surface in \cite{bayer2014mmp, bayer2014projectivity}, which we now recall. Throughout this section $S$ denotes a K3 surface. 

\subsection{Stability conditions for K3 surfaces}\label{sec:stabk3}
Let $\beta, \omega \in \NS(S) \otimes \R$ with $\omega \in \Amp(S)\otimes \R$. Define
\begin{equation}\label{eq:cc}
  Z_{\omega, \beta}(E) = (e^{\beta + i\omega}, v(E))
\end{equation}
to be the pairing of the exponential of the complexified K\"ahler class with the Mukai vector of $E \in D^b(S)$.
Let
\begin{equation}\label{eq:p_stab}
  \mc{P}_0^+(S)\subset H^*_{alg}(S, \Z)\otimes \C
\end{equation} 
be the set of $\Omega$ such that
\begin{itemize}
  \item The real and imaginary parts of $\Omega$ span a positive definite $2$-plane in $H^*_{alg}(S, \Z)\otimes \R$.
  \item For $\Omega\in \mc{P}_0^+(S)$, we have $(\Omega, s) \neq 0$ for all $s$ spherical, i.e. with $(s,s) = -2$.
  \item The orientation of the 2-plane spanned by $\Omega$ agrees with that of $\Omega = e^{\beta + i\omega}$.
\end{itemize}
This allows us to write down a description of the \emph{Bridgeland component} of the space of stability conditions on $D^b(S)$. Let 
\[ \mc{Z}: \Stab(S) \to  H^*_{alg}(S, \Z)\otimes \C\]
denote the map sending a stability condition $\sigma$ to the vector $\mc{Z}(\sigma)$ such that the central charge of $\sigma$ is $( \mc{Z}(\sigma), -)$.  

\begin{thm}[Bridgeland \cite{bridgeland2008stability}]\label{thm:bridgeland_cmp}
There is a connected component $\Stab^\dagger(S) \subset \Stab(S)$ such that $\mc{Z}: \Stab^\dagger(S) \to H^*_{alg}(S, \Z) \otimes \C$ is a covering map over $\mc{P}_0^+(S)$.
\end{thm}
Let $U(S)$ denote the subset of $\Stab^\dagger(S)$ such that all skyscraper sheaves are stable of the same phase. In \cite{bridgeland2008stability} it is shown that the universal cover $\widetilde{\GL}^+(2, \R)$ (which doesn't change the classification of objects as stable, semistable, etc.) acts freely on $U(S)$, and for every $\sigma\in U(S)$ there is a unique element of $g$ such that $g\sigma$ has central charge $Z_{\omega, \beta}$ given by \eqref{eq:cc} and skyscraper sheaves are stable of phase 1. Thus following Bridgeland define 
\begin{equation}
  \label{eq:stab_v}
  V(S) := \{ \sigma \in U(S) \mid \mc{Z}(\sigma) = e^{\beta + i\omega}, \text{ each } \mc{O}_x \text{ is stable of phase } 1 \}
\end{equation}
where $\beta$ and $\omega$ run over $\NS(S)\otimes \R$ with $\omega$ positive.
Restricted to $V(S)$ the map $\mc{Z}$ is a homeomorphism 
\[ \mc{Z}: V(S) \to \mc{L}(S)\]
where 
\begin{equation}
  \label{eq:Lstab}
  \mc{L}(S) := \{e^{\beta + i\omega} \mid \omega \in \Amp(S),~ (e^{\beta + i\omega}, \delta)\not \in \R_{\le 0} \text{ if } \delta^2 = -2 \text{ and } r(\delta) > 0 \}
\end{equation} 
where $v(\delta) = (r(\delta), c_1(\delta), s(\delta))$ gives the $r(\delta)$ component of the Mukai vector of $\delta$. 

The boundary $\partial U$ of this set of stability conditions is understood by Bridgeland \cite[Theorem 12.1]{bridgeland2008stability} in terms of walls which correspond to destabilizing sequences for skyscraper sheaves with respect to spherical twists.
Then walls are denoted 
\begin{equation}
(A_+), (A_-), (C_k)\label{eq:uwalls_stab}
\end{equation}
depending on the destabilizing object. The $(A_+)$ and$(A_-)$ cases corresponds to a spherical vector bundle where all skyscraper sheaves are destabilized and is not relevant to us, and when $\sigma$ lies generically on the $(C_k)$ wall for $k \in \Z$ and a smooth rational curve $C\in S$, then 
$k(x)$ is $\sigma$-stable for $x\not \in C$ and if $x\in \C$ then the destabilizing triangle is 
\[ \mc{O}_C(k+1)\to k(x) \to \mc{O}_C(k)[1] \to  \]
which is exactly the triangle defining the spherical twist of $k(x)$ by $\mc{O}_C(k+1)$.
\subsection{Matsuki-Wentworth twisted stability}
Under mild hypotheses, Bridgeland stability reduces at large volume $\omega^2 \gg 0$ to a twisted version of Gieseker stability, introduced earlier by Matsuki and Wentworth \cite{matsuki_mumford_1997}. This fact will be used so we recall the relevant definitions and equivalence at large volume. 

Given $\beta, \omega \in \NS(S)\otimes \R$ with $\omega$ ample, and a torsion-free sheaf $E$ with $v(E) = (r, c_1, s)$, let
\begin{equation}\label{eq:slope_mw}\begin{split}
   \mu_{\beta, \omega}(E) &= \frac{(c_1 - r\beta) \cdot \omega }{r}\\
   \nu_{\beta, \omega}(E) &= \frac{s - c_1 \cdot \beta }{r}
\end{split}
\end{equation}
\begin{defn}\label{def:matsuki_wentworth}
The torsion-free sheaf $E$ is $(\beta, \omega)$-twisted semistable if for every subsheaf $0 \neq A \subset E$ either $\mu_{\beta, \omega}(A) < \mu_{\beta, \omega}(E)$ or ($\mu_{\beta, \omega}(A) = \mu_{\beta, \omega}(E)$ and $\nu_{\beta, \omega}(A) \le \nu_{\beta, \omega}(E)$). 
\end{defn}
With this definition
\begin{prop}[Bridgeland \cite{bridgeland2008stability}~ Prop. 14.2]\label{prop:bridge_is_twist}
For $\omega$ ample with $\omega^2 \gg 0$ and fixed $\beta$ there is a unique stability condition $\sigma \in V(S)$ with $\mc{Z}(\sigma) = e^{\beta + i \omega}$ and if $E$ is an object with Mukai vector $(r, c_1, s)$ with $r > 0$ and $(c_1 - r\beta)\cdot \omega > 0$ then 
\[ E \text{ is }\sigma\text{-semistable} \Leftrightarrow E[k] \text { is } (\beta, \omega)\text{-twisted semistable for some } k. \]
\end{prop}

\subsection{Moduli spaces of stable complexes}
It is interesting that if one is only interested in studying the birational geometry of $S^{[n]}$ it is necessary to understand the moduli space of stable objects in $D^b(S)$. We briefly review the construction of projective moduli spaces of Bridgeland stable objects in $D^b(S)$ completed in \cite{bayer2014projectivity} and important facts about these moduli spaces.
\paragraph{Mukai homomorphism}
Let $\sigma\in \Stab^\dagger(S)$ be a stability condition on $S$, and let $\mc{E}\to B$ be a family of semistable objects of the same phase in $D^b(S)$ with Mukai vector $v$. Consider $\Phi_\mc{E}: D^b(S) \to D^b(B, \alpha)$, the convolution, or Fourier-Mukai functor, with image in $\alpha$-twisted objects on $B$. Let the \emph{Mukai homomorphism}
\begin{equation}\label{eq:mukai_homo}
  \theta_v : v^\perp \to \NS(B)
\end{equation}
be defined by requiring that
\[\theta_v(w)\cdot C = (w, \mathbf{v}(\Phi_{\mc E}(\mc O_C))).\]
A key property proved for Moduli spaces of sheaves in \cite{yoshioka2001moduli} is that if $v^2 \ge 0$ and $B$ is the moduli space $M_\sigma(S)$ for generic $\sigma$, such that $M_\sigma(v)$ is equivalent to a moduli spaces of sheaves under a derived equivalence, then
\[ \theta_v : v^\perp \xrightarrow{~} \NS(M_\sigma(v))\]
is an isomorphism.
\begin{rmk}
\begin{enumerate}[(i)]
  \item When $v^2 = 0$ this is really an isomorphism $\theta_v : v^\perp/ v \xrightarrow{~} \NS(M_\sigma(v))$.
  \item For any primitive $v$ and generic $\sigma$ since $M_\sigma(v)$ is an irreducible holomorphic symplectic variety, $\NS(M_\sigma(v))$ admits on general grounds the Beauville-Bogomolov form, which is identified with the Mukai pairing under $\theta_v$. Thus for example we can interpret constructions such as Proposition \ref{prop:markman_refl} defined in $\NS(M_\sigma(v))$ with respect to the Beauville-Bogomolov form concretely with respect to the Mukai lattice $H^*_{alg}(S, \Z)$.
\end{enumerate}
\end{rmk}

\paragraph{Construction of projective moduli spaces} We now give some details on the construction of the moduli spaces of stable objects, which will be necessary later.
\begin{const}\label{const:bridgeland_stable_spaces} Let $\sigma\in \Stab^\dagger(S)$ be a stability condition and $v = mv_0\in H^*_{alg}(S, \Z)$ a Mukai vector with $m > 0$ and $v_0$ primitive. Assume $v_0^2 \ge - 2$ so the space is non-empty. We construct the moduli space $M_\sigma(v)$ of $\sigma$-semistable objects with Mukai vector $v$ together with an ample line bundle, in some cases including those relevant to our situation.
  \subparagraph{Step 1} If $v_0^2 = -2$ then all semistable objects are $S$-equivalent and the
  corresponding moduli space $M_\sigma(v)$ is a point. Likewise, if $\sigma$ is generic and $v_0^2
  = 0$ then when $m = 1$, $M_\sigma(v)$ is a projective K3 surface and there is a derived
  equivalence $D^b(S) \simeq D^b(M_\sigma(v), \alpha)$ between $S$ and twisted derived category on
  $M_\sigma(v)$. When $m > 1$ we have $M_\sigma(v) \simeq \on{Sym}^n(M_\sigma(v_0))$.
  \subparagraph{Step 2} (Finding coarse moduli space for generic $\sigma$) Now we can assume $v^2 > 0$. Assume $\sigma$ is generic. Deforming $\sigma$
  slightly within the chamber (which doesn't affect stability, or therefore the moduli stacks) we
  can find primitive $w$ such that $w^2 = 0$, and such that $Z(w)$ and $Z(v)$ are positive real scalar multiples
  of each other. Further, there is a  Fourier-Mukai transform  $\Phi : D^b(S) \simeq D^b(M_\sigma(w),
  \alpha)$ as in step 1 and under $\Phi$ we have  $\Phi (\sigma) \in U \subset \Stab^\dagger(M_\sigma(w))$, so let it be
  equivalent up to $\widetilde{\GL}(2, \R)$ to one with central charge $Z_{\omega', \beta'}$. It
  turns out that under $\Phi$ composed with a shift $[-1]$, $\sigma$ stability for objects of Mukai
  vector $v$ is equivalent to $\omega'$-Gieseker stability for objects of Mukai vector $-\Phi(v)$ essentially due to Proposition \ref{prop:bridge_is_twist},
  so the Gieseker moduli space $M_{\omega'}(-\Phi(v))$ is the desired coarse moduli space.
  \subparagraph{Step 3} (Ample line bundle for generic $\sigma$ and $v$ primitive) Suppose $v^2 \ge 0$ with $v$ primitive. We have that $M_\sigma(v)$ from Step 2 is an irreducible holomorphic symplectic manifold with universal family $\mc E$, and Fourier-Mukai transform $\Phi_{\mc{E}} : D^b(M_\sigma(v)) \to D^b(S)$.
Pick a stability condition equivalent to $\sigma$ up to $\widetilde{\GL}(2, \R)$ with $Z(v) = -1$, and let $\Omega_Z$ be defined by the central charge by requiring that $Z(~) = (\Omega_Z, ~)$. Then a key result of \cite{bayer2014projectivity} is that under the Mukai homomorphism \eqref{eq:mukai_homo} the class
\begin{equation}\label{eq:bm_ample}
\ell_\sigma := \theta_v(\on{im}(\Omega_Z))\in \NS(M_\sigma(v))
\end{equation}
is ample. In particular, $M_\sigma(v)$ is projective.
\subparagraph{Step 4} (What happens for non-generic stability conditions, primitive $v$) Now let $\sigma_0$ lie on a wall but be generic in this wall, and let $\sigma_+, \sigma_-$ lie in chambers on either side of the wall. Since
the definition \eqref{eq:bm_ample} makes sense when the base is any family $S$ instead of the moduli space $M_\sigma(v)$, and since $\mc{E}_\pm \to M_{\sigma_\pm}(v)$ are in particular families of $\sigma_0$ semistable objects, we get line bundles $\ell_{\sigma_0, \mc{E}_\pm}$ on $M_{\sigma_\pm}(v)$, which are big and nef, and induce morphisms
\[\pi^{\pm}: M_{\sigma_\pm}(v) \to M_{\pm}\]
to normal projective irreducible varieties contracting $S$-equivalent objects. Even when there is a natural identification $M_+ = M_-$, this $M_+ = M_-$ is not really a natural definition for $M_{\sigma_0}(v)$, because there may be polystable sheaves which are not accessible as sums of Jordan-Holder factors of $\sigma_+$-stable objects of Mukai vector $v$. For example, take a spherical object $S$ with Mukai vector $s$ and any primitive Mukai vector $v$ with $v^2 \ge -2$. Then in a stability condition $\sigma_0$ where the phase of $v$ overlaps with that of $S$, and $S$ is $\sigma_0$-stable, and any $\sigma_0$-semistable object $\mc{E}$ of Mukai vector $v$, we have that $\mc{E} \oplus S^{\oplus k}$ is $\sigma_0$-semistable for any $k > 0$. But $(v + ks)^2 = v^2 + 2k(v,s) -2k^2$ so for large enough $k$ there are no semistable objects of Mukai vector $(v + ks)$ for an adjacent generic stability condition.  Compare this with the quiver variety case where the map \eqref{eq:semisimplification} is not in general surjective. 

\end{const}
\subsection{Bayer-Macr\`i description of MMP} In addition to the projectivity of the moduli spaces, Bayer and Macr\`i give an comprehensive description of the birational geometry of the same moduli spaces\cite{bayer2014mmp} based on the analytic map for a chamber $\mc{C}\subset \Stab^\dagger(S)$ and $\sigma \in \mc{C}$
\begin{equation}\label{eq:lmap}\begin{split}
\ell : \mc{C} &\to \NS(M_\sigma(v))\\
\sigma' &\mapsto \ell_{\sigma'}\end{split}
\end{equation}
sending a stability condition in the chamber to the corresponding line bundle from \eqref{eq:bm_ample}. The following description follows very closely \cite{bayer2014mmp, bayer2014projectivity}, with the modification described in Remark \ref{rmk:no_bounce}.
These results relate this description of cones with its wall and chamber structure to the wall and chamber structure of $\Stab^\dagger(S)$ via the map \eqref{eq:lmap}.
\begin{thm}[Bayer-Macr\`i\cite{bayer2014mmp} Thm. 1.1, Thm. 1.2]\label{thm:bm_mmp} Fix a generic basepoint $\sigma \in \Stab^\dagger(S)$ and $v\in H^*_{alg}(S, \Z)$ primitive with $v^2 > 0$.
  \begin{enumerate}[(i)]
    \item Given $\tau \in \Stab^\dagger(S)$ generic there is a birational transformation $M_\tau(v) \simeq M_\sigma(v)$. These birational transformations can be chosen so that if we identify $H^2(M_{\tau}(v), \Z)$ for different $\tau$ using these birational transformations, the maps $\ell: \mc{C} \to \NS(M_\sigma(v))$ for different chambers $\mc{C}$ glue to give an analytic map
    \begin{equation}\label{eq:lstab}
      \ell : \Stab^\dagger(S) \to \NS(M_\sigma(v)).
    \end{equation}
    \item The image of $\ell$ is the positive cone $\Pos(M_\sigma(v))$ in $\NS(M_\sigma(v))$.
    \item The map $\ell$ is compatible with the decomposition of $\Pos(M_\sigma(v))$ where $\Pos(M_\sigma(v))$ if first decomposed into chambers for the group $W_{Exc}$ from Proposition \ref{prop:markman_refl} and each chamber is then decomposed into (the Weyl reflection of) the decomposition \eqref{eq:amp_decomp} of the moveable cone.  The image $\ell(\sigma')$ of a generic stability condition $\sigma'$ corresponds to the birational model $M_{\sigma'}(v)$.
    \item The image of a chamber $\mc{C}$ is exactly the ample cone $\Amp(M_{\mc{C}}(v))$ of the corresponding birational model.
  \end{enumerate}
\end{thm}

\begin{rmk}
\label{rmk:no_bounce} Our convention for the map $\ell: \mc{C}\to \NS(M_\sigma(v))$ differs slightly from the one in Thm. 1.1, Thm. 1.2 in \cite{bayer2014mmp}, which we will denote $\ell_{BM}$. The key differences are 
\begin{itemize}
  \item The image of $\ell$ is the entire positive cone. The image of $\ell_{BM}$ is the cone $\Mov(S) \cap \Pos(S)$ of big movable divisors.
  \item This $\ell$ is analytic, while $\ell_{BM}$ is continuous and \emph{piecewise} analytic.
  \item  When the image of $\ell_{BM}$ hits the wall of the movable cone, it bounces back into $\Mov(S)$ while $\ell$ continues across the wall. 
\end{itemize}
Lemma 10.1 from \cite{bayer2014mmp} relates the two descriptions, since if two chambers $\mc{C}, \mc{C}'\subset \Stab^\dagger(S)$ are separated by a wall inducing a contraction of the divisor $D$ then $\ell_{BM}$ restricted to $\mc{C}$ and $\mc{C'}$ are analytic continuations of each other after the Markman reflection $\rho_D$. 

The map $\ell_{BM}$ is more natural for the purposes of the minimal model program where it is desirable to identify a complete set $\{M_i\}$ of minimal models of a variety $M$. On the other hand, $\ell$ is useful for the present application in geometric representation theory where the action on cohomology induced by a birational transformation is of central importance, even if the two varieties are isomorphic. 
\end{rmk}

We will also need an explicit formulation of the map sending a stability condition to an element of $\NS(M_\sigma(v))$, which involves writing out \eqref{eq:bm_ample} explicitly.

\begin{prop}[Bayer-Macri \cite{bayer2014projectivity} Lemma 9.2]
  \label{prop:explicit_stab_to_pos}
  Let $v = (r, c, s)$ be a primitive Mukai vector with $v^2 \ge -2$ and fix a generic basepoint $\sigma\in \Stab^\dagger(S)$.  The map $\ell$ of \eqref{eq:lstab} restricted to the set $V(S)$ of \eqref{eq:stab_limit_small_exc} sends a stability condition $\sigma_{\beta, \omega} \in \mc{V}(S)$ with central charge $Z(-) = (e^{\beta + \omega}, -)$ to $\ell(\sigma_{\beta, \omega})$ which is a real positive multiple of  $\theta_v((r_{\beta, \omega}, C_{\beta, \omega}, s_{\beta, \omega}))$ where 
  \begin{align*}
    r_{\omega, \beta} &= c\cdot \omega - r\beta\cdot \omega\\
    C_{\omega, \beta} &= (c\cdot \omega - r\beta \cdot\omega)\beta + \left(s - c\cdot \beta + r \frac{\beta^2 - \omega^2}{2}\right)\omega\\
    s_{\omega, \beta} &= c\cdot \omega\frac{\beta^2 - \omega^2}{2} + s\beta\cdot \omega - (c\cdot \omega)(\beta\cdot \omega).
  \end{align*}
\end{prop}

They also provide a description of relevant cones in $\overline{\Pos}(M)$. We use implicitly the duality between curves and divisor classes induced by the Beauville-Bogomolov form.
\begin{thm}[Bayer-Macr\`i\cite{bayer2014mmp} \S 12]\label{thm:bm_cones} Let $v, \sigma$ be as in Theorem \ref{thm:bm_mmp}. Let $M := M_\sigma(v)$.
  \begin{enumerate}[(i)]
    \item The nef cone $\Nef(M)\subset \overline{\Pos}(M)$ is cut out by linear subspaces
    \[ \{\theta_v(v^\perp \cap \alpha^\perp) \mid \alpha \in H_{alg}^*(S, \Z), ~\alpha^2 \ge -2,~ 0\le (v,  \alpha) \le \frac{v^2}{2}\}.\]
    \item Dually, the cone of curves $\NE_\R(M)$ is generated by positive curves (i.e. $C$ such that $(C, C) > 0$ and $(C.A) > 0$ for fixed ample class $A$) and classes
    \[\{\theta(a) \mid \alpha \in H_{alg}^*(S, \Z), ~a^2 \ge -2, ~ |(v, a)| \le v^2/2, ~ (\theta(a), A) > 0\}\]
    \item The movable cone $\Mov(M)\subset \overline{\Pos}(M)$ is cut out by linear subspaces
    \[ \{\theta_v(v^\perp \cap s^\perp) \mid  s\in v^\perp \text{ spherical}\}\]
    and subspaces
    \[ \{\theta(v^\perp \cap w^\perp) \mid w \in H_{alg}^*(S, \Z),  w^2 = 0, ~ 1\le (w, v) \le 2\}\]
    \item The effective cone $\on{Eff}(M)$ is generated by $\overline{\Pos}(M)$ and the exceptional divisors
    \[ \{D :=  \theta(s)\mid s\in v^\perp, ~s\text{ spherical}, ~ (D, A) > 0\}\]
    and
    \[ \{D := \theta(v^2 \cdot w - (v, w) \cdot v) \mid w \in H_{alg}^*(S, \Z), ~ w^2 = 0, ~ 1\le (w, v) \le 2, (D, A) > 0\}.\]
  \end{enumerate}
\end{thm}
\begin{rmk}
  Not all of the classes described form extremal rays or walls of the corresponding cones.
\end{rmk}

\subsection{Local structure of singularities via Ext-quivers}\label{sec:local_k3_ext_quiver}
 A key tool in the study of singularities of moduli spaces of sheaves is the local analytic description of singularities in terms of $\Ext$ quivers \cite{kaledin2007local, arbarello2018singularities, toda2018moduli}. 
 
 It will turn out that the situation described in this paper is a certain negative result to a certain conjecture that chambers in 
 $\Stab^\dagger(S)$ around a specific $\sigma_0$ correspond 1-1 with chambers in the stability space for an Ext-quiver of a polystable object so that the resolution for $M_\sigma(v)$ is locally around this point given by the resolution of the corresponding Ext quiver. 
 It will turn out instead of matching chambers in $\Stab^\dagger(S)$ adjacent to a given $\sigma_0$ with chambers in the space $\Stab^\dagger(S)$ is a certain resolution of the stability space of the corresponding local affine ADE quiver (c.f. Section \ref{sec:k3_stability_conditions_corner}) where the map $\ell: \Stab^\dagger(S) \to \NS(M_{\sigma}(v))$ acts as the resoluton. There will, however, for the stability conditions under consideration, be a 1-1 correspondence between chambers in $\Pos(M_\sigma(v))$ around a given point and chambers for the corresponding quiver variety giving a local description of the singularity. 
 
 Thus instead of using a general argument that says the local structure of moduli space around a polystable object is described by its Ext-quiver, this paper uses a more brute-force technique 
  (c.f. Section \ref{sec:local_analytic_brute_force}) to describe the
   local analytic structure of singularities for the case of Mukai vectors corresponding to rank 1 torsion free sheaves. We will however state and use results about Lie algebra actions which hold when we do have a description of a singularity around a polystable object as an Ext quiver.

 A key property to understand completely the local singularity of a moduli space of sheaves or of complexes at a sheaf $F$ is formality which is only known in some cases \cite{arbarello2018singularities, budur2019formality}.
\begin{defn}\label{defn:formality}
  \item[(1)] A DG-algebra $A$ is \emph{formal} if there is a quasi-isomorphism
  \[ A\simeq H^*(A)\]
  where $H^*(A)$ is the cohomology complex of $A$ taken with $0$ differentials.
\item[(2)] A sheaf or object $F\in D^b(S)$ satisfies the \emph{formality condition} if $R\Hom(F, F)$ is formal, i.e. $R\Hom(F, F)\simeq \Ext^*(F,F)$.
\end{defn}
The following definition extends Definition \ref{defn:ext_quiver_quiv} to the present context of objects in the derived category.
\begin{defn}\label{defn:ext_quiver_object}
  Let $\sigma$ be a Bridgeland stability condition on $S$ and let $F = \bigoplus_{i = 1}^s F_i ^{\oplus n_i}$ be a polystable object where the $F_i$ are pairwise distinct $\sigma$-stable objects. The $\Ext$-quiver $Q(F)$ has vertices labelled by $I = \{1, \ldots, s\}$, one for each stable factor of $F$, and the number of edges between vertices $i$ and $j$ is
  \[
  \begin{cases}
    \ext^1(F_i, F_i)/ 2 &  i = j\\
    \ext^1(F_i, F_j) &  i \neq j
  \end{cases}
  \]
  and given a stability parameter $\theta$ the corresponding quiver variety is denoted
  \[ \mf{M}_\theta(F) := \mf{M}_\theta(\mathbf{n}, 0).\]
\end{defn}
\begin{rmk}
  In this situation $\Rep_{\overline{Q(F)}}(\mathbf{n}, 0)\simeq \Ext^1(F, F)$ and $\mf{g}(\mathbf{n})^* \simeq \Ext^2(F,F)$ such that the moment map is given by the Yoneda product $e\mapsto e\cup e$, which is also the quadratic part of the \emph{Kuranishi map} governing obstructions to deformations.
\end{rmk}
These definitions allows us to state the key result on the local description of singularities in the case of sheaves.
\begin{thm}[Arbarello-Sacc\`a \cite{arbarello2018singularities}]
Let $H_0$ be a polarization of $S$ and $F = \bigoplus_{i = 1}^s F_i ^{\oplus n_i}$ as in Definition \ref{defn:ext_quiver_object}, except with stability taken with respect to $H_0$, with Mukai vector $v$ and Ext-quiver $Q(F)$ such that $F$ satisfies the formality property Definition \ref{defn:formality}. Then
\begin{enumerate}[(i)]
  \item There is a local analytic isomorphism
  \[ \psi: (\mf{M}_0(F), 0) \simeq (M_{H_0}(v), [F]).\]
  \item Given a chamber $C\subset \Amp(S)$ there is a chamber $D$ of stability parameters for $\Rep_{\overline{Q(F)}}(\mathbf{n}, 0)$, for every $H\in C$ and $\theta \in D$ the natural maps
  \[ \xi: \mf{M}_\theta(F) \to \mf{M}_0(F) ~\text{  and  } ~ h: M_H(v) \to M_{H_0}(v)\]
  coincide over neighborhoods of $0$ and $[F]$ as long as $h$ is regular over $[F]$.
  \item The correspondence between chambers $D$ and $C$ is given by the assignment $H \mapsto \theta = (\theta_1, \ldots, \theta_s)$ where
  \[ \theta_i = c_1(F_i) \cdot ( H - H_0)\]
  as long as $h$ is regular over $[F]$.
\end{enumerate}
\end{thm}

\section{Stability conditions and birational geometry for specific K3s}\label{sec:stabk3_specific}
We now restrict to the case of a K3 surface $S$ such that $\bar{\NE}(S)$ is the cone spanned irreducible $-2$ curves, and any pair of these either don't intersect or intersect transversely at a single point.

Thus $\Nef(S)$ is a locally polyhedral cone in $\NS(S)_\R \simeq \R^{1, \rho(S) -1}$, and $\Nef(S)\cap \mc{C}_+$ is a fundamental domain for the action of $W_S$ on $C_+$.  Consider the wall and chamber structure on $C_+$ induced by all $-2$ classes, and let a face denote a face of any dimension of this decomposition. Let $W_1, W_2$ be two distinct walls of this cone.  Then the dihedral angle between $W_1$ and $W_2$ is $\pi/2$ or $\pi/3$. If we consider a big and nef divisor on the boundary of $\Nef(S)$ it induces a contraction of the $-2$ curves corresponding to the walls it lies on. Thus faces of the nef cone correspond to contractions of some disjoint set of ADE systems of -2 curves, and locally near a big and nef divisor $\ell$ the wall and chamber decomposition corresponds to the product of a Euclidean space and the wall and chamber structure around 0 induced by a finite ADE root system.

We are interested in not just the wall and chamber decomposition for $S$ but also for $S^{[n]}$ given by Theorem \ref{thm:bm_cones}. The Hilbert scheme we view as parametrizing ideal sheaves, so the Mukai vector of $S^{[n]}$ is $(1,0, 1-n)$. We know the Mukai homomorphism \eqref{eq:mukai_homo} says that for $n > 1$, the Neron Severi group of $S^{[n]}$ is the orthogonal direct sum
\begin{equation}\label{eq:ns_direct_sum}
  \NS(S^{[n]}) = \theta((1,0,1-n)^\perp) \simeq H^2(S, \Z) \oplus \Z |2, 1^{n-2}\rangle 
\end{equation}
where $|2, 1^{n-2}\rangle = \theta((-1,0,1-n))$ (c.f. \eqref{eq:h2xn}) is the locus of a moving double point, and is $2|2,1^{n-2}\rangle$ is the exceptional divisor of the Hilbert-Chow contraction. 

There is an interesting sign involved in the Mukai homomorphism on $H^2(S, \Z)$ so that if $H$ is ample on $S$ then $\theta((0,-H,0))$ is big and nef in $\NS(S^{[n]})$, rather than $\theta((0,H,0))$. This $\theta((0,-H,0))$ induces the Hilbert-Chow map and gives an ample bundle on $Sym^n(S)$. For this reason, given a class $\alpha\in H^2(S, \Z)$ we define 
\begin{equation}\label{eq:tilde_sign}
   \tilde{\alpha} := \theta((0,-\alpha, 0))
\end{equation}
so that $\tilde{H}$ is big and nef for $H$ ample. 

\subsection{Walls near $Sym^n(S)$}
We use Theorem \ref{thm:bm_mmp} in the present case to find all walls nearby the locus of movable  divisors on $S^{[n]}$ corresponding to $Sym^n(S)$. In particular, given $H\in \Amp(S)$, there is an $\epsilon > 0$ so that $\tilde H - \epsilon |2, 1^{n-2}\rangle \in \Amp_\R(S^{[n]})$. Thus $\theta(-\Nef(S))\subset |2, 1^{n-2}\rangle^\perp$ forms a maximal dimensional face of the cone $\Nef(S^{[n]})$, corresponding to big and nef divisors inducing the Hilbert-Chow map.  We investigate birational chambers in a short cylinder over this face. Denote the Hilbert-Chow wall $W_{HC} := |2, 1^{n-2}\rangle^\perp$, and call this face $\Nef(S^{[n]}) \cap W_{HC} \cap \Pos(S^{[n]})$ the $Sym^n (S)$ face of the Nef cone. Also denote this face $F_{Sym}$. See the figures in Section \ref{sec:figures} for an example. 

This subsection leads to Proposition \ref{prop:stab_corner_arranged} which says that around a point n the boundary of this face, the walls are arranged according to a set of roots in an affine root system. 

\begin{prop}
Let $W$ be a wall of $\Mov(S^{[n]})$ corresponding to a divisorial contraction intersecting $W_{HC}\cap \Pos(S^{[n]})$. Then $W = \tilde{\alpha}^\perp$ for $\alpha\in H^2(S, \Z)$ an irreducible -2 curve. In particular $W$ meets $W_{HC}$ perpendicularly.
\end{prop}
\begin{proof}
Recall from Theorem \ref{thm:bm_cones} that the wall and chamber structure on $\Pos(S^{[n]})$ cutting out the movable cone is given by walls in $\{\theta_v(v^\perp \cap s^\perp) \mid  s\in v^\perp \text{ spherical}\} \cup \{\theta(v^\perp \cap w^\perp) \mid w \in H_{alg}^*(S, \Z),  w^2 = 0, ~ 1\le (w, v) \le 2\}$ where $v = (1,0,1-n)$ is the Mukai vector. Suppose that one of these walls is $W = \theta_v(v^\perp \cap a^\perp)$ and $W$ meets $W_{HC}$. 
Then there is a positive class $H\in H^2(S, \Z)$ with $(a, H) = 0$. Suppose we are in the first case with $a$ spherical and $a\in (1,0,1-n)^\perp$. 
Then $\theta(a) = \tilde{a_S} + c|2, 1^{n-1}\rangle$ with $a_S \in H^2(S, \Z)$ and $(a_S , H) = 0$. Thus $a_S^2 < 0$ or $a_S = 0$. But $a_S = 0$ is excluded since $a^2 = -2$.
 Also 
\[ -2 = a^2 = a_S^2 + c^2(-1,0,1-n)^2 = a_S^2 + c^2(2-2n). \]
 Since $ n >1$  we have $c = 0$ and $\theta(a) = \tilde{a_S}$ as desired.
 
 Now suppose $W = \theta_v(v^\perp \cap a^\perp)$ with $a^2 = 0$ and $1\le (a,v) \le 2$. Write $a = (c_1, a_s, c_2)$ with $(a_s, H) = 0$ as before, so $a_s^2 < 0$ or $a_s = 0$. 
 Again $a_s = 0$ is excluded or we just get the Hilbert-Chow wall.
  Then $a^2 = 0$ implies $c_1c_2 < 0$. But then the condition $1 \le (a, (1,0,1-n)) \le 2$ becomes
 \[1 \le c_1(n-1) - c_2 \le 2\] 
 and the only solution which doesn't give the Hilbert-Chow wall is $ n = 2, c1 = 1, c2 = -1$ so $a = v + a_s$, which is equivalent to the previous case with $\theta(a) = \tilde{a_s}$. 
\end{proof}

Also, given $H$ big and every $\tilde{H} - \epsilon|2, 1^{n-2}\rangle$ is NEF for small enough $\epsilon>0$ so there are no flopping walls intersecting $W_{HC}\cap \Nef(S^{[n]})$ except at the place where divisorial walls intersect, and by Weyl group symmetry, no flopping walls intersecting $W_{HC}$ at all except where divisorial walls intersect $W_{HC}$. Thus the wall and chamber structure on $W_{HC}$ induced by the other walls in $\Pos(S^{[n]})$ coincides exactly with the wall and chamber structure on $\Pos(S)$ controlling the birational geometry of $S$. 

The only ingredient left to describe all of the birational geometry nearby the $Sym^n S$ face is to collect flopping walls which intersect $W_{HC}$ at the intersection $W_{HC}\cap \tilde{a}^{\perp}$ for $a \in H^2(S, \Z)$ spherical. 

\begin{defn}\label{def:arranged_hyperplanes}
  Consider a collection of hyperplanes $\mc{H}$ defining a wall and chamber structure on some cone $C$. Let $x$ be a point in the intersection $\cap_{H\in \mc{H}} H \cap C$. We say that the collection $\mc{H}$ is \emph{arranged according to the root system $\Delta$} at $x$ if 
  there is a affine subspace $H'$ of the vector space generated by $C$ with $x \in H'$ such that 
  \begin{itemize}
    \item  tangent vectors to $x$ parallel to $H'$ or $H\in \mc{H}$ together span $T_x C$, 
    \item $H \perp H'$ for all $H\in \mc{H}$, in the sense that  if $v \in T_x H, v\not\in T_x H', w\in T_x H', w\not \in T_x H$ then $v\cdot w = 0$. 
    \item  and the polyhedral decomposition in a neighborhood of $x \in H'$ induced by the hyperplanes $H'\cap H$ for $H \in \mc{H}$ coincides with the polyhedral decomposition induced by $\Delta$ in a neighborhood around $0$ in the ambient space of $\Delta$. 
  \end{itemize}
  More generally, given a collection of roots $\Delta_0 \subset \Delta$ we say that a hyperplanes are \emph{arranged according to the roots $\Delta_0$ } at $x$ if the same condition holds but we only consider the walls $\alpha^\perp$ for $\alpha \in \Delta_0$ for the wall and chamber structure on the ambient space of $\Delta$. 
\end{defn}

\begin{eg}\label{eg:associated_GIT}
  Consider the GIT stability space $\R^I$ for $X_\Gamma^{[n]}$, with walls $\mc{W}$. Let $\Delta$ be the root system for the affine Kac-Moody algebra associated to the quiver, $\Delta_f$ the associated finite root system. Then Theorem \ref{thm:birat_sym} says that the walls $\mc{W}\subset \R^I$ are arranged according to to the roots 
  \[\{\delta,  m \delta + \alpha \mid \alpha \in \Delta_f ,0 \le m < n \} \subset \Delta \]
  near the point $0\in \R^I$. 
\end{eg}

Now consider the wall and chamber structure on $\Pos(S)$ induced by $-2$ classes. Around a point $x \in \partial (\Nef(S)) \cap \Pos(S)$ there is a unique root system $\Delta$ so that the hyperplanes in 
$\{s^\perp \mid s\in H^2(S, \Z), s^2 = -2 \}$ meeting $x$ are arranged according to $\Delta$. Using the isomoprhism $\partial (\Nef(S)) \cap \Pos(S)\simeq \partial F_{Sym}$ between this boundary and the boundary of the $Sym^n(S)$ face We upgrade this to a map
\begin{equation}
  \label{eq:delta_of_x}
  \Delta : \partial F_{Sym} \to D
\end{equation}
where $D$ is the set of finite root systems whose irreducible factors are of type ADE. This map factors 
through the map (seen in the cone $\Pos(S)$) 
\[\on{Exc} : \partial (\Nef(S)) \cap \Pos(S) \to 2^\mc{C}\] sending a big and nef divisor to its exceptional locus, where $\mc{C}$ is the set of exceptional curves in $S$ and $2^{\mc{C}}$ is its power set. Thus $\Delta(H)$ describes the root system associated to the (possibly disconnected) Dynkin diagram formed as the intersection graph of $\on{Exc}(H)$.   Let $S_{H}$ denote the blown-down surface associated to $H$. Seen as an element of $\Pos(S^{[n]})$, we have that $\theta(H)$ is the big and nef divisor inducing the contraction $S^{[n]} \to Sym^n(S_H)$. 

\paragraph{Flopping walls through a point} Consider a point $x\in \partial F_{Sym}$, 
we investigate the walls which pass through this point.

To this end, let $\mc{W}$ be the set of walls intersecting $F_{Sym}$ and $\mc{W}_x\subset \mc{W}$ those walls which intersect the point $x$. Consider $\Delta(x)$ from \eqref{eq:delta_of_x} the root system associated to $x$. Let $W$ be one of these walls, distinct from $W_{HC}$. We know that $W$ intersects $W_{HC}$ in one of the walls of $\Pos(S)$ under the isomorphism $\Pos(S) \simeq W_{HC}$, so let $\alpha_x(W)$ be the root in $\Delta(x)$ associated to the wall $W$.

The wall $W$ intersects $W_{HC}$ in the locus $\tilde{s}^\perp \cap W_{HC}$ for $s \in H^2(S, \Z)$ a -2 class on the surface itself, but the walls $\tilde{s}^\perp$ are orthogonal to $W_{HC}$ while the flopping walls are not necessarily. The following lemma writes down the implication of Theorem \ref{thm:bm_cones} in this situation. 

\begin{lem}\label{lem:walls}
The walls in $\mc{W}_x$, i.e. those that intersect $x$ are given by 
\[\{ W_{HC}, \theta_v(v^\perp \cap (0, s, m)^\perp) \mid s\in H^2(S, \Z), s^2 = -2, 0 \le m < n \}.\]
\end{lem}
\begin{proof}
  Recall from \ref{thm:bm_cones} that the walls of $\Nef(S^{[n]})$ are contained in \[\{\theta_v(v^\perp \cap \alpha^\perp) \mid \alpha \in H_{alg}^*(S, \Z), ~\alpha^2 \ge -2,~ 0\le (v,  \alpha) \le \frac{v^2}{2}\}.\] Let $\tilde{H} = x$ for a big and nef divisor $H$ on $S$. So if $\theta_v(v^\perp\cap \alpha^\perp)$ is a wall of $\Nef(S^{[n]})$ passing through $x$, and \[\alpha = (c_1, a_s, c_2)\] then $a_s \cdot H = 0$. Therefore either $a_s = 0$ (in which case we have the Hilbert-Chow wall as before) or up to an inconsequential sign, $a_s$ is the class of some subset of $\on{Exc}(x)$. In particular, $a_s^2 = -2$. The condition $\alpha^2 \ge -2$ reads $c_1c_2 \le 0$
  and the condition $0 \le (\alpha, v) \le v^2/2$, since $v^2/2 = n-1$, reads
\[
  0 \le c_1(n-1) - c_2  \le n-1
\]
so $c_1 \in \{ 0, 1\}$ and $c_1 = 1 \implies c_2 = 0$.  The other case is $c_1 = 0, c_2 \in \{0,-1, \ldots, 1-n\}$.  The case $\alpha = (1, a_s, 0)$ is equivalent to 
$v- \alpha = (0, -a_s, 1-n)$ so we can assume $c_1 = 0$, and we are left with exactly the proposed list. 
\end{proof}

It should now be clear that we have the following local-to-global type correspondence between chambers for $S^{[n]}$ near the big nef divisor $x = \tilde{H}$ corresponding to $Sym^n(S_H)$ for the symmetric power of the blowndown surface, and GIT 
chambers for quiver varieties $\quiv{\theta}{\mb{v}}{\mb{w}}$ for relevant quiver varieties. 

\begin{prop}\label{prop:stab_corner_arranged}
Let $x\in \partial F_{Sym}$ be a point on the boundary of the $Sym^n$ face such that $\on{Exc}(x)$ is a connected ADE collection of -2 curves. Pick a small open set $x\in U \subset \Pos(S^{[n]})$ so that any wall $W$ intersecting $U$ passes through $x$. Consider the walls
\[ \mc{W}_{Q, \mb{v}, \mb{w}} = \{ \delta^\perp, (m\delta + \alpha)^\perp \mid 0 \le m < n, \alpha \in \Delta_f\}\]
in the GIT stability space for the quiver $Q$ of affine ADE type, $\mb{v} = n\delta, w = \mb{w_0}$ where $\Delta_f$ is the corresponding finite root system.   Then
\begin{enumerate}[(i)]
  \item there is a 1-1 correspondence between faces in the polyhedral decomposition of $U$ with respect to the walls $\mc{W}_x$ and $\R^{I}$ with respect to $\mc{W}_{Q, \mb{v}, \mb{w}}$ sending the closed point $x$ to $0$
  \item This correspondence preserves the relation $F' \le F \Leftrightarrow F' \subset \bar{F}$.
\end{enumerate}
\end{prop}
\begin{proof}
Theorem \ref{thm:birat_sym} says the wall and chamber decomposition of $\R^I$ is arranged according to some roots $\Delta_0$ in a root system $\Delta$, and Lemma \ref{lem:walls} says the wall and chamber structure around $x$ is arranged according to the same roots. 
\end{proof}

\subsection{Stability conditions}\label{sec:k3_stability_conditions_corner}
We describe specific Bridgeland stability conditions $\sigma\in \Stab^\dagger(S)$ which produce birational models $M_\sigma((1,0,1-n))$ for $S^{[n]}$ corresponding to chambers described previously. 

Our stability conditions will generally lie near a certain limit where the central charge $Z_{\beta, \omega}(-) = (e^{\beta+ i\omega}, -)$ satisfies
\begin{align}
  \omega^2  &\gg 0 \label{eq:stab_limit_lv} \\
  |\beta \cdot \omega | &< \lambda \label{eq:stab_limit_bdd_hc} \\
  0 \lneq  C \cdot \omega &< \frac{\kappa}{\omega^2} \label{eq:stab_limit_small_exc}
\end{align}
where $\lambda, \kappa$ are positive constants and $C$ lies in some possibly empty family $\mc C$ of $-2$ curves on the surface. We also consider limits where the first inequality in \eqref{eq:stab_limit_small_exc} is allowed to become an equality. 

From a metric perspective, the first condition \eqref{eq:stab_limit_lv} corresponds to the ``large volume limit'' of both the K3 surface $S$ and also the birational model of $S^{[n]}$, while the second condition 
\eqref{eq:stab_limit_bdd_hc} implies in particular that the volume in (the birational model of) $S^{[n]}$ of the (strict transform of the) exceptional locus of the Hilbert-Chow map is small compared to the volume of the moduli space, so that we are near a contraction of this locus. The third condition \eqref{eq:stab_limit_small_exc} 
implies that the volume of the -2 curves in $\mc{C}$ are extremely small in $S$ so that we are close to contracting them to a blowndown surface $S_{bd}$, and also close to contracting the moduli space to $Sym^n(S_{bd})$. 

As we vary the collection $\mc{C}$ we vary which which curves are allowed and required to contract by imposing that $C' \cdot \omega$ remains bounded below for $C'\not \in \mc{C}$. Note that by Proposition \ref{prop:bridge_is_twist}, if $\mc{C}$ is empty then we therefore remain in the Gieseker chamber and are unable to hit walls which induce birational contractions. We will also impose an additional numerical condition which essentially states that $\omega$ does not lie too close to the boundary of $\Pos(S)$.  Thus fixing $\mc{C}$ makes it so that only a fixed contractible collection of curves is relevant to how $\sigma$-stability differs from Gieseker stability. 

We record some useful calculations of central charges and slopes:

\begin{table}[h!]
  \centering
  \begin{tabular}{|l|l|l|l|}
  \hline
&$v(-)$&$\re (Z_{\beta,\omega}(-))$ &$\im(Z_{\beta, \omega}(-))$ \\\hline 
$k(x)$&$(0,0,1)$&$-1$& $0$ \\ \hline
$\mc{O}_C(k)$&$(0,C,1+k)$&$C\cdot \beta - 1 -k$&$C\cdot \omega$  \\ \hline
$\mc{O}_S$&$(1,0,1)$&$-1 +(\omega^2-\beta^2)/2$&$-\beta\cdot \omega$  \\ \hline
$\mc{L}$&$(1,c_1,1 + c_1^2/2)$&$-1+(\omega^2-\beta^2)/2 + c_1\cdot \beta - c_1^2/2$& $(c_1 - \beta )\cdot \omega$ \\ \hline
$\mc{I}_Y$&$(1,0,1-n)$&$-1 + (\omega^2-\beta^2)/2 +n$&  $-\beta\cdot\omega $ \\ \hline
$\mc{L}_Y$&$(1,c_1,1 + c_1^2/2 -n)$& $-1+(\omega^2-\beta^2)/2 + c_1\cdot \beta - c_1^2/2+n$ & $(c_1 - \beta )\cdot \omega$ \\ \hline
\end{tabular}
\caption{Mukai vectors and real and imaginary parts of central charges with respect to $Z_{\beta, \omega}(-) = (e^{\beta + i \omega}, v(-))$. Here $\mc{L}_Y$ is a torsion free subsheaf of the line bundle $\mc{L}$ with first chern class $c_1 = c_1(\mc{L})$ and with quotient $Y$ of finite length $n$.\label{tbl:slopes_z} }
\end{table}

\paragraph{Precise limit stability space} We give a precise version of our stability limit and describe the structure of the space near this limit, e.g. there are gaps in the space $\Stab^\dagger(S)$ near our limit corresponding to $Z_{\beta, \omega}(\delta) = 0$ for classes $\delta$ with $\delta^2 = -2$. Also the relevant region of $\Stab^\dagger(S)$ will actually depends on the specific chern characters we are considering, i.e. we may need to choose different constants for larger values of $n$ in $S^{[n]}$. 

For a collection $\mc{C}$ of $-2$ curves and real numbers $N, \xi, V > 0$ define the set 
\begin{equation}
  U_{\mc{C}, N, \xi, V}\subset \{(\omega, \beta) \mid \beta, \omega \in \NS(S)\otimes \R, ~ \omega^2 > 0\}\label{eq:Ulimit}
\end{equation}
 consisting of pairs $(\omega, \beta)$ such that $\omega$ is nef and
\begin{align*}
  |C\cdot \beta|&< N &  |\beta^2| &< N &-N < \beta\cdot \omega &\le 0   \\
  \omega^2 &> V & |\omega \cdot q| &> \xi & |\omega\cdot \alpha| &> \xi\\
  && 0 \le C\cdot \omega &< \frac{N}{\omega^2} & &
\end{align*}
for every $C\in \mc{C}$, every $-2$ class $\alpha\in H_2(S, \Z)$ not in $\mc{C}$ and every non-zero isotropic class $q\in H_2(S, \Z)$. Denote this set simply $U_{\mc{C}}$ if we are allowed to take $V$ as large as desired and there is no confusion. 

\begin{lem}\label{lem:stab_limit_gaps}
  Fix some contractible collection $\mc{C}$ of $-2$ curves on $S$. Given $N>0$ and $\xi> 0$ there is a  $V>0$ such that if $(\omega, \beta)$ is in $U_{\mc{C}} = U_{\mc{C}, N, \xi, V}$ then 
  \begin{enumerate}[(i)]
    \item $e^{\beta + i\omega} \in \mc{P}_0^+(S)$ where $\mc{P}_0^+(S)$ is defined in \eqref{eq:p_stab}  (i.e. $Z_{\beta, \omega}$ is the central charge for some stability condition $\sigma \in \Stab^\dagger(S)$) if and only if there is no $-2$ class $C = C_1 + \ldots C_k$ with each $C_i\in \mc{C}$ and
  \[\omega\cdot C = 0 \text{ and } \beta \cdot C = m  \]
  for some $m  \in \Z$.
  \item If $e^{\beta + i\omega} \in \mc{P}_0^+(S)$ then $e^{\beta + i\omega} \in \overline{\mc{L}(S)}$ with $\mc{L}(S)$ from \eqref{eq:Lstab}. 
  \item We have $e^{\beta + i\omega}\in \mc{L}(S)$ if and only if $\omega$ is ample. 
  \end{enumerate}
\end{lem}
\begin{proof}
This is a standard argument \cite{bridgeland2008stability,bayer2014projectivity}.  For the forward direction of (i), if there is such a $C$ then $Z_{\beta, \omega}((0,C, \ell))) = 0$ for some $\ell$. To prove the other direction, assume there is a class $\delta = (r, D, s)$ with 
$\delta^2 = D^2 -2 rs  = -2$ and $(e^{\beta + i \omega}, \delta) = 0$. Assume $r \ge 0$ without loss of generality. If $r = 0$ then our Mukai vector is $\delta = (0,C, \ell)$ and we are done. Thus $r> 0$. Since $\im(Z(\delta)) = (D - r\beta)\cdot \omega = 0$ i.e. $(D- r\beta)\in \omega^\perp$ which is a negative definite subspace (here we use the condition involving isotropic classes). Write $D = A + r\beta$ for $A \in \omega^\perp$. Then we can calculate 
\[ s = \frac{D^2 + 2}{2r}\]
so
\begin{align*}
   \re Z(\delta)) &= -r \left(\frac{\beta^2 - \omega^2}{2}\right) + D\cdot \beta - s\\
   &=  -r\left(\frac{\beta^2 - \omega^2}{2}\right) + (A + r\beta)\cdot \beta - \frac{(A+ r\beta)^2 + 2}{2r}
\end{align*}
which has no solutions $A$ for large enough $\omega^2$ and other coefficients bounded since the pairing restricted to $\omega^\perp$ is negative definite. To see (iii) note that the previous calculation is exactly the remaining condition for membership in $\mc{L}(S)$ beyond $\omega$ ample, and (ii) follows from finding and ample class $\omega'$ arbitrarily close to a given nef $\omega$.  
\end{proof}
\begin{rmk}
\begin{itemize}
  \item[(i)] Not all of the conditions on the set $U_{\mc{C}}$ are necessary for the lemma. 
  \item[(ii)] When the class $\omega$ tends to a class on the boundary of the ample cone, the stability condition $\sigma$ tends to $\partial U(S)$. If we hit a point such that $e^{\beta + i\omega}$  still lies in  $\mc{P}_0^+(S)$ then $\sigma$ lies on the intersection of walls of type $(C_{i, k_i})$ from \eqref{eq:uwalls_stab} where $C_i$ ranges over curves with $C_i \cdot \omega = 0$ and the integer $k_i$ is determined by the value of $\beta\cdot C_i$, namely it satisfies
  \[k_i < \beta\cdot C_i < k_i + 1.\] 
  The loci where $C_i \cdot \beta\in \Z$ are not in $\mc{P}_0^+(S)$ and are therefore do not correspond to Bridgeland stability conditions. 
\end{itemize}
\end{rmk}

We now describe the walls which induce birational contractions for moduli spaces $M_\sigma((1,0,1-n))$ for $\sigma$ near this limit. Given Mukai vectors $v,w$ let $W_{v, w} \subset \overline{\mc{L}(S)}$ be the locus of $(\beta, \omega)$ such that $Z_{\beta, \omega}(v)$ and $Z_{\beta, \omega}(w)$ are real positive multiples of each other, so that the phases $\phi(\mc{E})$ and $\phi(\mc{E'})$ of objects with mukai vectors $v$ and $w$ are potentially allowed to overlap in a corresponding stability condition. 

Thus in the specific case of $v = (1,0,1-n)$ and $w = (0, C, k)$ we can use Table \ref{tbl:slopes_z} and find that 
\begin{equation}
  \label{eq:wall_vcrv}
  W_{v, w} = \{ (\beta, \omega)\in \overline{\mc{L}}(S) \mid (C\cdot \beta -k)(-\beta\cdot \omega) = (C\cdot \omega)((\omega^2 - \beta^2)/2 + n -1) \}.
\end{equation}

Recall the set $V(S)\subset \Stab^\dagger(S)$ from \eqref{eq:stab_v} mapping homeomorphically onto $\mc{L}(S)$ by $\mc{Z}$ defined in Theorem \ref{thm:bridgeland_cmp}. 
Fix some $n$ and let $N> 2n$, $\xi> 0$ and $\mc{C}$ a contractible collection of $-2$ curves,  and take $V$ from Lemma \ref{lem:stab_limit_gaps}, and the space $U_{\mc{C}}$ of $(\beta, \omega)$ from \eqref{eq:Ulimit}.  Let $v = (1,0,1-n)$.  Define the set 
\[U_{Mov}\subset U_{\mc{C}}\]
consisting of $(\beta, \omega)$ with $\beta, \omega \in NS(S)_\R$ such that 
\begin{itemize}
\item $(\beta, \omega) \in U_{\mc{C}}$
\item For every $C\in \mc{C}$, which is an irreducible -2 curve, the point $(\beta, \omega)$ lies on the side of the wall $W_{v,w}$ for $w = (0, C, 0) = v(\mc{O}_C(-1))$ such that in the stability condition $\sigma_{\beta, \omega}\in V(S)$ with $\mc{Z}(\sigma_{\beta, \omega}) = e^{\beta + i\omega}$ the phases $\phi$ of $\mc{O}_C(-1)$ and $\mc{I}_Y$ for $Y$ of finite length $n$ satisfy
\[ \phi(\mc{O}_C(-1)) \ge \phi(\mc{I}_Y).\]
\end{itemize}

The next proposition will describe these sets as inverse images of certain sets in $\Pos(S^{[n]})$ under the map $\ell: \Stab^\dagger(S) \to \Pos(S^{[n]})$. To this end, for a contractible collection $\mc{C}$ let $H_{\mc{C}}$ denote a big and nef divisor on $S$ inducing a contraction $S_{Bd}$ of $S$ with exceptional locus $\mc{C}$. Thus $\tilde{H}_{\mc{C}}\in \NS(S^{[n]})$ is a big movable divisor corresponding to a contraction $ S^{[n]} \to Sym^n(S_{Bd})$.

\begin{prop}\label{prop:walls_bstab_hilb}
Take $n>1$, a contractible collection $\mc{C}$ of -2 curves on $S$, and Mukai vector $v= (1,0,1-n)$ and sets $U_{Mov}, U_{\mc{C}}$ as defined above.  
\begin{itemize}
  \item[(i)] The image of $U_{\mc{C}}$ under the map $\ell_v$ of \eqref{eq:lstab} consists of divisors $D\in \Pos(S^{[n]})$ such that $D\cdot |2, 1^{n-1}\rangle \ge 0$ and $D$ lies arbitrarily close to the ray spanned by  $\tilde{H}_{\mc{C}}$.
  \item[(ii)] The image of $U_{Mov}$ under the map $\ell_v$ consists of divisors $D$ satisfying the same conditions and also lying in $\overline{Mov}(S^{[n]})$.
  \item[(iii)] The walls $W\subset U_{\mc{C}}$ which are $\mc{Z}(\sigma)$ for $\sigma\in V(S)$ inducing birational contractions of $M_\sigma(v)$ are the loci 
\[\{\beta\cdot \omega = 0\}\cap U_{\mc{C}} \]
and  $W_{v,w}\cap U_{\mc{C}}$ for 
\[w \in \{ (0,C,k) \mid C\in \Z\mc{C},~ C^2 = -2,~ k = 0, -1, \ldots, 1-n \}.\]In other words they are 
\[ \{W_{v, w} \cap U_{\mc{C}} \mid w = v(\mc{O}_S) \text{ or } w =  (0, C, k) \}\]
for the same conditions on $C$ and $k$. 
\item[(iv)] Under the map $\ell_v$, the wall $W_{v,w}\cap U_{\mc{C}}$ for $w = (0, C, k)$ is sent to the wall 
\[\theta_v(v^\perp \cap (0, C, k)^\perp)\]
in $\Pos(S^{[n]})$. The wall $\{\beta\cap \omega = 0\}\cap U_{\mc{C}}$ is sent to the wall 
$|2, 1^{n-1}\rangle^\perp\cap \Nef(S^{[n]})$ inducing the Hilbert-Chow contraction. 
\end{itemize} 
\end{prop}
\begin{proof}
  Proposition \ref{prop:explicit_stab_to_pos} implies that the class $\ell(\sigma_{\beta, \omega})\in \Pos(S^{[n]})$ corresponding to $\sigma_{\beta, \omega} \in V(S)$ with central charge $\mc{Z}(\sigma) = e^{\beta + i\omega}$ is a positive multiple $\eta$ of $\theta_v(r_{\omega, \beta}, C_{\omega, \beta}, s_{\omega, \beta})$ with 
  \begin{align*}
     r_{\omega, \beta} &= - \beta\cdot \omega\\
  C_{\omega, \beta} &= -(\beta \cdot\omega)\beta + \left(1-n + \frac{\beta^2 - \omega^2}{2}\right)\omega\\
  s_{\omega, \beta} &= (1-n)\beta\cdot \omega
  \end{align*}
  or in the notation of \eqref{eq:ns_direct_sum}  and \eqref{eq:tilde_sign}
  \[ \frac 1 \eta \ell(\sigma_{\beta, \omega}) = (\beta\cdot \omega) \tilde{\beta} + \left(\frac{\omega^2 - \beta^2}{2} + n-1\right)\tilde{\omega} + \beta\cdot\omega |2, 1^{n-1}\rangle.\]
  The fact that we have upper bounds on the sizes all coefficients in this expression except for $\omega^2$, together with the $\beta\cdot\omega \le 0$ condition implies that for large enough $V = \omega^2$ this class is equivalent under rescaling by a positive multiple to one of the form 
  \[\tilde{H} -\epsilon |2, 1^{n-1}\rangle\]
  for arbitrarily small $\epsilon \ge 0$. The conditions on $\omega\cdot C$ for curves $C\in \mc{C}$ and $C\notin \mc{C}$ implies that $H$ can be any class which is arbitrarily close to $H_{\mc{C}}$, the big and nef divisor on $S$ which contracts $\mc{C}$. This proves (i). Since $\beta\cdot \omega = 0$  corresponds to $\epsilon = 0$ this proves the last line of (iv). 

  Recall that the walls passing through the point $x = \tilde{H_{\mc{C}}}$ are described by Lemma \ref{lem:walls} (see also Proposition \ref{prop:stab_corner_arranged} when the collection $\mc{C}$ is connected).  They are arranged according to the walls 
  \[ \mc{W} = \{ \delta^\perp, (m\delta + \alpha)^\perp \mid 0 \le m < n, \alpha \in \Delta_f\}\]
  in the root system of the affine lie algebra corresponding to the finite dimensional lie algebra $\mf{g}$ with root system $\Delta_f$ where $\Delta_f$ has associated (not necessarily connected) Dynkin diagram corresponding to the Dynkin diagram of the collection $\mc{C}$. Then those points also in the movable cone coincides with locus of $x$ where $\delta\cdot x \ge 0$, $\alpha_i \cdot x\ge 0$ for $\alpha_i$ simple. 

  Now consider the dual graph $\Gamma_{\Stab}$ of the chambers in the wall-and-chamber decomposition of $U_{\mc{C}}$ for the vector $v$, where chambers are adjacent if they share a codimension 1 face. Also consider the dual graph $\Gamma_{\Pos}$ of chambers in $\{ D\in \Pos(S^{[n]})\mid D\cdot |2, 1^{n-1}\rangle \ge 0\}$ adjacent to the point $x = \tilde{H_{\mc{C}}}$, where again chambers are adjacent if they share a codimension 1 face. Then Theorem \ref{thm:bm_mmp} implies that 
  \[\Gamma_{\Stab} \simeq \Gamma_{\Pos}.\]

  We can also explicitly match the walls. 
Using the formula for $ \ell(\sigma_{\beta, \omega})$, it is possible to calculate $\ell(\sigma_{\beta, \omega})\cdot \theta_v((0, C, k))$ for $C \in \Z\mc{C}$ with $C^2 = -2$ and $k \in \Z$. Namely up to the positive multiple $\eta$ which has no effect
\begin{align*}
  \ell(\sigma_{\beta, \omega})\cdot \theta_v((0, C, k)) &= ((r_{\omega, \beta}, C_{\omega, \beta}, s_{\omega, \beta}), (0, C, k))\\
  &= -(\beta\cdot \omega) (\beta\cdot C) + \left(1-n + \frac{\beta^2 - \omega^2}{2}\right)\omega\cdot C + k\beta\cdot\omega
\end{align*}
and comparing this with the formula \eqref{eq:wall_vcrv} for $W_{v,(0,C,k)}$ we find that the two walls coincide, proving (iii) and (iv). 

  Now note that we can match the boundaries \[\{\beta\cdot\omega = 0\}\cap U_{\mc{C}}, W_{v, (0,C,0)}\cap U_{\mc{C}}\] of $U_{\Mov}$ with the divisorial wall boundaries 
  \[ \{|2,1^{n-1}\rangle^\perp , \tilde{C}^\perp \mid C\in \mc{C}\} \]
  of the image $\ell(U_{\Mov})\subset \Pos(S^{[n]})$, and since the Gieseker chamber of $S^{[n]}$ intersects this image, $\ell(U_{\Mov})$ coincides with the intersection of $\ell_v(U_{\mc{C}})$ with $\overline{\Mov}(S^{[n]})$, showing (ii).
\end{proof} 

\subsection{Other rank 1 torsion-free sheaves}
We need to relate moduli spaces of torsion free sheaves $\mc{F}$ with different values of $c_1(\mc{F})$ to capture the entire Fock space, and also birational models for these moduli spaces. 
The condition \eqref{eq:stab_limit_small_exc} describing the limit is also designed to make it easy to compare the stability conditions for ideal sheaves $\mc{I}_Y$ of length $n$ zero dimensional subschemes and torsion free subsheaves of $\mc{L}$ where $\mc{L}$ is a line bundle with $c_1(\mc{L}) \in \Z \mc{C}$ for a contractible collection $\mc{C}$. 

If $\mc{L}_Y$ is a torsion-free subsheaf of a line bundle $\mc{L}$ with quotient $\mc{O}_Y$ of length $n$ then tensoring with $\mc{L}^{-1}$
\[ \ses{\mc{L}_Y}{\mc{L}}{\mc{O}_Y}\]
recovers
\[\ses{\mc{I}_Y}{\mc{O}_S}{\mc{O}_Y}\]
inducing the isomorphism between 
$M(1, c_1(\mc{L}),1 + c_1(\mc{L})^2/2 - n)$ and $S^{[n]}$ 
seen in \eqref{eq:hilb_eq_m}. This tensor product does not preserve the Bridgeland stability condition, even though Gieseker stability is the same for all Mukai vectors $(1, c_1, 1 + ch_2)$.  At least in the relevant limit, the modification to Bridgeland stability given by tensoring with $\mc{L}$ is tracked by the shift
\[ \beta \mapsto \beta - c_1 (\mc{L}).\]
We begin with the calculation on the central charge.
\begin{lem}\label{lem:cc_shift_c1}
  For $(\beta, \omega)$ corresponding to a stability condition in $V(S)$, and any $\mc{E} \in D(S)$, the action of tensoring with a line bundle on the central charge is given by the formula
  \[Z_{\beta, \omega} (\mc E\otimes \mc{L}) = Z_{\beta - c_1(\mc{L}), \omega}(\mc E).\]
\end{lem}
\begin{proof}
  For any object $\mc{E}\in D(S)$  we have
   \begin{align*}
     Z_{\beta, \omega}(\mc{E}\otimes \mc{L}) &= (e^{\beta + i\omega},v(\mc{E}\otimes \mc {L}) )\\
     &= -\int_S e^{-(\beta + i\omega)} ch(\mc{E}\otimes \mc{L})\sqrt{\on{Td}(S)}\\
     &= -\int_S e^{-(\beta + i\omega)}e^{c_1(\mc{L})} ch(\mc{E})\sqrt{\on{Td}(S)}\\
     &= (e^{\beta - c_1(\mc{L}) + i\omega}, v(\mc{E}))\\
     &= Z_{\beta - c_1(\mc{L}), \omega}(\mc{E}).
   \end{align*}
\end{proof}
We also identify the wall and chamber structure in the relevant region of $\Stab^\dagger(S)$. The following is the analogue of Proposition \ref{prop:affine_quiver_shift}. To this end, we need a generalization of the set $U_{\mc {C}}$ more applicable to Mukai vectors of the form $(1, D, s)$. 

\begin{defn}\label{defn:ucd} Let $D \in \NS(S)$ be a divisor. 
The set $U_{\mc C, D}$ depending on the same data $N, \xi, V$ as $U_{\mc C}$ is defined to be the subset 
\[ U_{\mc C, D} = \{(\beta, \omega) \in U_{\mc C} \mid -\beta \cdot \omega >0 \text{  and  } (D - \beta)\cdot \omega > 0  \}\]
of points in $U_{\mc C}$ satisfying these additional bounds. 
\end{defn}

\begin{prop}\label{prop:wall_chamber_vd}
  Let $\sigma_{\beta, \omega}$ be a generic stability condition in $V(S)$ (in particular, $\omega$ is ample) lying above a  point $(\beta, \omega) \in U_{\mc{C},D}$
     where $D = c_1(\mc{L})$ is $c_1$ of any line bundle on $S$. Let $v_{D} = (1, D,1 + D^2/2 - n )$.  Let $W_{v,w}$ for Mukai vectors $v$ and $w$ denote the locus where $Z_{\beta, \omega}(v)$ and $Z_{\beta, \omega}(w)$ are real positive multiples of each other.  Then after potentially increasing $V$  (which is the lower bound on $\omega^2$) and the other data defining $U_{\mc{C}}$, $\sigma_{\beta, \omega}$ lies on a wall for Mukai vector $v_{D}$ if and only if $\sigma_{\beta, \omega}$ lies on one of the walls $W_{w,v_{D}}$ for
      \[ w \in \{ (0, C, C\cdot D + k) \mid C\in \Z\mc{C},C^2 = -2, k = 0,-1 \ldots,1-n  \}.\]
\end{prop}
\begin{proof}
If $\Phi_{\mc{L}}$ is the derived equivalence given by tensoring with $\mc{L}$ then the pullback stability condition $\Phi_{\mc{L}}^*\sigma_{\beta, \omega}$ lies in the Gieseker chamber for $v = (1,0,1-n)$ if and only if $\sigma_{\beta, \omega}$ lies in the Gieseker chamber for $v = v_D$. 
Also, given $(\beta, \omega)\in \mc{L}(S)$ satisfying $-\beta \cdot \omega >0 \text{  and  } (D - \beta)\cdot \omega > 0 $, $(\beta, t\omega)$ is in $\mc{L}(S)$ and $\sigma_{\beta, t\omega}$ is in the Gieseker chamber for both $v = (1,0,1-n)$ and $v = v_D$ for all $t \gg 0$.
 But since the  central charge of the pullback is given by $\Phi_{\mc{L}}^*\sigma_{\beta, \omega}$ a shift of $\beta$ by $D$ and the pullback preserves the set $V(S)$, it follows that after potentially increasing $V$ we have that  $\Phi_{\mc{L}}^*\sigma_{\beta, \omega}\in U_{\mc{C}}$ and so the walls on which $\Phi_{\mc{L}}^*\sigma_{\beta, \omega}$  might lie are a subset those described in Proposition \ref{prop:walls_bstab_hilb}. 
 On the other hand, $\Phi_{\mc{L}}^*\sigma_{\beta, \omega}$ cannot lie on the wall $\beta\cdot\omega = 0$ since this is the wall where $\mc{O}_S$ and $\mc{I}_Y$ have the same phase for $\Phi_{\mc{L}}^*\sigma_{\beta, \omega}$ , but the conditions $-\beta \cdot \omega >0 \text{  and  } (D - \beta)\cdot \omega > 0 $ and the calculations of Table \ref{tbl:slopes_z} show that for the stability condition $\sigma_{\beta, \omega}$ the phases of $\mc{L}$ and $\mc{L}_Y$ do not agree. 

Thus the potential walls for Mukai vector $v_D$ on which $\sigma_{\beta, \omega}$ may lie are a subset of those $W_{w,v_{D}}$ for  
\[w \in \{ \Phi_{\mc{L}, *}((0, C, k)) \mid C\in \Z\mc{C},C^2 = -2, k = 0,-1 \ldots,1-n  \}\].
Where $\Phi_{\mc{L}, *}$ denotes the action of the equivalence $\Phi_{\mc L}$ on Mukai vectors, which acts by multiplication with $\on{ch}(\mc{L})$, i.e. 
\begin{align*}
  \Phi_{\mc{L},*} : (r, C, s) &\mapsto (r, C, s)(1, c_1(\mc{L}), c_1(\mc{L})^2/2)\\ 
  &= (r, C + r c_1(\mc{L}), c_1(\mc{L})\cdot C + s + rc_1(\mc{L})^2/2)
\end{align*}
so that a destabilizing sequence 
\[ \ses{\mc{G}}{\mc{F}}{\mc{O}_C(-k)}\]
for a semistable object $\mc{F}$ with Mukai vector $(1,0,1-n)$ is sent to a destabilizing sequence 
\[ \ses{\mc{G}\otimes \mc{L}}{\mc{F}\otimes \mc{L}}{\mc{O}_C(-k)\otimes \mc{L}}.\]
Thus the potential destabilizing walls are $W_{w,v_D}$ are those for
\[ w\in \{(0, C, C\cdot D + k)\mid C\in \Z\mc{C},C^2 = -2, k = 0,-1 \ldots,1-n   \}.\]

But they are all actually possible since again looking at Table \ref{tbl:slopes_z} we can find $\beta, \omega$ where we do lie on these walls, since for $w = (0, C, \ell)$
\begin{multline}
  \label{eq:Wvd_w}
  W_{w,v_D}\cap U_{\mc{C}} = \{ (\beta, \omega) \in U_{\mc{C}} \mid (C\cdot \beta - \ell)((D - \beta)\cdot \omega ) =  \\ (C\cdot \omega)(-1+(\omega^2-\beta^2)/2 + D\cdot \beta - D^2/2+n)\}.
\end{multline}
\end{proof}

\subsection{Figures for generic elliptic K3 surface}\label{sec:figures}
To clarify the results of this section we include figures of the positive cone of $S^{[n]}$ with its wall and chamber structures where $S$ is a generic elliptic K3 surface with section. This walls near the point corresponding to $Sym^nS_{bd}$ on the surface where we contract the section will coincide with the structure around a single $A_1$ collection in any K3 surface. 

\begin{figure}[h!]
\centering
\includegraphics[width=\linewidth]{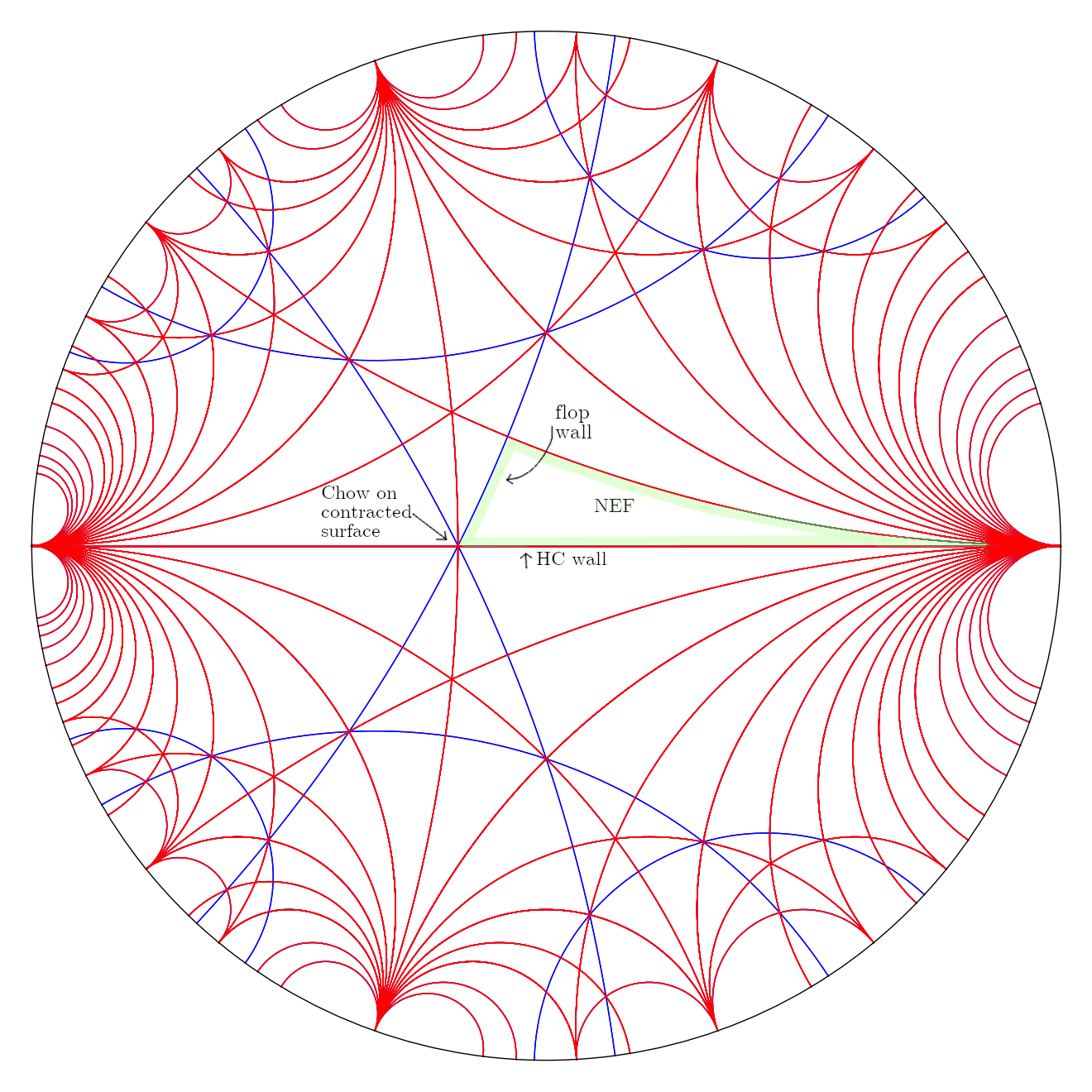}
\caption{The wall and chamber structure on $\Pos(S^{[n]})$ with the Nef cone, Hilbert-Chow wall, and a ray corresponding to $ Sym^nS_{bd}$ labelled. Red walls correspond to walls of the Movable cone or its reflections, and blue walls correspond to higher codimension contractions. }
\end{figure}

\begin{figure}[h!]
  \begin{subfigure}{.5\textwidth}
    \centering
    \includegraphics[width=.9\linewidth]{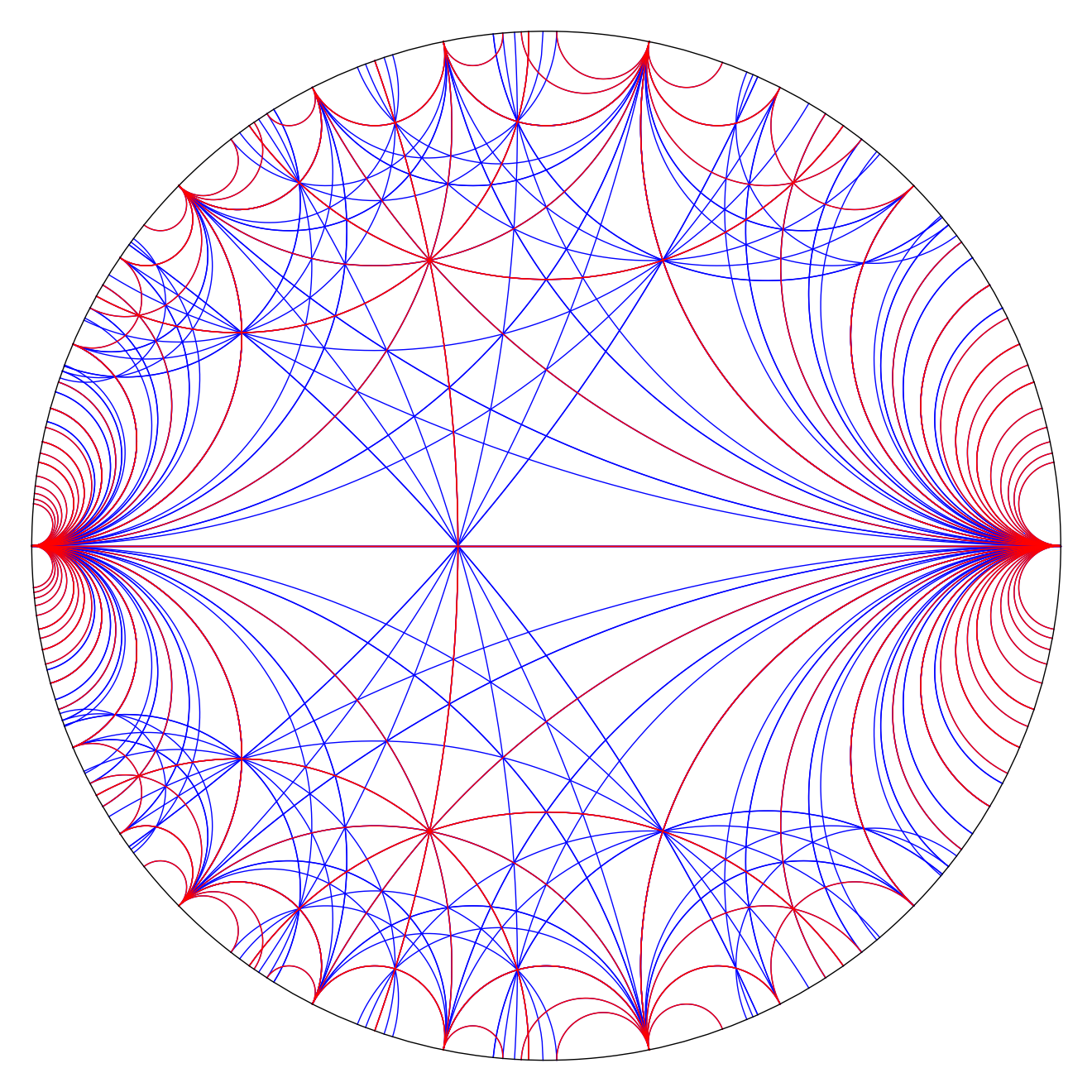}
  \end{subfigure}%
  \begin{subfigure}{.5\textwidth}
    \centering
    \includegraphics[width=.9\linewidth]{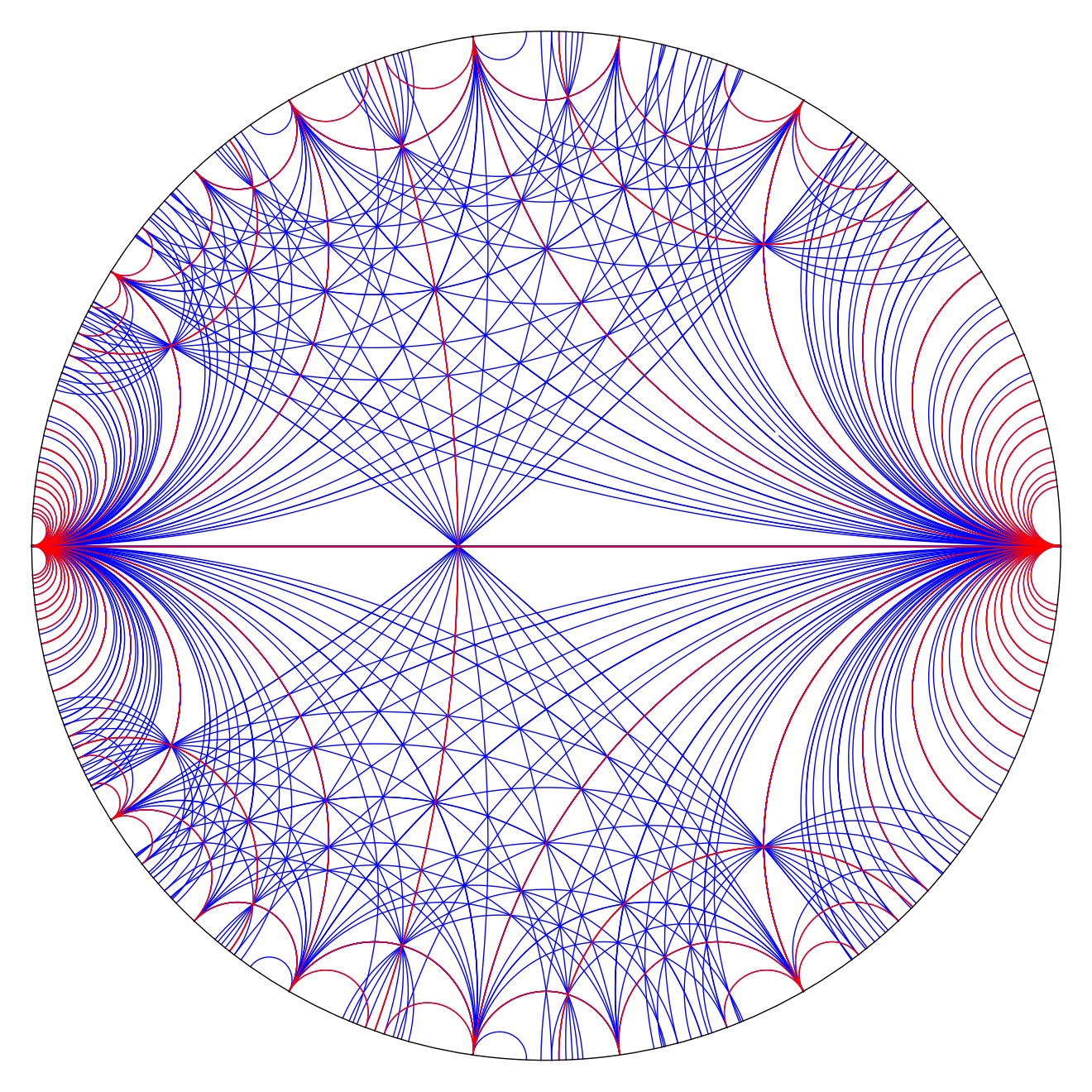}
  \end{subfigure}
  \begin{subfigure}{.5\textwidth}
    \centering
    \includegraphics[width=.9\linewidth]{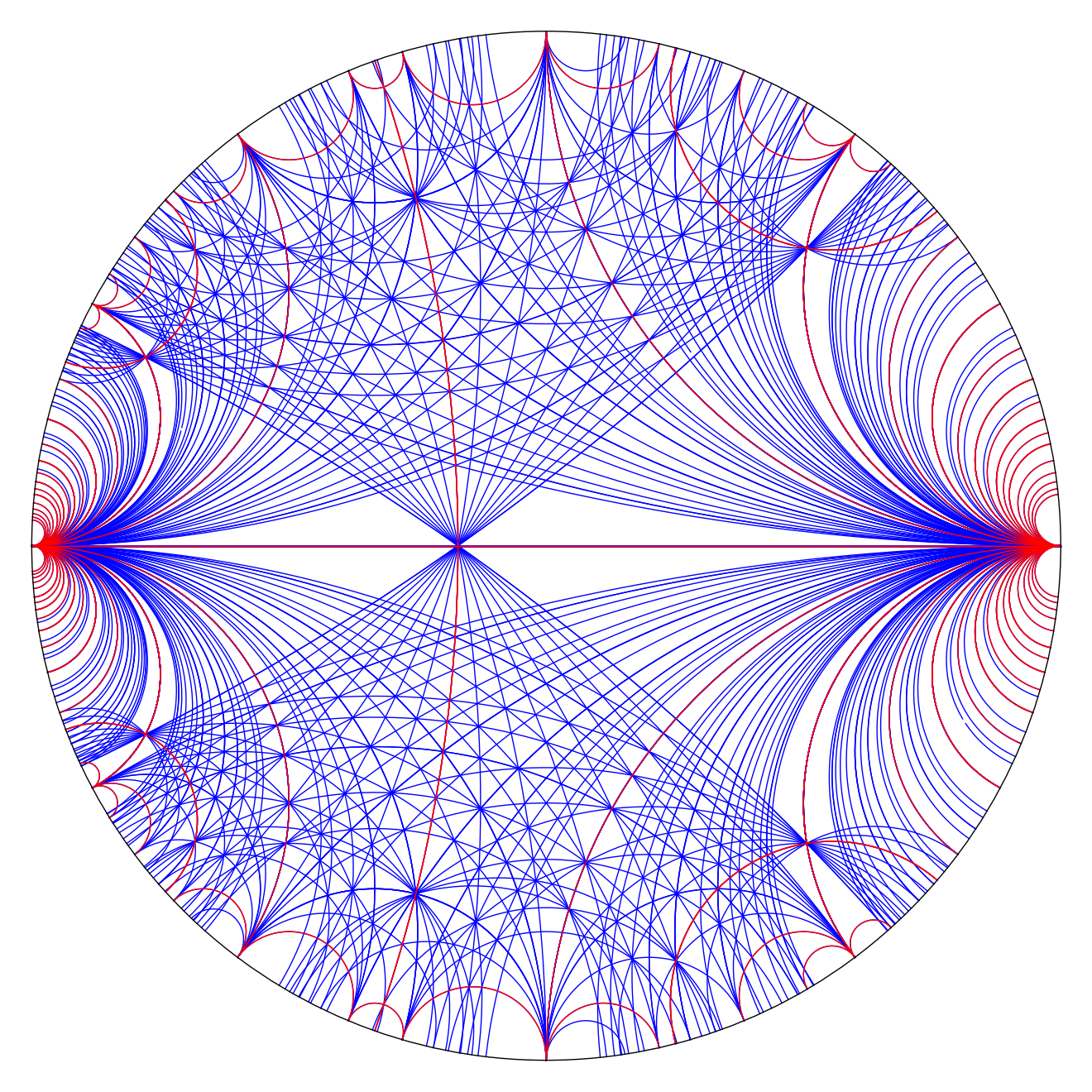}
  \end{subfigure}%
  \begin{subfigure}{.5\textwidth}
    \centering
    \includegraphics[width=.9\linewidth]{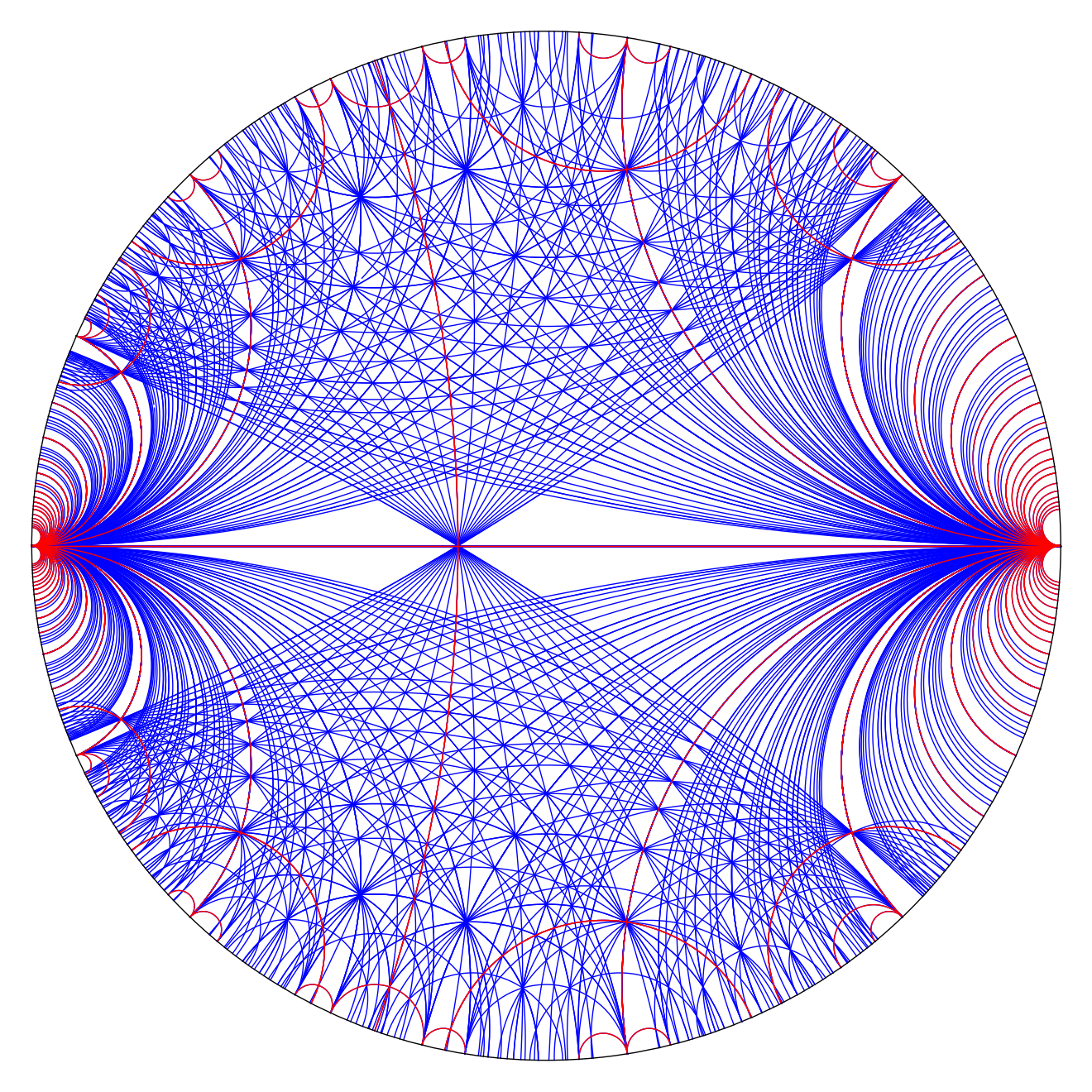}
  \end{subfigure}
  \caption{Positive cones of $S^{[n]}$ for $n \in \{ 4, 7, 9, 12\}$.}
\end{figure}

\newpage
\section{Construction of action}\label{sec:construction}
In this section we construct the action of the universal enveloping algebra of the affinization of a Lorentzian Kac-Moody algebra on the cohomologies of moduli spaces of rank 1 torsion free sheaves on $S$, which is of the same form as in Section \ref{sec:stabk3_specific}.

\paragraph{Stability chambers for smooth moduli spaces}
For the quiver variety case, the natural stability chamber to produce the action consisted of 
$\theta = (\theta_0, \theta_1, \ldots, \theta_r)$ such that $\theta_i > 0$ for all $i$. From the present vantage point, this chamber is the most useful because every wall $\{ \theta \cdot \dim_{gr} S_i = 0\}$ is a wall of this chamber,  where $S_i$ a 1 dimensional quiver representation. The analogous chamber for moduli spaces of objects on K3 surfaces will consist of a chamber $\mc A$ of stability conditions $\sigma_{\beta, \omega}$ such that for every irreducible curve $C\in \mc{C}$ there is a wall on the boundary of $\mc A$ such that on this wall the phases of the torsion-free sheaf $\mc L_Y$ and the sheaf $\mc{O}_C(-1)$ overlap. The following notation described a chamber in $\Stab^\dagger(S) \cap U_{\mc C}$ which is mapped to a chamber in $\Pos(S^{[n]})$ which corresponds to the fundamental alcove of the affine root system under the equivalence of such chambers with alcoves given by Proposition \ref{prop:stab_corner_arranged}. 

\begin{notn}\label{not:mca}
Fix a connected contractible collection $\mc{C}$ of $-2$ curves on
 $S$ and fix some class $D \in \NS(S)$, and Mukai vector $v_D= (1, D,1 + D^2/2 - n )$.
Define the set $\mc{A}_D \subset V(S)$ of stability conditions $\sigma = \sigma_{\beta, \omega}$ such that 
\begin{enumerate}[(i)]
  \item The pair $(\beta, \omega) \in U_{\mc C, D}$ where $U_{\mc C, D}$ is defined in Definition \ref{defn:ucd}. Thus $\mc{A}_D$ depends on the same data as the set $U_{\mc{C}, D}$.
   \item Under $\sigma \in \mc{A}_D$ the phases of a $\sigma$-stable object $\mc{E}\in \mc{P}(0, 1]$ of Mukai vector $v_D$ and objects $\mc{O}_C(k)$ for $k = -1, \ldots, 1-n$ and an irreducible curve $C\in \mc{C}$ satisfy
   \[\phi(\mc{O}_C(-1)) > \phi(\mc{E}) > \phi(\mc{O}_C(-2)) > \cdots > \phi(\mc{O}_C(1-n)). \]
   \item Under the wall and chamber decomposition with respect to the Mukai vector $v_D$ of the subset of $(\beta, \omega)$ in  $U_{\mc{C}}$ satisfying condition (ii), $\sigma$ lies in the unique chamber which for every $C\in \mc{C}$  this chamber has a point on its boundary where we have $\phi(\mc{O}_C(-1)) = \phi(\mc{E})$ for $\mc{E}$ as above. 
\end{enumerate}
\end{notn}

 Under Proposition \ref{prop:wall_chamber_vd}, condition (iii) in the above definition can be replaced by the condition that with respect to a stability condition $\sigma\in \mc{A}_D$ the phases of stable objects $\mc{E}$ and $S$ satisfy $\phi(\mc{E}) > \phi(S)$  where $\mc{E}$ is as above and $S$ is an object in $\mc{P}(0,1]$ of Mukai vector $(0, C_h, -1)$ where $C_h \in \Z \mc{C}$, $C_h = \sum a_i C_i$ with $C_h^2 = -2$ is the class of the highest root under the identification of $\mc{C}$ with the set of simple positive roots of a finite root system. This is the analogue of the condition that $\theta_0 > 0$ in the case of the positive chamber for affine quiver varieties. 

 It will be useful to have a more general description of a chamber in $U_{\mc{C}, D}$ corresponding to a different alcove of the affine reflection group associated to $\mc{C}$, such that $\mc{A}_D$ will correspond to the usual fundamental alcove. 
To this end, consider the finite root system $\Delta$ of rank $r$ associated to $\mc C$ with simple roots $\alpha_i$, positive roots $\alpha\in \Delta_+$. Note that to each alcove $A$ of the corresponding affine reflection group on $\R^r$ (which we might think of as the level 1 hyperplane) we can associate a sequence of numbers $\mb k = (k_\alpha)_{\alpha\in \Delta_+}$ such that for $x\in A$ we have 
\[k_\alpha < \alpha(x) < k_{\alpha} + 1 ~\text{ for all }\alpha\in \Delta_+ .\]
Then the sequence $\mb k$ uniquely determines the alcove (although not every sequence corresponds to an alcove). Using this observation, and Proposition \ref{prop:wall_chamber_vd} which implicitly identifies chambers in $U_{\mc{C}, D}$ with alcoves for the affine reflection group we extend the previous notation. In the following, writing $\phi(v)$ for a Mukai vector $v$ means $\phi(\mc{E})$ for a semistable object in $\mc{P}(0,1]$ of Mukai vector $v$. 

\begin{notn}\label{not:adsk}
  Fix a contractible collection $\mc{C} = \{ C_1, \ldots, C_r\}$ of $-2$ curves on
  $S$ with class $C_\alpha \in \Z\mc C$ corresponding to the positive root $\alpha\in \Delta_+$. Fix some class $D \in \NS(S)$, and Mukai vector $v_D= (1, D,s)$, and consider the set $U_{\mc C, D}$ as above. Given a vector $\mb k = (k_\alpha)_{\alpha\in \Delta_+}$ define the (possibly empty) set $\mc{A}_{D,s, \mb{k}}\subset V(S)$  of stability conditions $\sigma = \sigma_{\beta, \omega}$ such that 
  \begin{itemize}
\item We have $(\beta, \omega) \in U_{\mc{C}, D}$. In particular, $\mc A_{D,s, \mb k}$ depends on everything $U_{\mc{C}}$ does. 
\item  Under $\sigma \in \mc{A}_{D, s,\mb k}$ and for any $\alpha\in \Delta_+$ we have 
\[ \phi((0, C_\alpha, k_\alpha -1)) < \phi(v_D) < \phi((0, C_\alpha, k_\alpha)).\]
  \end{itemize}
\end{notn}
When $s$ is not specified, but there is some integer implicit integer $n$, we take Mukai vector $(1, D,1 + D^2/2 - n )$ and define $\mc{A}_{D,\mb k } := \mc{A}_{D,s,\mb{k} }$ for $s = 1 + D^2/2 - n $. Whenever bold-faced $\mb k$ is not specified we take $\mb k = \mb 0 = (0, 0,\ldots, 0)$.
This is a generalization of the previous notation, since for $\mb k = \mb 0 = (0, 0,\ldots, 0)$ we have $\mc{A}_D = \mc{A}_{D, \mb 0}$.  
Also, note that for a given $n$ (= number of points) and data defining $U_{\mc{C}}$ the set $\mc{A}_{D, \mb k}$ may not be an the intersection of a stability chamber for Mukai vector $v_D$ and $U_{\mc{C}, D}$. But as $n$ and the bounds on $\beta, \omega$ defining $U_{\mc{C}}$ increase every $\mc{A}_{D, \mb{k}}$ is one of these chambers for any $\mb k$ corresponding to an alcove.

\subsection{Stability conditions relating different Mukai vectors}
\label{ssec:different_mukai_vectors}
In order to produce Steinberg correspondences between different smooth moduli spaces, we need to include one singular moduli space into another to parallel the quiver variety case.

To be precise, for stability conditions $\sigma_D \in \mc{A}_D\cap  \mc{A}_{D + C_i}$ the action of $e_i$ mapping $H^*(M_{\sigma_D}(1,D, s))$ to $H^*(M_{\sigma_{D}}(1, D+ C_i, s))$ will be given by a convolution by a Lagrangian correspondence which is an irreducible component of $M_{\sigma_D}((1,D, s)\times_M M_{\sigma_{D}}(1, D+ C_i, s)$ where $M$ is a certain singular moduli space, which we can think of as a stratum of $M_{\sigma}(1, D+ \infty C_i, s)$ for a specifically chosen stability condition $\sigma$ which must be on the boundary of both $\mc{A}_D$ and $ \mc{A}_{D + C_i}$. 

\paragraph{Boundary stability conditions}
There will be some freedom in the choice of non-generic stability condition because we are not taking the fiber product over the analogue of the affine quotient, but rather, a moduli space for a lower-codimension face in the space of stability conditions. Here we describe this face. Pick a base 
integer $s\in \Z$ and $D \in \NS(S)$, a contractible collection $\mc{C}$ of $-2$ curves, and a curve $C_i \in \mc{C}$. Let $U_{\mc{C}, D}$ be the set of $(\beta, \omega)$ defined in Definition  \ref{defn:ucd}.

Now we choose a specific stability condition corresponding to a point in this set. Let 
\begin{equation}
\label{eq:sigma_Di}
\sigma_{D,s,i} = \sigma_{\beta, \omega} \text{ for } (\beta, \omega)\in U_{\mc{C}, D}
\end{equation}
be a stability condition that for all $\ell\in \Z$ such that $(1, D + \ell C_i, s)^2 \ge -2$ we have $\sigma_{D,s,i}\in \partial \mc{A}_{D + \ell C_i,s}$, and further that for each $\ell$ the condition $\sigma_{D,s, i}$ lies generically on the wall of $\mc{A}_{D + \ell C_i,s}$ such that the phase of $\mc{O}_{C_i}(-1)$ overlaps with that of an object of Mukai vector $(1, D + \ell C_i, s)$. We will soon show this exists.

It will actually be necessary to define a more general stability condition where an adjacent generic stability condition, instead of lying in the chamber $\mc{A}_{D,s}$ will lie in $\mc{A}_{D,s,\mb k}$ and our stability condition, instead of lying on a single wall will lie the intersection of several. In the following, we continue  writing $\phi(v)$ for a Mukai vector $v$ means $\phi(\mc{E})$ for a semistable object in $\mc{P}(0,1]$ of Mukai vector $v$. 

\begin{notn}\label{not:boundary_stab}
Fix $D\in \NS(S)$, integer $s\in \Z$, contractible collection $\mc{C} = \{C_1, \ldots, C_r\}$ of $-2$ curves, sequence $\mb{k}$ of integers corresponding to an alcove $A$, and some subset $R \subset \Delta_+$ with of positive roots such that there is a point $x\in \partial A$ on the boundary of the alcove such that $\alpha(x) = k_{\alpha}$ for $\alpha \in R$. Let $\mc{C}_R\subset \Z \mc{C}$ denote the set of $C_{r}\in \Z\mc{C}$ corresponding to $r\in R$. Then let
\begin{equation}
  \label{eq:sigma_DsRk}
  \sigma_{D,s,R,\mb{k}} = \sigma_{\beta, \omega} \text{ for } (\beta, \omega)\in U_{\mc{C}, D}
  \end{equation}
  be a stability condition such that 
  \begin{itemize}
\item  Given $\mb{t}\in \Z^R$ let $D_{\mb t} = D + \sum_{r\in R} t_r C_r$ and 
$s_{\mb t} = s + \sum_{r\in R} {t_r k_r}$. For all $\mb{t}\in \Z^R$
such that 
\[(1,D_{\mb t} ,s_{\mb t} )^2 \ge -2\]
 we have $\sigma_{D,s,R, \mb{k}}\in \partial \mc{A}_{D_{\mb t},s_{\mb t}, \mb{k}}$.
 \item Further, for any such $\mb{t}$, the stability condition $\sigma_{D,s,R, \mb{k}}$ lies generically on the intersection over $r\in R$ of the walls $W_r$ of $\mc{A}_{D_{\mb t},s_{\mb t}, \mb{k}}$ where the phase $\phi((0, C_r, k_r))$ of the unique stable object of Mukai vector $(0, C_r, k_r)$ overlaps with $\phi((1, D_{\mb t},s_{\mb t}))$. 
  \end{itemize}
\end{notn}

\begin{lem}
  \label{lem:sigma_DI_exists}
For large enough data $N$ and $V$ defining $U_{\mc{C}, D}$ there exists such a $\sigma_{D,s,i}$ satisfying the above conditions. More generally given the data in Notation \ref{not:boundary_stab} there exists a $\sigma_{D,s,R,\mb{k}}$ satisfying the required conditions. 
\end{lem}
\begin{proof}
  This is not the most immediate way to prove this result, but it makes the geometry of the space of stability conditions clear. We first prove the existence of $\sigma_{D,s,i}$ where the notation is simpler. Also,
  to simplify formulas, define
  \[ \lambda_{D, \beta} = D\cdot \beta  - \beta^2/2 -s.\]
Recall that, for example, equation \eqref{eq:Wvd_w} implies that the condition which allows the phase of $\mc{O}_{C_i}(-1)$ to overlap with that of an object of Mukai vector $(1, D, s)$ is
\[(C_i\cdot \beta)((D - \beta)\cdot \omega ) =  (C_i\cdot \omega)(\omega^2/2 + \lambda_{D, \beta}) \]
and the equation which allows the analogous condition for 
Mukai vector $(1, D+ \ell C_i, s)$ is the equation 
\[(C_i\cdot \beta)((D + \ell C_i - \beta)\cdot \omega ) =  (C_i\cdot \omega)(\omega^2/2 + \lambda_{D, \beta} + \ell C_i \cdot \beta ). \]
Then in $U_{\mc{C}, D}$ both of these conditions are equivalent to 
\begin{equation}
\label{eq:ratio_sigma_Dsi}
(D-\beta)\cdot \omega = \frac{(C_i \cdot \omega)(\omega^2/2 + \lambda_{D,\beta})}{C_i \cdot \beta}
\end{equation}
which is independent of $\ell$ and which for fixed $\omega$ and fixed $C_i\cdot \beta$ (and in fact for fixed $C_j\cdot \beta$ for $j \neq i$ as well, meaning we can avoid other walls) we can vary $\beta$ subject to these constraints so that the value of $\beta\cdot \omega$ is adequate to produce a solution to \eqref{eq:ratio_sigma_Dsi} by slightly deforming the constant solution of $\beta\cdot \omega$ to 
\[ (D-\beta)\cdot \omega = \frac{(C_i\cdot \omega)(\omega^2/2)}{C_i \cdot \beta}\]
since $\lambda_{D, \beta}$ is extremely small compared to $\omega^2/2$. For fixed values of $C_i\cdot \beta$ and $(C_i \cdot\omega)\omega^2$ it may be necessary to extend the data $V$ and $N$ in order for this value of $\beta\cdot\omega$ to be compatible with the condition that $(\beta, \omega)$ lie in $U_{\mc{C}, D}$.

Now fix $\mb{k}$ and $R$ as in Notation \ref{not:boundary_stab}. The Mukai vectors 
$(0, C_r, k_r)$ span a negative definite plane in the Mukai lattice therefore there are only a finite number of $\mb t$ such that $(1,D_{\mb t} ,s_{\mb t} )^2 \ge -2$, meaning that we are free to choose $N$ to be large enough to bound all relevant chern characters. Let $C_{\mb t} = \sum_{r\in R} t_r C_r$ and $k_{\mb t} = \sum_{r\in R} t_r k_r$. The equations for the relevant walls are for all $r \in R$ and relevant $\mb t\in \Z^R$
\[(C_r\cdot \beta -k_r)((D  + C_{\mb t} - \beta)\cdot \omega ) =  (C_r \cdot \omega)(\omega^2/2 + \lambda_{D,\beta} + C_{\mb t} \cdot \beta - k_{\mb t}) \]
which by the same argument admit simultaneously a solution if we choose a fixed common ratio $\xi$ with 
\[\frac{C_r \cdot \omega}{C_r \cdot \beta - k_r } = \xi \] 
for all $r \in R$. 
\end{proof}

This is exactly what is needed to relate moduli spaces for different Mukai vectors since, for example, if $\mc{F}$ is a $\sigma_{D,s, i}$-semistable object (e.g. if $\mc{F}$ is $\sigma_D$-stable) of Mukai vector $(1, D, s)$ then
\[ \mc{F}\oplus \mc{O}_{C_i}(-1)^{\oplus \ell}\]
is $\sigma_{D,s,i}$-semistable for any $\ell \ge 0$ of Mukai vector $(1, D + \ell C_i, s)$.

\subsection{Common base of symplectic resolution}

In the quiver variety case, even in the case of non-generic stability parameter, the construction via Hamiltonian reduction and GIT quotient defines moduli spaces of quiver representations.   Thus in the setting of Section \ref{sec:quiver_action} there was no issue defining a common singular base to construct correspondences between different smooth moduli spaces. 

In this situation, for a non-generic stability condition $\sigma_0$ the relevant Step 4 of Construction \ref{const:bridgeland_stable_spaces} does not produce a space of all semistable objects of a given primitive Mukai vector, only the space of S-equivalence classes of objects which are stable for a generic stability condition in an adjacent chamber. This space can be thought of as a specific \emph{stratum} in an (as of yet presently undefined) larger space of all semistable objects of that Mukai vector. 

Rather than actually define a moduli space of semistable objects for non-generic $\sigma_0$, we will simply take an appropriate union of $\pi_{\sigma_+, \sigma_0}( M_{\sigma_+}(v))$, which will essentially consist of defining the largest relevant stratum of the undefined large moduli space of semistable objects for non-generic $\sigma_0$. 

This is slightly nuanced and some care must be taken because it is possible that other spherical classes show up in the lattice spanned by $(1, D, s)$ and $(0, C, k)$ so we cannot simply look at the lattice spanned by $(1, D, s)$ and $(0, C, k)$ to determine the polystable representative of a given semistable object.  We have to look at the set of effective spherical classes in this lattice, which depends not just on the lattice but on the central charge. It is the specific limit under consideration which makes it so that no spherical twists by higher rank spherical objects influences the decomposition of objects of Mukai vector $(1, D, s) + \ell (0, C, k)$.

For stability conditions lying on the intersection of two or more walls,  we will need a slight extension of some of the results in \cite[\S 6, \S 8]{bayer2014mmp} in the present setting. Also recall the definition of an  effective class $u$, with respect to an ambient Mukai vector $v$ and stability condition $\sigma$ with central charge $Z$. This is a class with $u^2 \ge -2$ and $\re Z(u)/Z(v) > 0$. 

\begin{prop}\label{prop:corner_jh}
  Fix a phase $\phi$ so that we are consider moduli spaces of objects of Mukai vector $(1, D, s)$ of phase $\phi$.  Consider $D, s, \mc{C}, \mb{k}, R \subset \Delta_+, \mc{C}_R\subset \Z \mc{C}$ and 
$\sigma_0 := \sigma_{D,s,R,\mb{k}}$ as in Notation \ref{not:boundary_stab}. Let $\sigma_+\in \mc{A}_{D,s,\mb{k}}$ be a generic stability condition in the adjacent chamber.  Let $\mc{H}$ be the hyperbolic lattice spanned by $v_{D} := (1,D,s)$ and $(0, C_r, k_r)$ for $r \in R$. 
\begin{enumerate}[(i)]
\item Let $\mc{R}\subset \mc{H}$ be the rank $\rho = |R|$ negative definite sublattice spanned by $c_i := (0, C_{r_i}, k_{r_i}), i = 1, \ldots, \rho$. This basis $c_1, \ldots, c_\rho$ of $\mc{R}$ is such that there is a $\sigma_{0}$-stable spherical object $\mc{S}$ of phase $\phi$ if and only if the Mukai vector of $\mc S$ is $c_i$ for some $i$, and if so this object is unique.  Let $\mc{S}_1, \ldots, \mc{S}_\rho$ be the corresponding spherical objects. 
\item Let $\mb{t} \in \Z^\rho$ and let $c_{\mb{t}} = \sum_{i = 1}^\rho t_ic_i$. If $(v_{D} + c_{\mb t})^2 \ge -2$ then the map 
\[\pi_{\sigma_+, \sigma_0} : M_{\sigma_+}(v_{D} + c_{\mb t}) \to M_0 \]
sending a $\sigma_+$ stable object to its $S$-equivalence class under $\sigma_0$ has target consisting entirely of strictly semistable objects if and only if $\langle v_{D,s}, c_i \rangle < 0$  for some $i$.
\item The Jordan-H\"older factors of a $\sigma_+$-stable object of Mukai vector  $v_{D,s} + c_{\mb t}$ consist of $\sigma_0$-stable objects of Mukai vector $c_i$ or those lying in the set 
\[\{v_{D} + c_{\mb {t'}} \mid \mb{t'}\in \Z^\rho, (v_{D} + c_{\mb {t'}})^2 \ge -2, \langle (v_{D} + c_{\mb {t'}}),  c_i \rangle \ge 0 \text{ for all } i\}.\]
\end{enumerate}
\end{prop}
\begin{proof}
  To prove (i), first note that Mukai vectors of stable spherical objects of phase $\phi$ must span $\mc R$. This is because we can deform to a stability condition where these phases don't overlap with $v_{D}$ and so the Jordan-H\"older factors of $\sigma_+$-stable objects with Mukai vectors in $\mc{R}$ have Mukai vectors in $\mc{R}$. But then $\mc{R}$ has a basis of spherical objects and the Jordan-H\"older factors of spherical objects are spherical \cite[\S 2]{huybrechts_stability_2008} (also see \cite[Lemma 6.2]{bayer2014mmp}.) Let $\{c_1, \ldots, c_n\}$ be the spanning set of Mukai vectors of these spherical objects $\mc{S}_1, \ldots, \mc S_n$. 
  
  Now we show that there are no linear relations among the Mukai vectors of stable spherical objects. Note that the fact that they correspond to objects of phase $\phi$ means the $H^2$ part of the Mukai vector of any stable spherical object is a positive root. The fact that they are all stable of the same phase means $\langle  c_i, c_j \rangle = \Ext^1(\mc S_i, \mc S_j)\ge 0$ since $Hom(\mc S_i, \mc S_j) = 0$ for $i \neq j$ because they are stable. But this implies that the $H^2$ parts of the $c_i$ form a choice of simple roots for a root system, and it must be the root system generated by $R$, so up to relabelling $\{ c_1, \ldots, c_n \} = \{ c_1, \ldots, c_\rho \}$ is the given basis. But when the phase of $v_{D}$ overlaps with the $c_i$, the objects $\mc{S}_i$ cannot be destabilized because the real part $\re Z(\mc{E})$ of the central charge of positive rank objects is too large for any object $\mc{E}$ of Mukai vector $a v_{D} + b q$ for $q\in \mc{R}$ to destabilize any of the $\mc S_i$ with $a \neq 0$. 

  To prove (ii) note that the $(v_D + c_{\mb t})^2 =  -2$ case follows from the fact that destabilizing objects must have rank $\le 1$ due to the central charge, and thus this is an effective spherical object precisely if it is not destabilized by objects in $\mc{R}$. Now assume $(v_D + c_{\mb t})^2 \ge 0$, and assume up to a shift in $D$ and $s$ that $\mb{t} = 0$. A similar argument will work in this case, but we relate it to the techniques of \cite{bayer2014mmp}. In particular, from this source Proposition 6.8 implies that if if there is a spherical object $\mc{S}$ with Mukai vector $\mb s$ such that generically on a wall where the phase of $\mc{S}$ is overlaps with that of $v_{D}$ and $\mc{S}$ is stable such that $\langle \mb s, v_{D} \rangle < 0$ then every object of Mukai vector $v_{D}$ is destabilized for stability conditions on this wall. Conversely, Lemma 6.5 and Proposition 8.4 imply that if for effective spherical class $\mb s \in \mc{H}$ we have $\langle \mb s, v_{D} \rangle \ge 0$ and also there are no isotropic classes $\mb{w}\in \mc{H}$ with $\langle\mb{w} , v \rangle  =1$ then the wall is not totally semistable, i.e. the generic $\sigma_+$-stable object of Mukai vector $v_D$ is $\sigma_0$-stable. To recover our formulation using a higher rank lattice from theirs, if the generic object of Mukai vector $\mb{v}$ is destabilized on this wall, take one factor $q$ in its Jordan-H\"older filtration and the wall where the Mukai vectors of objects with Mukai vector $q$ have overlapping phase with objects of Mukai vector $v_D$ will give the situation described in the cited work. 
  
  In particular, suppose $\langle v_D, c_i \rangle < 0$ for some $i$, then every object of Mukai vector $v_D$ is destabilized by deforming the stability condition to where only these phases overlap. Conversely, assume $\langle v_D, c_i \rangle \ge 0$. We need to show that 
  \begin{enumerate}[(a)]
\item There is no effective spherical class $\mb s \in \mc{H}$ with $\langle \mb s, v_D \rangle < 0$. 
\item There is no isotropic class $w\in \mc{H}$ with $\langle w, v_D \rangle =1$.
  \end{enumerate}
For (a), suppose that $\mb s = a v_D + c_{\mb t}$ is this class for $a\in \Z$ with $a \neq 0$ (we know all effective spherical classes in $\mc{R}$, they are the $c_i$). Then for $\mb s$ to be effective we must have $a > 0$ because of the stability condition limit. But 
\[ -2 = \mb s^2 = a^2 v_D^2 + 2a \langle v, c_{\mb t} \rangle + c_{\mb t} ^2 \]
and 
\[\langle \mb s, v_D \rangle = a v_D^2 + \langle v, c_{\mb t} \rangle\]
so that combining these we get 
\[ \langle \mb s, v_D \rangle = \frac{-2 - c_{\mb t}^2}{a} - \langle v, c_{\mb t} \rangle. \]
If $ \langle \mb s, v_D \rangle < 0$ we must have $\langle v, c_{\mb t} \rangle<0$ but then $-c_{\mb t}^2 \ge 2$ and so $\langle \mb s, v_D \rangle \ge 0$ after all. 
For (b) suppose $w =  a v_D + c_{\mb t}$ is isotropic and $\langle w, v_d \rangle = 0$. Then similarly we have the formulas
\begin{align*}
0 &= w^2 = a^2 v_D^2 + c_{\mb t}^2 + 2 a \langle v_D, c_{\mb t} \rangle\\
\langle w, v_D \rangle &= av^2 + \langle v_D, c_{\mb t} \rangle = - \frac{c_{\mb t}^2}{a}.
\end{align*}
Thus again we can deduce from $av^2 \ge 0$ that $\langle v_D, c_{\mb t} \rangle \ge 0$ and dividing the formula for $w^2$ by $a$, the only solution to $w^2 = 0$ occurs with $a = 1, v^2 = 2$, but then we can't have $\langle w, v_D \rangle = 1$. 

To show (iii) first consider the decomposition of $\mc{H}$ via the walls $c_i^\perp$.  When $v_D + c_{\mb t}$ lies in the chamber of where $\langle v_D + c_{\mb t} , c_i \rangle \ge 0$ then the destabilizing sequences are iterated extensions of stable objects with Mukai vectors of the form $ v_D + c_{\mb t'}$ by the stable spherical objects $\mc{S}_i$ for some $v_D + c_{\mb t'}$ also in this chamber (or the extensions are the other way around). When we are in another chamber, there is a sequence of spherical twists by the $\mc S_i$ under which an object with Mukai vector $v_D + c_{\mb t}$ is sent to one in the positive chamber, where the previous case applies. In this case we combine the previous decomposition with the decomposition given by these spherical twists. 
\end{proof}
This description will allow us to locate the maximal common stratum after the following:
\begin{lem}\label{lem:singular_stratum_corner}
Using the notation of Proposition \ref{prop:corner_jh}, there exists a unique $\mb m\in \Z^\rho$ such that for every $\sigma_+$-stable object $\mc F \in M_{\sigma_+}(v_{D,s} + c_{\mb t})$ for any $\mb t \in \Z^\rho$, there are $\ell_i \ge 0$ such that the direct sum $\mc F \oplus \bigoplus_{i = 1}^\rho \mc{S}_i^{\oplus \ell_i}$ is $S$-equivalent under the stability condition $\sigma_0$ to an object of the form 
\[ \mc{E} \oplus \bigoplus_{i = 1}^\rho \mc{S}_i^{\oplus k_i}\]
for some $k_i \ge 0$ where $\mc{E} \in M_{\sigma_+}(v_{D,s} + c_{\mb m})$ is $\sigma_+$-stable of Mukai vector $v_{D,s} + c_{\mb m}$. 
\end{lem}
\begin{proof}
  Recall the proof of part (iii) of the previous proposition, in particular the decomposition of $\mc{H}$ via the walls $c_i^\perp$ so that the polystable representative of on object of Mukai vector $v_D + c_{\mb t}$ is equivalent to  that of $\mc F \bigoplus \oplus_i \mc S_i$ for some (possibly repeated) $\mc S_i$ where $\mc F$ has Mukai vector $v_D + c_{\mb t'}$  in the positive chamber, i.e.  $\langle v_D + c_{\mb t'}, c_i \rangle \ge 0$ for all $i$.  Now write $\mc{R}_{\R}$ as the $\R$-span the $c_i$, so that there is a unique $v_0\in v_D + \mc{R}_\R$ such that $v_0 \in c_i^\perp$ for all $i$. But if there were two different $\alpha, \beta\in \mc{R}_\R$  such that $x := v_0 + \alpha$ and $y: = v_0 + \beta$ both lie in the lattice $v_D + \mc R$ and satisfy the condition that for all $\mb t$ with all components positive we have $(x + c_{\mb t})\cdot c_j < 0$  for some $j$ (and analogously for $y$) then there must exist $\mb t$ and $\mb t'$ will all components positive such that $y  = x +c_{\mb t} - c_{\mb t'}$ where we can choose $\mb t$ and $\mb t'$ such that they have no non-zero components in common. But then we can find $c_j$ such that
  \begin{align*} 
    \langle y, c_j \rangle =  \langle  x + c_{\mb t} - c_{\mb t'}, c_j \rangle &\ge 0\\
    \langle  x , c_j \rangle &\ge 0\\
    \langle  x + c_{\mb t}, c_j \rangle &< 0
  \end{align*}
  from which it follows that $\langle c_{\mb t}, c_j \rangle < 0$ and $\langle c_{\mb t'}, c_j \rangle < 0$ contradicting that only one of $c_{\mb t}$ and $c_{\mb t'}$ has a $c_j$ component. 

  Let $\mb{m}$ be defined so that $v_D + c_{\mb m} = x = y$. Then all polystable representative of objects $v_D + c_{\mb t}$ with Mukai vector in the positive chamber can, after some extension by some of the $\mc S_i$ be a $\sigma_+$ stable object of Mukai vector $v_D + c_{\mb m}$, while vectors which don't lie in the positive chamber are equivalent to those in the positive chamber after formally subtracting some $\mc S_i$ by Proposition \ref{prop:corner_jh} (iii). 
  \end{proof}

It is now possible to finally realize the goal of this section and define the common base of the maps $\pi_{\sigma_+, \sigma_0}$ for different Mukai vectors. 

\begin{defn}
  Fix $D\in \NS(S)$, integer $s\in \Z$, contractible collection $\mc{C} = \{C_1, \ldots, C_r\}$ of $-2$ curves, sequence $\mb{k}$ of integers, and some subset $R \subset \Delta_+$ corresponding to positive roots with corresponding Mukai vectors $c_i = (0, C_i, k_i)$ and let $\mc{R}$ denote the set $\{c_1, \ldots, c_\rho\}$.  Choose stability conditions $\sigma_0 = \sigma_{D,s,R,\mb k}$ as in Notation \ref{not:boundary_stab} and choose a generic stability condition $\sigma_+$ in the adjacent chamber $\mc{A}_{D,s,\mb k}$. Let $v_D = (1,D,s)$.

  Given $\mb{m}$ from Lemma \ref{lem:singular_stratum_corner} define 
  \[M_{\sigma_0}(v_D + \Z\mc{R})\]
  to be the target of the map $\pi_{\sigma_+, \sigma_0}: M_{\sigma_+}(v_D + c_{\mb m}) \to M_0$ contracting S-equivalent objects. Also for any $c_{\mb t}$ such that $(v_D + c_{\mb t})^2\ge -2$ define a map
  \[\pi_{\sigma_+, \sigma_0}: M_{\sigma_+}(v_D + c_{\mb t}) \to M_{\sigma_0}(v_D + \Z\mc{R}) \]
  sending a $\sigma_+$-stable sheaf $\mc{F}$ to its equivalence class up to S-equivalence and the transitive closure of the relation where objects are equivalent up to taking direct sums with the $\mc S_i$. The previous lemma implies that this map is well defined. 
\end{defn}

\paragraph{General Mukai vector} While for Mukai vectors of the form $v_D = (1, D,s)$ and some contractible collection $\mc{C} = \{C_1, \ldots, C_\rho\}$ it is easy to write down explicitly stability conditions so that the phases of objects of Mukai vector  $v_D$ and those of Mukai vectors in the span $\Z \mc{R}$ of $\mc R = \{ (0, C_i, k_i)\}_{i = 1}^\rho$ overlap and the only stable factors for objects of Mukai vector $v_D$ were other objects with Mukai vector $(1, D', s')$ or elements of $\mc R$, and this is the key property which allows one to define common bases for Steinberg correspondences. 

That being said, once we have this property for some other Mukai vector $v$ and some spherical classes $\mc{S} = \{s_1, \ldots, s_\rho\}$ which span a negative definite lattice, it is again possible to define a common base $M_{\sigma_0}(v + \Z \mc S)$ under reasonable hypotheses. We summarize how the arguments in this section apply in the general case, where the proofs go though without change. 

\begin{prop} 
  \label{prop:general_base}
  Let $v$ be some primitive Mukai vector with $v^2 \ge -2$ and let $\mc{S} = \{s_1, \ldots, s_\rho\}$ be a collection of spherical classes spanning a negative definite lattice arranged according to a simply laced root system $\Delta$. Consider a stability condition $\sigma_0$ for which there are semistable objects of Mukai vector $v$ of phase $\phi$ such that 
  \begin{itemize}
    \item For $s_{\mb t} \in \Z \mc {S}$ such that $(v + s_{\mb t})^2 \ge -2$ we have that $v + s_{\mb t}$ is primitive. In particular, the corresponding moduli spaces for generic $\sigma$ are smooth.
\item There is a $\sigma_0$-stable object $S_i$ of Mukai vector $s_i$ and phase $\phi$ for every $i$. 
\item  For an adjacent generic $\sigma$ and $\sigma$-stable object $\mc E$ of Mukai vector $v$, the Jordan-H\"older filtration of $\mc E$ consists of objects $S_i$ with some multiplicity and a unique object of Mukai vector $v - s_{\mb t}$ for $s_{\mb t} \in \Z \mc S$. 
  \end{itemize}
  Given such $\mc E$ let $\pi(\mc E)$ denote this unique object of Mukai vector $v - s_{\mb t}$. 
Then there exists a unique $\mb m \in \Z^\rho$ such that $\langle v+ s_{\mb m}, s_i \rangle\ge 0$ for all $i$ and a variety 
$M_{\sigma_0}(v + \Z \mc S)$ 
defined  as the image of the contraction for generic adjacent $\sigma$
\[\pi_{\sigma, \sigma_0}: M_{\sigma}(v + s_{\mb m}) \to M_{\sigma_0}(v + \Z \mc S)  \]
such that for any $s_{\mb t} \in \Z \mc S$ and generic adjacent $\sigma$ there is a map 
\[\pi_{\sigma, \sigma_0}: M_{\sigma}(v + s_{\mb t}) \to M_{\sigma_0}(v + \Z \mc S)  \]
which is the composition of the map contracting $S$-equivalent objects with an inclusion such that $\pi_{\sigma, \sigma_0}(\mc E)$ is equivalent to $\pi(\mc E)$ after potentially adding and/or subtracting some $S_i$. 
\end{prop}

\subsection{Local analytic structure}\label{sec:local_analytic_brute_force}

We need to describe locally analytically the structure of the maps $\pi_{\sigma_+,\sigma_0}: M_{\sigma_+}(v_D + c_{\mb t}) \to M_{\sigma_0}(v_D + \Z\mc{R})$. It will turn out that in this case, the Ext quiver description of Section \ref{sec:local_k3_ext_quiver} is correct, but the proof here instead involves the map 
$M_{\sigma_+}(v_D + c_{\mb t}) \to Sym^n(S_{bd})$ to the symmetric product on the blown-down surface and the fact that this map factors through $\pi_{\sigma_+,\sigma_0}$ and a comparison with the ADE surfaces obtained as  neighborhoods around the connected components of contracted curves. 

Fix $D\in \NS(S), s \in \Z$, let $v = (1, D, s)$. First note that when a stability condition $\sigma \in \mc{A}_{D,s, \mb{k}}$ (see Notation \ref{not:adsk}) approaches the boundary of this set where $(D- \beta)\cdot \omega = 0$ for $\omega$ ample the phases of $\mc{L_D}$ and objects of Mukai vector $v$ overlap, and we are on the wall inducing the Hilbert-Chow map. If, in addition we have $C\cdot \omega = 0$ for $C \in \mc{C}$ for a contractible collection $\mc{C}$ then the stability condition corresponds to $Sym^n(S_{\mc{C}})$ of n points on the surface $S_{\mc C}$ where the curves in $\mc{C}$ are contracted. Let $\sigma_{\mc{C}}$ denote a stability condition corresponding to $Sym^n(S_{\mc{C}})$ on the boundary of $\mc{A}_{D,s, \mb{k}}$, and let $\pi_{\mc{C}}: M_{\sigma}(v) \to Sym^n(S_{\mc{C}})$ denote the contraction map. The birational maps $F_{\mb k, \mb k'}: M_{\sigma}(v) \dashrightarrow M_{\sigma'}(v)$ between generic $\sigma \in \mc{A}_{D,s, \mb{k}}$ and $ \sigma' \in \mc{A}_{D,s, \mb{k}'}$ are maps over $Sym^n(S_{\mc{C}})$. 

Write the collection $\mc{C}$ as $\mc{C} = \mc{C}_1 \sqcup \cdots\sqcup \mc{C}_m$ as a union of disjoint ADE collections. Let $Q_i$ be the affine ADE quiver corresponding to the collection $\mc{C}_i$ and $\mb{w_0}$ the usual 1-dimensional framing vector at the affine node for each $i$, which should not cause any confusion. 

\begin{prop}\label{prop:hilb_is_quiv_amenable}
  Let $\sigma_{\mb k} \in \mc{A}_{D,s, \mb{k}}$ be generic stability conditions for $\mb k$ ranging over a set $K$ sequences giving all chambers in the relevant limit $U_{\mc{C}, D}$ of $\Stab^\dagger(S)$. 
  \begin{enumerate}[(i)]
\item There is an analytic open covering of $M_{\sigma_{\mb k}}(v)$ for every $\mb k\in K$ consisting of sets of the form 
\[ U_{\mb k} = \pi_{\mc C}^{-1}(\prod_{i=1}^m Sym^{\lambda_i} (U_i) \times U_0)\] for some partition $\lambda_0 + \lambda_1 + \ldots \lambda_m = n$ of $n$,  where $U_i \subset S_{\mc C}$ is a set containing the singular locus where $\mc{C}_i$ is contracted and also $U_i$ is biholomorphic to the corresponding ADE surface and $U_0\subset Sym^{\lambda_0}(S\backslash \sqcup U_i)$ is a small open set around a configuration of points outside of the $U_i$. The $U_i$ are required to be pairwise disjoint. 
\item Pick one such set of open sets $U_i, U_0$ and the corresponding $U_{\mb k}$. Let $\tilde{U}_0$ denote the inverse image of $U_0$ under the Hilbert-Chow map.
 Let $\mb{v}_{D,\lambda_i}$ denote the dimension vector for the quiver $Q_i$ such that the Hilb chamber for this dimension vector corresponds to rank 1 torsion-free sheaves with $c_1 = c_1(\mc{L}_D|_{U_i})$ and the finite length quotients have length $\lambda_i$. 
For every $\mb k$ there is a chamber $C_{\mb k, i}$ in the stability space for $Q_i$ with dimension vector $\mb{v}_{D,\lambda_i}$ and framing $\mb{w_0}$ such that for $\theta_i \in C_{\mb k, i}$  there is an  isomorphism 
\[U_{\mb k} \simeq \prod_{i = 1}^m \quiv{Q_i, \theta_i}{\mb{v}_{D,\lambda_i}}{\mb{w_0}} \times \tilde U _0 \]
  such that the map $\pi_{\mc{C}}$ coincides with the map 
  \[\pi_{\theta_1, 0}\times\cdots \times \pi_{\theta_m, 0}\times  \pi_{HC},\] the product of the map to the affine quotient \[\pi_{\theta_i, 0} : \quiv{Q_i, \theta_i}{\mb{v}_{D,a, i}}{\mb{w_0}} \to Sym^a(U_i)\] on the first $m$ factors with the Hilbert-Chow map $\pi_{HC}:\tilde{U}_0 \to U_0$ on the last factor. 
  \item The correspondence between $\mb k$ and chambers $C_{\mb k, i}$ is as follows: given the alcove $A_{\mb k}$ of the affine reflection group action on $\R^\rho$ where $\rho = |\mc{C}|$, we can write $A_{\mb k}$ as $\prod_{i = 1}^m A_{k_i}$ where $A_{k_i}$ is an alcove for the $i$th factor under the decomposition of $\R^\rho$ into a product of affine hyperplane arrangements.  the chamber $C_{\mb{k}, i}$ is the one intersecting $A_{k_i}$ on the alcove structure on the level 1 hyperplane $\{\theta\cdot \delta = 1\}$ in the stability space for the quiver $Q_i$. 
  \end{enumerate}
\end{prop}
\begin{proof}
The statement of (i) just follows from the fact that the sets 
\[
  \prod_{i = 1}^mSym^\lambda_i(U_i) \times U_0
  \]
   cover $Sym^n(S_{\mc C})$. Now note that (ii) is true for chambers for the quivers and for $S$ which actually correspond to the space of rank 1 torsion-free sheaves. Then since all of the birational transformations between different $M_{\sigma_{\mb k}}(v)$ 1) are birational maps over $Sym^n(S_{\mc C})$ and 2) do not contract the exceptional locus of the Hilbert-Chow map, we know that all of the flops are locally isomorphisms over the $\tilde{U}_0$ factor. More generally, we know that all of the codimension 1 walls are walls where the phase of an object in our moduli space overlaps with an object with support in a \emph{connected} collection $\mc{C}_i$ and therefore contracted curves for this wall only occur in the corresponding factor over $U_i$. It follows that all of the flops preserve the product structure on $U_{\mb k}$ since every flop can be factors into a sequence of such flops hitting codimension 1 walls, and each of these preserves the product structure. This proves (ii) after Theorem \ref{thm:birat_sym} implies that all smooth symplectic birational models of $\quiv{Q_i, \theta_i}{\mb v_{D, \lambda_i}}{\mb{w_0}}$ over $Sym^a(U_i)$ are given by variation of GIT stability. 
   
   Then (iii) follows from Lemma \ref{lem:cc_shift_c1}, Proposition \ref{prop:affine_quiver_shift} which by matching shifts on both the quiver and K3 surface sides, reduces it to the case where we are dealing with the Hilbert scheme, i.e. for trivial first chern class. There it follows from the fact that the map $N^1(S^{[n]}/Sym^{n}(S_{\mc C}))\to N^1(U_i^{[a]}/Sym^{\lambda_i}(U_i)))$ sends $|2, 1^{n-1} \rangle$ to $|2, 1^{n-2} \rangle$ and $|1_{C}, 1^{n-1} \rangle$ to $|1_{C}, 1^{n-1} \rangle$ for every $C\in \mc{C}_i$, and so the birational maps induced on the $i$th factor of $U_{\mb k}$ by varying Bridgeland stability conditions must be the ones from varying GIT stability on the quiver.

\end{proof}
This implies a geometric modular interpretation of some quiver varieties for affine ADE quivers $Q$ corresponding to finite ADE Dynkin diagrams that show up as dual graphs of collections of -2 curves on K3 surfaces. 
\begin{cor}\label{cor:quiver_geometric_moduli}
Let $Q$ be an affine quiver corresponding to a connected contractible collection $\mc C$ on $S$. Fix an open set $U \subset S_{\mc C}$ containing $\mc{C}$ which is biholomorphic to the corresponding ADE surface.  Fix framing vector $\mb{w_0}$. Then given a generic stability condition $\theta$ and dimension vector $\mb v$ there is a Mukai vector $v$, a generic stability condition $\sigma \in \Stab^\dagger(S)$ in a chamber which has a stability condition on its boundary inducing a contraction $\pi_{\mc C}$ onto $Sym^k(S_{\mc C})$ and an isomorphism 
\[ \quiv{\theta}{\mb{v}}{\mb{w_0}} \simeq M_\sigma(v, U)\]
between the corresponding quiver variety and an open set $M_\sigma(v, U) = \pi_{\mc {C}}^{-1}(Sym^k(U))$ of $M_\sigma(v)$ parametrizing $\sigma$-stable objects on $S$ of Mukai vector $v$ such that under the map $\pi_{\mc C}$ the support of $\pi_{\mc C}(\mc{E})$ for an object $\mc E \in M_\sigma(v,U)$ lies in $U$. This correspondence is such that the birational transformations between different chambers in the stability space for the quiver are induced by those between different chambers in $\Stab^\dagger(S)$.
\end{cor}
\begin{proof}
This is straightforward from the previous proposition when $(c_1(\mc{E}) - \beta )\cdot \omega > 0$ which gives all stability chambers for $\theta$ such that $\theta\cdot \delta > 0$. But since $-\beta\cdot \omega$ can be chosen to have a much larger magnitude than $c_1(\mc{E})\cdot \omega$ the general case follows from the fact (see e.g. \cite[Prop. 2.11]{bayer2014mmp}) that mapping $(\omega, \beta) \mapsto (\omega, -\beta)$ gives isomorphic moduli spaces $M_{\sigma_{\omega, \beta}}(v) \simeq M_{\sigma_{\omega, -\beta}}(v) $ where the isomorphism is given by taking a derived dual of a stable object, and so we obtain moduli spaces of Bridgeland stable objects for stability chambers containing $\theta$ such that $\theta\cdot \delta < 0$, and hence for all quiver stability chambers. Further, the walls $\{\theta\cdot \delta = 0\}$ for the quiver side and $(c_1(\mc{E}) - \beta )\cdot \omega = 0$ both induce, for stability conditions generically on this wall, the Hilbert-Chow contraction, so the birational transformations induced by variation of Bridgeland stability between different chambers are exactly those induced by variation of GIT locally for the quiver varieties. 
\end{proof}

\paragraph{Definition allowing for Lie algebra actions}
The previous paragraphs describes the local structure of certain singularities for  moduli spaces birational to moduli spaces of rank 1 torsion free sheaves in a way which will be shown to be sufficient to produce an action of a Lie algebra. We propose a definition which should hold in quite general scenarios (c.f. Section \ref{sec:local_k3_ext_quiver}) which allows one to construct finite dimensional Lie algebra actions. We restrict for simplicity to connected Dynkin diagrams but the generalization is immediate. 

Let $v$ be some primitive Mukai vector with $v^2 \ge -2$ and let $\mc{S} = \{s_1, \ldots, s_\rho\}$ be a collection of spherical classes spanning a negative definite lattice arranged according to a simply laced root system $\Delta$. Consider a stability condition $\sigma_0$ for which there are semistable objects of Mukai vector $v$ of phase $\phi$. We know that $\sigma_0$ lies on the intersection of codimension 1 walls $W_\alpha$ arranged according to the root system $\Delta$. Let $\{C_w\}_{w\in W}$ be the chambers for the corresponding finite quiver variety and pick an isomorphism between wall and chamber structures near $\sigma_0$ and for the quiver variety preserving the polyhedral structure, and $\sigma_w$ for $w \in W$ generic in the chamber corresponding to $C_w$. 

\begin{defn}\label{def:amenable}
 Then the data $(v, \mc S, \sigma_0)$ is called \emph{amenable to a local quiver Lie algebra action} if 
\begin{itemize}
  \item The conditions of Proposition \ref{prop:general_base} are satisfied so there is a base $M_{\sigma_0}(v + \Z \mc S)$ to which all $M_{\sigma}(v + s_{\mb t})$ map, such that the map is the composition of the map contracting S-equivalent objects under $\sigma_0$ followed by an inclusion.  
  \item For each $\mc{E}\in M_{\sigma_w}(v + s_\mb {t})$ for all $\mb t$ and all $w\in W$ there is a neighborhood around $\mc {E}$ such that the map $\pi_{\sigma_w, \sigma_0}$ is isomorphic to the product $\pi_{\theta, 0}\times \Id$ where $\Id$ is the identity on some factor and $\pi_{\theta, 0}$ is the map to the affine quotient for the Ext quiver of Definition \ref{defn:ext_quiver_object} of the $\sigma_0$-polystable representative of $\mc E$ and $\theta$ is a generic stability condition. 
\end{itemize}
\end{defn}

In particular, Proposition \ref{prop:hilb_is_quiv_amenable} together with the local description of Theorem \ref{thm:local_quiver} imply that for any $v_D = (1, D, s)$, and negative definite collection $\mc{R}$ corresponding to a connected Dynkin diagram
and $\sigma_0 = \sigma_{D,s, \mc R, \mb k}$, the data $(v_D, \mc R, \sigma_0)$ is amenable to a local quiver Lie algebra action. 

We will later see in Theorem \ref{thm:finite_action}
that this condition allows for the construction of a geometric finite dimensional Lie algebra action on the cohomologies of associated moduli spaces, which justifies the name. 

\subsection{Action near corners}
In this section, given a fixed Mukai vector $v_D = (1, D,1 + D^2/2 - n)$ and contractible collection $\mc C$ with corresponding finite dimensional semisimple Lie algebra $\mf g$ of rank $\rho$ and we define a geometric action of the affine Lie algebra $\hat{\mf g}$ on 
\[ \bigoplus_{\substack{\mb t \in \Z^\rho \\ s \in \Z}} H^*(M_{\sigma_{D,s}}((1, D + c_{\mb t}, s)))\]
where $\sigma_{D, s}$ is a stability condition in $\mc{A}_{D,s} = \mc{A}_{D,s, \mb 0}$ from Notation \ref{not:adsk}. We can also produce Lie algebra actions for different stability conditions (e.g. for the Gieseker chamber) by conjugating by the birational transformations induced by varying the stability condition. 

The same technique will produce finite dimensional Lie algebra actions on cohomologies for more general collections of Mukai vectors subject to some conditions on stable factors and local structures of singularities so that we may use local quiver calculations. Thus the result in this setup will also be recorded. 

\paragraph{Glueing Hecke corespondences}

We construct Lagrangian correspondences between a pair of moduli spaces which locally agree with the Hecke correspondences using the local description of moduli space contractions as quiver variety contractions. One of these correspondences will locally coincide with the Hecke correspondence, and thus convolution with this correspondence will essentially induce the action of the Chevalley generators $e_i$ and $f_i$ on cohomology. 

We work in the setup of Proposition \ref{prop:general_base} and Definition \ref{def:amenable}. Let $\mc S = \{s\}$ be a single spherical class and pick Mukai vector $v_0$ and stability condition $\sigma_0$ satisfying the conditions of the proposition and pick $\sigma$ in a chamber adjacent to the wall on which $\sigma_0$ lies. Let $M_{\sigma_0}(v_0+ \Z s)$ denote the base produced by the theorem. Assume that the data $(v_0, \mc{S}, \sigma_0)$ is amenable to a local quiver Lie algebra action. 

\begin{lem}\label{lem:stratum_open_cover}
  Given $v\in v_0 + \Z s$, let $M_0$ denote the smallest stratum of $M_{\sigma_0}(v_0 + \Z s)$ which contains the image of both $M_{\sigma}(v)$ and $M_{\sigma}(v + s)$ under $\pi_{\sigma, \sigma_0}: M_{\sigma}(v)\to M_{\sigma_0}(v_0 + \Z s)$ and $\pi'_{\sigma, \sigma_0}: M_{\sigma}(v + s)\to M_{\sigma_0}(v_0 + \Z s)$. 
There is a finite open cover $\{U_i\}$  of $M_0$ such that 
\[\pi_{\sigma, \sigma_0}^{-1}(U_i) \to U_i\]
and 
\[\pi_{\sigma, \sigma_0}^{'-1}(U_i) \to U_i\]
as maps onto their images are isomorphic to a trivial factors times Springer resolutions of closures of nilpotent orbits. More precisely if $\pi_{k,n}$ denotes the Springer map $\pi_{k,n}: T^* Gr(k, n) \to \mc{N}$ restricted to some neighborhood of the central fiber for some $k,n$, there is an open set $V\subset \C^n$ such that locally $\pi_{\sigma, \sigma_0} \simeq \pi_{k,n} \times \Id_V$  while  $\pi_{\sigma, \sigma_0}^{'}\simeq \pi_{k+1,n} \times \Id_V$ for a different Springer map $\pi_{k+1,n}: T^* Gr(k+1, n) \to \mc{N}'$ times the same trivial factor. Further these maps are compatible under the inclusion of one of $\mc N$ or $\mc N'$ into the other.
\end{lem}
\begin{proof}
Pick some finite set of points $x$ in $M_0$ corresponding to polystable sheaves $\{ \mc F_x\}$ of Mukai vector $v$ or $v+s$ such that by our assumption that the data given is amenable to a local quiver Lie algebra action there is an open covering $U_i$ of $M_0$ around these $x$ for which the maps $\pi_{\sigma, \sigma_0}$ and $\pi'_{\sigma, \sigma_0}$  admit local Ext-quiver descriptions. Then by the assumption that the stable factors of these objects may only be a unique object with Mukai vector in $v_0+ \Z s$ and some number of copies of $s$, the local Ext quiver must be empty or the quiver with one node and no loops. And if $x$ is in the image of both $\pi_{\sigma, \sigma_0}$ and $\pi'_{\sigma, \sigma_0}$ then the unique object with Mukai vector in $v_0+ \Z s$ agrees, which forms the framing node in the Crawley-Boevey quiver $Q_\infty$ so the framing dimensions agree as well as the dimension of the trivial factor, and the dimension vectors for the two Ext quivers must differ by $1$. 
\end{proof}
This description allows us to define the correspondence, whose definition is the same if we are considering the affine Lie algebra action or the finite dimensional one. 

\begin{defn}
\label{defn:heckek3}
Let
\[ \mf{P}_s^{\mathrm{K3}}(v) \subset M_{\sigma}(v) \times_{M_0} M_{\sigma}(v + s) \] 
denote the subvariety such that given the open cover $U_i$ of the previous lemma, we have that 
\[ \mf{P}_s^{\mathrm{K3}}(v) \big |_{\pi^{-1}(U_i)\times_{U_i} \pi^{'-1}(U_i)} = \mf{P}(v)\times \Delta\]
where 
$\mf{P}(v)$ is the Hecke correspondence on the first factor and $\Delta$ is the diagonal on the factors from the previous lemma on which $\pi$ and $\pi'$ are the identity. Note that we suppress the subscripts on these maps. 
\end{defn}

\begin{prop}
  This subvariety $ \mf{P}_s^{\mathrm{K3}}(v)$ is well defined. In addition it is smooth, irreducible, and Lagrangian under the usual product holomorphic symplectic form on $M_{\sigma}(v) \times M_{\sigma}(v + s)$, i.e. with a minus sign on the second factor. 
\end{prop}
\begin{proof}
That the subvariety is smooth irreducible and Lagrangian follows from a local calculation where it is true over each open set $U_i$. That it is well defined follows from the fact Proposition \ref{prop:hecke_from_1_wall} which implies that it is uniquely determined by its fiber over a point on each factor $M_{\sigma}(v)$ and $M_{\sigma}(v + s)$, regardless of the choice of point around which we take a local Ext quiver description. 
\end{proof}

\begin{rmk}
The variety 
\[ \mf{P}_s^{\mathrm{K3}}(v) \subset M_{\sigma}(v) \times_{M_0} M_{\sigma}(v + s) \]  is  essentially the variety of extensions 
\[\{\ses{\mc E}{\mc E'}{S} \mid \mc E \in M_{\sigma(v)}, \mc E' \in M_{\sigma}(v + s)  \} \]
where $S$ is the unique $\sigma_0$-stable object of class $s$ and the maps to the two factors $M_{\sigma}(v)$ and $M_{\sigma}(v + s)$ project onto $\mc{E}$ and $\mc E'$. This means that convolution with this class encodes multiplication in some subset of some version of the Hall algebra. 
\end{rmk}

\paragraph{Finite dimensional Lie algbra actions}

We first construct the finite dimensional Lie algebra action, and justify the name of Definition \ref{def:amenable}. To this end let $v$ be a Mukai vector, $\sigma_0$ a stability condition and $\mc S = \{ s_1, \ldots, s_\rho\}$ be a set of spherical classes such that 
$(v, \mc S, \sigma_0)$ is amenable to a local quiver Lie algebra action. Pick $\sigma_{w}$ generic in chambers adjacent to $\sigma_0$, labelled by $w\in W$ the Weyl group for the corresponding finite dimensional Lie algebra $\mf{g}$. Consider the finite disjoint union and common base
\begin{gather*}
   M_{\sigma_w} := \bigsqcup_{(v + s_{\mb t})^2 \ge -2} M_{\sigma_w}(v + s_{\mb t})\\
   M_0 := M_{\sigma_0}(v + \Z \mc S)
\end{gather*}
and define the analogue of the Steinberg variety as 
\[ Z := M_{\sigma_w} \times_{M_0} M_{\sigma_w}.\]

We record a few facts about the convolution structure. For proofs see \cite[\S 2.7]{chriss_representation_2010}. Let $p_{ij}: M_{\sigma_w} \times  M_{\sigma_w} \times  M_{\sigma_w}\to M_{\sigma_w} \times  M_{\sigma_w}$ denote projection onto the product of the $i$ and $j$ factors. In what follows we will frequently suppress applications of Poincar\'e duality. 

\begin{prop}\label{prop:steinberg_properties}
The cohomology $H^*(Z)$ is an associative algebra under the convolution product 
\[ [\alpha] \circ [\beta] := p_{13 *}(p^*_{12}(\alpha) \cup p^*_{23}(\beta))\]
for $[\alpha], [\beta]\in H^*(Z)$. Further 
\begin{enumerate}[(i)]
  \item The diagonal $[\Delta] \subset H^*(Z)$ is the identity for the product.
  \item Let $H(Z)$ denote the subspace spanned by cycles which are middle dimensional in $M_{\sigma_w}(v + s_{\mb t}) \times M_{\sigma_w}(v + s_{\mb t'})$ for some $\mb t, \mb t'$. Then this subspace is actually a subalgebra. Also the set-theoretic convolution of two Lagrangian cycles is isotropic. 
  \item For any $x\in M_0$, if $M_x$ is the fiber of $M_{\sigma_w}$ over $x$ then $H^*(M_x)$ is a module over $H^*(Z)$ and $H(Z)$ preserves degrees on this module. 
  \item Given $U\subset M_0$ open if we repeat the construction using Borel-Moore homology for $Z_U := \pi^{-1}(U)\times_U \pi^{-1}(U)$ then the restriction map composed with Poincar\'e duality 
  \[H^*(Z) \to H_*^{BM}(Z_U) \]
  is an algebra homomorphism which preserves the module structure on $H^*(M_x)$ for $x\in U$. 
\end{enumerate}
\end{prop}
Define the map $\omega: M\times N \to N\times M$ which flips the two components of any product. We come to the analogue of (a simpler version of) Theorem \ref{thm:uea_action_nakajima}. Recall that convolution is written in the opposite order as composition of operators. 
\begin{thm}\label{thm:finite_action}
Let $e_i, f_i, h_i \in \mf g$ for $i = 1, \ldots, \rho $ denote the Chevalley generators for the finite dimensional Lie algebra associated to the data $(v, \mc{S}, \sigma_0)$.
Then there is an algebra morphism 
\[ U(\mf g) \to H(Z)\]
defined on generators by
\begin{align*}
e_i &\mapsto   \sum_{(v+ s_{\mb t})^2 \ge -2 } \big [ \mf{P}_{s_i}^{\mathrm{K3}}(v+ s_{\mb t})\big]\\
f_i &\mapsto \sum_{(v+ s_{\mb t})^2 \ge -2} (-1)^{r_i( s_{\mb t})}\big [ \omega\big(\mf{P}_{s_i}^{\mathrm{K3}}(v+ s_{\mb t})\big)\big]\\
h_i &\mapsto \sum_{(v+ s_{\mb t})^2 \ge -2} -\langle v + s_{\mb t} , s_i \rangle[\Delta_{M_{\sigma_w}(v+ s_{\mb t})}]
\end{align*}
where 
\begin{align*}
  r_i( s_{\mb t}) &:= \frac 12 (\dim M_{\sigma_w}(v + s_{\mb t} + s_i) - \dim M_{\sigma_w}(v + s_{\mb t})) \\
  &= \langle v + s_{\mb t}, s_i \rangle - 1.
\end{align*}
\end{thm}
\begin{proof}
We check first the case when $\mf{g}$ is ADE, and later deduce the case where $\mf{g}$ is a direct sum of ADE Lie algebras. The relations \eqref{bkm1}- \eqref{bkm5}, essentially using Proposition \ref{prop:steinberg_properties} (iv) which reduces to the local case,  where the local calculations are in \cite[\S 9]{nakajima1998quiver}. The relations \eqref{bkm1} and \eqref{bkm2} are immediate. Then as in \cite[9.iii]{nakajima1998quiver}, the relations \eqref{bkm4} and \eqref{bkm5} follow from the finite-dimensionality of $H^*(M_{\sigma_w})$ once we know \eqref{bkm3}. But this follows from the fact that 
\[ [\Delta] = \sum [\Delta_{M_{\sigma_w}(v + s_{\mb t})}]\] and then
locally using \cite[p. 545-6]{nakajima1998quiver} that the class
\[
[C] := (f_j \circ e_i - e_i \circ f_j)\circ [\Delta_{M_{\sigma_w}(v + s_{\mb t})}]
\]
is represented by the cycle
\[ C = \begin{cases} 
  \emptyset &   i \neq j\\
  c \Delta_{M_{\sigma_w}(v + s_{\mb t})} & i  = j
\end{cases}\]
where the constant can be checked over an open set $U\in M_0$ to be
\[ c =   -1 - \frac{1}{2} (\dim M_{\sigma_w}(v + s_{\mb t} + s_i) - \dim M_{\sigma_w}(v + s_{\mb t})) = -\langle v + s_{\mb t}, s_i \rangle \]
because taking the difference in dimensions eliminates the contribution of the trivial factor from Lemma \ref{lem:stratum_open_cover}. 

Now consider $\mf{g} = \oplus_{i=1}^d \mf{g}_i$ where each $\mf{g}_i$ is simple. Pick roots $\mc{S}_i = \{s_{i,j}\} \subset \mc{S}$ corresponding to each $\mf{g}_i$ and stability conditions $\sigma_{0,i}$ where the phase of $v$ generically overlaps with that of an object of each Mukai vector $s_{i,j}$. 
Then the data $(v, \mc{S}_i, \sigma_{0,i})$ are each amenable to local quiver Lie algebra actions, and the previous case proves maps $U(\mf{g}_i) \to H(Z)$, we are left to prove that they commute. 
Without loss of generality take $e_1 \in \mf{g}_1$ and $e_2 \in \mf{g}_2$ corresponding to classes $s_1, s_2 \in \mc{S}$.  
But since a local quiver variety for a disjoint quiver $Q_1\sqcup Q_2$ has the structure of a product $\mf M_{Q_1, \theta_1}(\mb v_1, \mb{w}_1)\times \mf M_{Q_2, \theta_2}(\mb v_2, \mb{w}_2)$,
 we have that $M_0$ is covered by open sets of the form $U$ such that their inverse image in each $M_{\sigma}(v + s_{\mb t})$ is of the form 
\[ U_1 \times U_2 \times U_3 \]
where $\mf{P}_{s_1}^{\mathrm{K3}}(v+ s_{\mb t})$ is of the form $L \times \Delta_{U_2\times U_3}$ and up to reordering $\mf{P}_{s_2}^{\mathrm {K3}}(v+ s_{\mb t})$ is of the form $L \times \Delta_{U_1\times U_3}$, from which it follows that $e_1$ and $e_2$ commute. 

\end{proof}

\paragraph{Affine Lie algebra actions}
We now restrict to Mukai vectors of the form $(1, D, s)$ and essentially glue the previous Lie algebra actions from corners of the chamber $\mc{A}_{D,s}$ together to form affine Lie algebra actions. We can accomplish the same effect by working with the map to $Sym^n(S_{\mc C})$ for a contractible collection $\mc{C}$. 

To this end fix $v_D = (1,D,s)$ and a (not necessarily connected) contractible collection $\mc{C} = \sqcup_{i = 1}^d \mc{C}_d$ with $\mc{C} = \{C_{1}, \ldots, C_{ \rho}\}$, consider the Mukai vectors
\begin{align*}
  c_{i} &:= v(\mc O_{C_{i}}(-1)) = (0, C_{i}, 0)\\
  c_{0_j} &:= (0, -C_{\beta_j}, 1) 
\end{align*} 
with $C_{\beta_j}$ the class corresponding to the highest root in $\Z \mc{C}_j$. Let $\mc{R}$ denote the set of $-2$ classes in the span of all $c_{i}$ and $c_{0_j}$.  If $\mb t \in \Z^\rho$ let 
\[c_{\mb t} : = (0, t_i C_i, 0)\]
  Let 
\[M_0 := \bigcup_{n\ge 0} Sym^n(S_{\mc C}) \]
where we identify cycles on $S_{\mc C}$ if they differ by cycles supported at singular points. Thus for a Mukai vector
\[(1, D', s') \in v_D + \Z \mc{R} = v_D + \Z \mc C + \Z (0,0,1) \]
and any stability condition $\sigma_{D',s'} \in \mc{A}_{D',s'}$ from Notation \ref{not:adsk} there is a map 
\[ M_{\sigma_{D,s}}(1, D', s') \to M_0\]
which factors though the map 
\[ M_{\sigma_{D,s}}(1, D', s') \to M_{\sigma_{R'}}((1, D', s') + \Z R')\]
where $R'\subset \mc{R}$ is a subset whose span is negative definite and $\sigma_{R'} = \sigma_{D', s', R', \mb 0}$ in the language of Notation \ref{not:boundary_stab}. 
Given $v_1 = (1, D_1, s_1)$ and $v_2 = (1, D_2, s_2)$ both in $v_D + \Z \mc R$ let 
\[ Z(v_1, v_2) := M_{\sigma_{D_1, s_1}}(v_1) \times_{M_0} M_{\sigma_{D_2, s_2}}(v_2)\]
so that there is a convolution algebra structure on (what we will abusively refer to as )
\[H(Z) := \bigoplus_{v_1, v_2} H(Z(v_1, v_2)). \]

Let 
\begin{align*}
  \hat{\mf g}&:= \mf g[t^{\pm 1}] \oplus \Q c \oplus \Q d
\end{align*}
 denote the affine Lie algebra whose finite Dynkin diagram agrees with the dual graph of the (not necessarily connected) contractible collection $\mc{C}$. This algebra is generated by the affine Lie algebras $\tilde{\mf{g}}_{\mc C_j}$ for connected components $\mc C_j$ by requiring that all of the central elements agree and by adjoining $d = t\frac{d}{dt}$. Thus $\hat{\mf g}$ is generated by $c, d$ and $e_i, f_i, h_i$ for $i = 1, \ldots, \rho$ and an additional $e_{0_j}, f_{0_j}, h_{0_j}$ for each connected component $\mc C_j$. 

Let $P_S\subset H^*(S, \Z)$ be the set of Mukai vectors of the form $v = (1,D,1 + s)$ and consider the map $\lambda: P_S \to P_{\hat{\mf g}}$ defined by
\[\lambda : v \mapsto \Lambda_0  + s \delta - \langle v, c_\alpha \rangle c_i \]
where we identify $c_i$ for $i = 1, \ldots, \rho$ with roots for $\hat {\mf g}$ and $\Lambda_0$ is the fundamental weight dual to $c\in \hat{\mf g}.$

Finally recall from \cite{lusztig_quiver_1998} (see \cite[\S 2]{nakajima1998quiver}) the modified universal enveloping algebra, which allows us to deal easily with infinite unions of Steinberg correspondences. This is an algebra $\tilde{U}(\mf g)$ for a Kac-Moody (or BKM) algebra $\mf g$ generated by  $a_{\lambda}\in P_{\mf g}$ and also $e_i a_{\lambda}$ and $a_{\lambda} f_i $ for $e_i\in U^+(\mf g), f_i\in U^-(\mf g)$ such that under mild conditions which are satisfied in this paper, a representation $M$ of $\tilde{U}(\mf g)$ is the same as a representation of $U(\mf g)$ with a weight decomposition 
\[M = \bigoplus_{\lambda \in P_{\mf g}} M_\lambda\]
 where $a_\lambda$ acts as the projection 
$M \to M_{\lambda}$. 

\begin{prop}\label{prop:affine_corners} 
  For base Mukai vector $v_D = (1,D,s)$ there exists a unique algebra morphism 
\[ \Phi: \tilde U(\hat{\mf g}) \to H(Z)\]
such that 
\begin{align*}
a_{\lambda} &\mapsto \begin{cases} 
  [\Delta_{M_{\sigma_{D',s'}}(1, D', s')}] & \lambda = \lambda(1, D', s')\\
  0 & \text{else}
\end{cases}\\
e_{i} a_{\lambda} &\mapsto E_i \Phi(a_\lambda)\\
a_{\lambda} f_i &\mapsto \Phi(a_\lambda) F_i
\end{align*}
where $E_i, F_i \in \prod_{v_1, v_2} H(Z(v_1, v_2))$ are operators on $H(Z)$ for $i = 1, \ldots, \rho$ or $i = 0_j$ for some $j$ defined by convolution on the right with the formal sums 
\begin{align*}
E_i  &:= \sum_{\substack{v\in v_D + \Z \mc R\\ v^2 \ge -2} }[ \mf{P}_{c_{i}}^{\mathrm{K3}}(v)\big] \\
F_i  &:= \sum_{\substack{v\in v_D + \Z \mc R\\ v^2 \ge -2} } (-1)^{r_i( v)}\big [ \omega\big(\mf{P}_{c_{i}}^{\mathrm{K3}}(v)\big)\big]
\end{align*}
where $r_i(v) = \langle v, c_i \rangle -1 .$
\end{prop}
\begin{proof}
Because of the local quiver description of Proposition \ref{prop:hilb_is_quiv_amenable}, the exact same argument as Theorem \ref{thm:finite_action} we get, for every connected component $\mc C_j$, a morphism $\hat{\mf g}_{\mc C_j} \to H(Z)$ of the stated form with the modification that instead of using a decomposition $[\Delta] = \sum_M [\Delta_M]$ we use the modified universal enveloping algebra. 

Since $c$ acts by $1$ in all of these representations and the action of $d$ coincide, it remains to see that the factors $\tilde{g}_{\mc C_j}$ and $\tilde{g}_{\mc C_j'}$
commute with each other for $j \neq j'$. But this follows from the argument in Lemma 
\ref{lem:stratum_open_cover} which implies that if $c_1$ and $c_2$ are classes supported in different components of $\mc{C}$ then the relevant stratum of $M_0$ is covered by open sets $U$ such that the inverse image of $U$ in $M_{\sigma_{D',s'}}(v)$ for any $v$ is a product
\[ U_1 \times U_2 \times U_3 \]
such that the Hecke correspondences $\mf{P}_{c_1}^{\mathrm{K3}}(v)$ are the identity in the factor $U_2\times U_3$ while the correspondences $\mf{P}_{c_2}^{\mathrm{K3}}(v')$
act as the identity in $U_1 \times U_3$. Thus they commute with each other. 
\end{proof}
\begin{cor}\label{cor:affine_on_VOA}
  By taking a set of representatives $\{ 1, D_i, s_i\}$ for the action of $\Z \mc{R}$ by addition on $P_S  = \{ (1, D, s) \in H^*(S, \Z)\}$, the previous proposition gives a representation of 
  $\tilde{U}(\hat{\mf g})$, and hence of $U(\hat{\mf g})$ on 
  \[\bigoplus_{(1,D,s) \in P_S} H^*(M_{\sigma_{D,s}}(1, D, s)).\]
\end{cor}

\subsection{Compatibility of corner actions}

The next section will prove the main theorem \ref{thm:main} of the paper.
The argument essentially proceeds by combining the previously constructed affine Lie algebra actions and thus rests on the following compatibility lemma. For two stability conditions $\sigma, \sigma'$  and primitive Mukai vector $v$ let
\[F_{\sigma, \sigma'}(v) \subset M_{\sigma}(v)  \times M_{\sigma'}(v) \]
denote the cycle which induces the isomorphism $H^*(M_{\sigma}(v)) \simeq H^*(M_{\sigma'}(v))$ via convolution. 

Given a Lie algebra action for a stability condition $\sigma$, we can induce Lie algebra actions for moduli spaces for another generic stability condition $\sigma'$ by conjugating by a flop relating $M_{\sigma}(v)$ and $M_{\sigma'}(v)$. The lemma says that if there are multiple choices of $\sigma$ where the algebra action is defined, this procedure induces the same action for stability condition $\sigma'$ regardless of $\sigma$.   

\begin{lem}\label{lem:flop_chain}
Let $v_1 = (1, D, s)$ and $v_2 = (1, D+ C, s)$ denote two Mukai vectors which differ by the class of $\mc{O}_C(-1)$ for an irreducible $-2$ curve $C$.
 Let $\sigma_{Hilb}$ be a stability condition in the Gieseker chamber for both $v_1$ and $v_2$. Let $\mc{C}$ and $\mc C'$ denote contractible collections both containing $C$ and stability conditions $\sigma_{\mc C}$, $\sigma_{\mc C'}$ which lie in the chambers $\mc {A}_{D, s}$ from Notation \ref{not:adsk} for the contractible collections $\mc C$ and $\mc{C}'$ respectively.
 \begin{enumerate}[(i)]
  \item Let $\mf{P}_{\mc C}^{\mathrm{K3}}(v)$ and $\mf{P}_{\mc C'}^{\mathrm{K3}}(v)$ denote the Hecke correspondence for class $c = (0, C, 0)$  between moduli spaces for stability conditions $\sigma_{\mc C}$ and $\sigma_{\mc C'}$ respectively.
   Then the cycles 
\[[F_{\sigma_{Hilb}, \sigma_{\mc C}} (v_1)]\circ [\mf{P}_{\mc C}^{\mathrm{K3}}(v_1)] \circ [F_{\sigma_{\mc C} , \sigma_{Hilb}}(v_2)] \]
and
\[[F_{\sigma_{Hilb}, \sigma_{\mc C'}} (v_1)]\circ [\mf{P}_{\mc C'}^{\mathrm{K3}}(v_1)] \circ [F_{\sigma_{\mc C'} , \sigma_{Hilb}}(v_2)] \]
agree and hence induce the same map 
\[H^*(M_{\sigma_{Hilb}}(v_1)) \to H^*(M_{\sigma_{Hilb}}(v_2)).\]
\item More generally, consider contractible collections $\mc C \subset \mc C'$, the class $c = (0, -C_{\max}, 1)$ with $C_{\max}$ the curve class corresponding to the highest root in $\mc{C}$,  with $v_1 = (1, D, s)$ and $v_2 = v_1 + c$. Then there is a corner of $\mc{A}_{D,s}$  for the collection $\mc{C}'$  and $\sigma_0$ on this corner such that for some $\sigma_w$ in a chamber (not necessarily $\mc{A}_{D,s}$)adjacent to $\sigma_0$ such that the wall where the phases of $v_1$ overlaps with that of $c$ is a wall of this chamber. We have an equality of cycles 
\begin{align*} &[F_{\sigma_{Hilb}, \sigma_{\mc C}} (v_1)]\circ [\mf{P}_{\mc C}^{\mathrm{K3}}(v_1)] \circ [F_{\sigma_{\mc C} , \sigma_{Hilb}}(v_2)] \\
&= [F_{\sigma_{Hilb}, \sigma_{\mc C}} (v_1)]\circ [F_{ \sigma_{\mc C}, \sigma_w} (v_1)] \circ [\mf{P}_{\mc C'}^{\mathrm{K3}}(v_1)] \circ [F_{\sigma_{w}, \sigma_{\mc C}} (v_2)]\circ [F_{\sigma_{\mc C} , \sigma_{Hilb}}(v_2)].
\end{align*}
\end{enumerate}
\end{lem}
\begin{proof}
First we prove (i). By Proposition \ref{prop:stab_corner_arranged} in the positive cone for $S^{[n]}$ and its analogue Proposition \ref{prop:wall_chamber_vd} which defines a the wall and chamber structure on  subset $U_{\mc C, D}\subset \Stab^\dagger(S)$  for $v_1$ and $v_2$, there is a continuous path $\sigma_t\in \Stab^{\dagger}$ for $t \in [0,1]$ such that $\sigma_t \in U_{\mc C, D}\cup U_{\mc C', D}$  (c.f. also Notation \ref{not:mca} and subsequent paragraph for the definition of $U_{\mc C, D}$).  Further, this path can be chosen so that it only passes through codimension 1 walls, and lies arbitrarily close to the wall $W_C$ where the phase of $\mc{O}_C(-1)$ overlaps with that of an object of Mukai vector
$v_1$ and $v_2$. In other words, we pass through the chambers of Notation \ref{not:adsk} through codimension 1 walls
\begin{align}\label{eq:sigmatpath}
  ~~~\mc{A}_{D,s, \mb 0} \to \mc{A}_{D,s, \mb j_1} \to \cdots \to &\mc{A}_{D,s, \mb j_n} \to  \\
  \to \mc{A}_{D,s, \mb k_m} &\to \cdots \to \mc{A}_{D,s,\mb k_1 }\to  \mc{A}_{D,s,\mb 0}~~~
\end{align}
where the first line refers to chambers $\mc A$ for contractible collection $\mc{C}$ and the second for $\mc C'$, and each vector $\mb j$ or $\mb k$ defining an alcove has component $0$ corresponding to the class $C$.  For convenience denote $\mb j_0 = \mb 0$ and $\mb k_0 = \mb 0$. Let $\sigma_{\mb j_i}$ denote a class in $\mc{A}_{D,s, \mb j_i}$ and analogously define $\sigma_{\mb k_i}$. 

For $\sigma_0$ generic subject to the condition that it lies on the wall $W_C$ and also on the wall $W_{\mb j_i}$ between $\mc{A}_{D,s, \mb j_i}$ and $\mc{A}_{D,s, \mb j_{i+1}}$ or analogously on both $W_C$ and $W_{\mb k_i}$, then  the collection $(v_1, \mc{S}, \sigma_0)$ is amenable to a local quiver Lie algebra action for $\mf g$ of type $A_2$ or of $A_1\times A_1$ with $\mc{S}$ consisting of $c$ and some other spherical class.  
In either case we have (for $\mb j$, and analogously for $\mb k$ ) an equality of cycles 
\[[F_{\sigma_{\mb j_i}, \sigma_{\mb j_{i+1}}} (v_1)]\circ [\mf{P}_{c,\sigma_{\mb j_{i+1}} }^{\mathrm{K3}}(v_1)] \circ [F_{\sigma_{\mb j_{i+1}} , \sigma_{\mb j_i}}(v_2)] = [\mf{P}_{c, \sigma_{\mb j_{i}}}^{\mathrm{K3}}(v_1)]  \]
 where $\mf{P}_{c, \sigma}^{\mathrm {K3}}(v)$ is the cycle of Definition \ref{defn:heckek3} for given generic condition $\sigma$ and class $s = c$. This follows  from the fact that locally, the geometric action of $\mf{g}$ is intertwined by the flops between different chambers, and so the difference between these cycles is represented by a cycle which is locally the zero cycle. 

By chaining the string of equalities of these cycles the result follows as the stability condition passes through  the chambers in \eqref{eq:sigmatpath} after conjugation by a flop  into the Gieseker chamber, i.e. 
\begin{align*}
  &[F_{\sigma_{Hilb}, \sigma_{\mc C}} (v_1)]\circ [\mf{P}_{\mc C}^{\mathrm{K3}}(v_1)] \circ [F_{\sigma_{\mc C} , \sigma_{Hilb}}(v_2)]  \\
 &= [F_{\sigma_{Hilb}, \sigma_{\mc C}} (v_1)]\circ [F_{\sigma_{\mc C}, \sigma_{\mc C'}} (v_1)] \circ [\mf{P}_{\mc C'}^{\mathrm{K3}}(v_1)]\circ [F_{\sigma_{\mc C'}, \sigma_{\mc C}} (v_2)]\circ [F_{\sigma_{\mc C} , \sigma_{Hilb}}(v_2)] \\
&= [F_{\sigma_{Hilb}, \sigma_{\mc C'}} (v_1)]\circ [\mf{P}_{\mc C'}^{\mathrm{K3}}(v_1)] \circ [F_{\sigma_{\mc C'} , \sigma_{Hilb}}(v_2)]. 
\end{align*}

The proof of (ii) is identical except we first conjugate the Hecke correspondence by a flop to write it as a correspondence 
\[ [F_{ \sigma_{\mc C}, \sigma_w} (v_1)] \circ [\mf{P}_{\mc C'}^{\mathrm{K3}}(v_1)] \circ [F_{\sigma_{w}, \sigma_{\mc C}} (v_2)] \]
between moduli spaces for $\sigma\in \mc{A}_{D, s}$, and then instead of choosing a path of stability conditions close to the wall where the phase of $v_1$ overlaps with $(0,C,0)$ we choose a path arbitrarily close to the wall where the phase of $v_1$ overlaps with $(0, -C_{\max}, 1)$. 
\end{proof}

\subsection{Main theorem}

Let $\mf{g}$ denote the Kac-Moody algebra with Cartan matrix the negative of the Gram matrix for the pairing on $\NS(S)$, where we reiterate that K3 surface $S$ such that $\bar{\NE}(S)$ is the cone spanned irreducible $-2$ curves, and any pair of these either don't intersect or intersect transversely at a single point.

We arrive at the main theorem which gives a representation of the algebra 
\[\hat{\mf g}(\NS(S)) = \mf{g}[t^{\pm 1}] \oplus \Q c \oplus \Q d\]
of Section \ref{sec:VOA} via geometric correspondences defined by variation of Bridgeland stability conditions. 

Now recall the space from \eqref{eq:rk_1_torsion_free_V} which is the direct sum of the cohomologies of all rank 1 torsion-free sheaves denoted
\begin{equation*}
V = \bigoplus_{\substack{D\in \NS(S)\\s\in \Z}} H^*(M(1, D, s)) = \bigoplus_{\substack{D\in \NS(S)\\s\in \Z}} H^*(M_{\sigma_{Hilb}(D,s)}(1, D, s))
\end{equation*}
where $\sigma_{Hilb}(D,s)$ is a stability condition in the Gieseker chamber for Mukai vector $(1, D, S)$. We combine the affine Lie algebra actions which Corollary \ref{cor:affine_on_VOA} imply act on $V$. Denote the triangular decomposition of $\mf{g}$ as $\mf{g} = \mf{g}_+ \oplus \mf h \oplus \mf{g}_-$ chosen so that positive real simple roots correspond to effective $-2$ classes. Recall the modified universal enveloping algebra $\tilde{U}(\hat {\mf{g}}(\NS(S)))$ from the paragraph before \ref{prop:affine_corners}, based on the decomposition
\begin{align*}
U_+(\hat{\mf g}(\NS(S))) &:=  U( t\mf {g} [t] \oplus \mf{g}_+ )\\
U_-(\hat{\mf g}(\NS(S))) &:=  U( t^{-1} \mf {g} [t^{-1}] \oplus \mf{g}_- ).
\end{align*}
This $\tilde{U}(\hat {\mf{g}}(\NS(S)))$ is generated by elements
\[ a_{\lambda}, e a_{\lambda}, a_{\lambda} f ~~ \text{ for } e \in U_+, f\in U_-, \lambda \in P\]
where $P = P_{\mf g} \oplus \Z \Lambda_0 \oplus \Z \delta$ is the weight lattice of $\hat{\mf g}(\NS(S))$ with $\Lambda_0(c) = \delta(d) = 1$ and $\Lambda_0(d) = \delta(c) = 0$ as usual. 

For the following, write $[\alpha]$ for the operator of convolution with $[\alpha]$ and write $[\alpha][\beta]$ for the operator $[\beta]\circ[\alpha](t)$ when $[\alpha]$ and $[\beta]$ are cohomology classes for which the convolution is well defined. Also write $[F_{\sigma, \sigma'}(v)]$ for the class of the cycle inducing the birational transformation between $M_{\sigma}(v)$ and $F_{\sigma'}(v)$. Let $\sigma_{Hilb}(v)$ denote a stability condition in the Gieseker chamber for Mukai vector $v = (1, D, s)$. 
\begin{thm}\label{thm:main}
There is an action of $\tilde{U}(\hat{\mf{g}}(\NS(S)))$ on $V$ generated by elements which are the conjugation of Steinberg correspondences by birational transformations induced by variation of Bridgeland stability conditions. Let $\lambda(v) = \lambda(1, D,s) \in P$ denote the weight corresponding to a specific moduli space so that $a_{\lambda(1,D,s)}$ acts by projection onto $H^*(M_{\sigma_{Hilb}(v)}(v))$ for $ v = (1,D,s)$. This action can be chosen so that
\begin{enumerate}[(i)]
\item For any negative definite (i.e. contractible) collection of irreducible $-2$ curves corresponding to an affine Lie algebra action of $\hat{\mf g}$ on $V'$ from Corollary \ref{cor:affine_on_VOA} where 
\[ V' := \bigoplus_{(1,D,s) \in P_S} H^*(M_{\sigma_{D,s}}(1, D, s)) \]
if $\hat{\mf g}\subset \hat{\mf g}(\NS(S))$ is generated by $c, d$ and $e_i, f_i, h_i$ for $i = 1, \ldots, \rho$ and $i= 0_1, \ldots, 0_m$ then the action of $e_ia_{\lambda}$ and $a_{\lambda}f_i$ for $\lambda = \lambda(v)$ are given by 
\begin{equation}\label{eq:mainE}
   [F_{ \sigma_{D,s},\sigma_{Hilb}(v + c_i)}(v + c_i)] E_i [F_{\sigma_{Hilb}(v), \sigma_{D,s}}(v)] a_{\lambda}
\end{equation}
and 
\begin{equation}\label{eq:mainF}
a_{\lambda}[F_{ \sigma_{D,s},\sigma_{Hilb}(v)}(v)] F_i [F_{\sigma_{Hilb}(v-c_i), \sigma_{D,s}}(v - c_i)] 
\end{equation}
respectively, where $E_i$ and $F_i$ are defined as in Proposition \ref{prop:affine_corners} and $c_i$ is the corresponding $-2$ class.
\item Write $V = V_{alg} \otimes V_T$ where 
\begin{align*}
  V_{alg} &= \bigoplus_{\substack{D\in \NS(S)\\s\in \Z}} H^*_{alg}(M_{\sigma_{Hilb}(D,s)}(1, D, s))\\
  V_{T} &= \bigoplus_{\substack{D\in \NS(S)\\s\in \Z}} H^*_{alg}(M_{\sigma_{Hilb}(D,s)}(1, D, s))^\perp.
\end{align*}
The action of $\hat{\mf g}(\NS(S))$ on $V$ coincides with the action given by Fourier coefficients of vertex operators from Proposition \ref{prop:main_from_rep_theory} for some choice of cocycle $\epsilon$ in \eqref{eq:VOA} defining the VOA structure on $V_{alg}$ acting on the first tensor factor of $V = V_{alg} \otimes V_T$. 
\end{enumerate}
\end{thm}
\begin{proof}
Let $\mc{C}_A$ be a connected contractible $A_n$ collection of curves on $S$. We get two actions of associated Heisenberg algebra $\Heis_{A_n}$ via Nakajima operators and $\Heis_{A_n}'$ from the local construction in Corollary \ref{cor:affine_on_VOA}, 
together with the inclusion $\Heis_{A_n}' \hookrightarrow \hat{\mf{g}}_{\mc C_A}$. 
For both of these, there is an open covering of each $M_{\sigma_{Hilb}(v)}(v)$ by sets of the form $U_{1,v} \times U_{2,v}$ such that for $h(n) \in \Heis_{A_n}$ we have that $h(n)$ acts by a correspondence represented by
\[ L\times\Delta_{U_2} \subset U_{1, v} \times U_2 \times U_{1, v+ (0,0,n)} \times U_2 \]
with $U_2 = U_{2, v}= U_{2, v+ (0,0,n)}$, and analogously $h(n)\in \Heis_{A_n}'$ is represented by a cycle $L'\times \Delta_{U_2}$. Here $U_{1, v}\times U_{2,v}$ is a set $U_{\mb k}$ of the form defined in Proposition \ref{prop:hilb_is_quiv_amenable}. Then Theorem \ref{prop:flop_intertwines} implies that the cycles $L$ and $L'$ agree. Further this proposition implies that for any $x\in \hat{\mf g}_{\mc C_A}$ agrees with a cycled generated by convolution by Nakajima operators and hence $\hat{\mf g}_{\mc C_A}$ acts by $x \otimes 1$ on $V = V_{alg}\otimes V_T$, so we can restrict to studying $V_{alg}$. 

Now let $\mc{C}$ be any contractible collection. Then Lemma \ref{lem:flop_chain} implies that for $\mc{C}_{A}\subset \mc{C}$ an $A_n$ collection, the Lie algebra action $\hat{\mf{g}}_{\mc C_A}$ on $V$ constructed by Corollary \ref{cor:affine_on_VOA} agrees with the action induced by the inclusion $\hat{\mf{g}}_{\mc C_A} \hookrightarrow \hat{\mf{g}}_{\mc C}$ and the action of the latter on $V$ given by the same corollary. 
Part (i) of  Lemma \ref{lem:flop_chain} implies that the action of $e_i$ and$f_i$ on $V$agree for $i \neq 0$ and part (ii) implies that the action of $e_0$ and $f_0$ agree.
 Thus the Heisenberg algebra actions $Heis_{\mc C}$ on $V$ induced by Nakajima operators and $Heis_{\mc C}'$ induced as a subset of $\hat{\mf g}_{\mc C}$ agree by the fact that they agree for all $A_n$ subalgebras, and by the same argument as the previous paragraph act by $x\otimes 1$ on $V = V_{alg}\otimes V_T$.

 Combining this construction for all possible contractible collections, we get that for the action of the Heisenberg algebra $Heis_{\NS(S)}$ on $V$ induced by Nakajima operators and any contractible collection $\mc{C}$, the action of the Heisenberg algebra $Heis_{\mc{C}} \subset \hat{\mf g}_{\mc C}$ on $V$ induced by Corollary \ref{cor:affine_on_VOA} agrees with that induced by  $Heis_{\mc{C}}\hookrightarrow Heis_{\NS(S)}$, and both of these come from an action on $V_{alg}$. Further if $\alpha\in \NS(S)$ is a class corresponding to $e_i$ or $f_i$ for $i = 1, \ldots \rho$ or $i = 0_{j}$ a generator of the Lie algebra $\hat{\mf g}_{\mc C}$ then the vertex operator 
 \[ Y(z) := \sum_{n \in \Z} y_n(\alpha)z^{-n} \]
 where $y_n(\alpha)$ is an operator which on weight space $\lambda$ acts by an appropriate operator generated by those of the form \eqref{eq:mainE} or \eqref{eq:mainF} has the correct commutation relations with elements of the Heisenberg algebra and therefore agrees with that defined in Proposition \ref{prop:main_from_rep_theory} by Proposition \ref{prop:frenkel_kac_follows_from_VOA} for some choice of cocycle $\epsilon$.  Then because these  Fourier coefficients generate the entire algebra $\hat{\mf{g}}(\NS(S))$ the result follows. 
\end{proof}

\paragraph{Further questions}
This result is hopefully a special case of a more general result which takes a collection of spherical classes $\mc{S}$ and Mukai vector $v$ and produces an action of a Lie algebra $\mf{g}_{\mc S}$ with real roots corresponding to $s\in \mc{S}$ on the space 
\[ \bigoplus_{s_{\mb t} \in \Z \mc{S}} H^*(M_{\sigma}(v + s_{\mb t}))\]
where $H^*(-)$ is some cohomology theory and this action is generated by multiplication in a version of the Hall algebra for this cohomology theory by objects of Mukai vector in $\mc{S}$. One set of examples is induced by Theorem \ref{thm:main} under any derived autoequivalence of $D^b(Coh(S))$. In particular if we spherical twist by a higher rank spherical vector bundle we obtain a somewhat peculiar action of $\hat{\mf g}(\NS(S))$ on some moduli spaces corresponding to Mukai vectors with varying rank. It is an interesting question as to how general this construction may be, and among these lie algebra actions which are induced by Fourier coefficients of vertex operators as is this one. 

Also, we expect that just as in the quiver variety case \cite{nakajima2001quiver_finite}, this construction is the classical limit $q \to 1$ of a construction which gives an evaluation representation of a $q$-deformation of $U(\on{Map}(\mathbb{G} \to \hat{\mf{g}}(\NS(S)))))$ on 
the cohomologies of these moduli spaces for some cohomology theory dependent on $\mathbb{G}$ which is the additive group or the multiplicative group,  or possibly an elliptic curve if maps are defined in the correct way. The role of $q$ here, and therefore the relevant cohomology theory is more mysterious. The author intends to return to this question and its potential applications to enumerative geometry in future work. 

\bibliographystyle{abbrv}
\bibliography{lorentzian-KM-K3}

\end{document}